\documentclass[11pt]{article}
\usepackage{amssymb,amsfonts,amsmath,latexsym,verbatim,amscd,amsthm}
\usepackage[english]{babel}
\usepackage[colorlinks,pagebackref,hypertexnames=false]{hyperref}
\usepackage[alphabetic,backrefs,initials]{amsrefs}
\usepackage{graphicx}
\renewcommand{\MR}[1]{} 

\usepackage{todonotes}

\newcommand{\sA}{\mathcal{A}\,}

\newcommand{\sC}{\mathcal{C}}
\newcommand{\sD}{\mathcal{D}}
\newcommand{\sE}{\mathcal{E}}

\newcommand{\sH}{\mathcal{H}\,}
\newcommand{\sI}{\mathcal{I}}

\newcommand{\sK}{\mathcal{K}}
\newcommand{\sL}{\mathcal{L}}

\newcommand{\sO}{\mathcal{O}}

\newcommand{\sQ}{\mathcal{Q}}
\newcommand{\sR}{\mathcal{R}}

\newcommand{\sV}{\mathcal{V}}

\newcommand{\IC}{\mathbb{C}}

\newcommand{\II}{\mathbb{I}}

\newcommand{\IS}{\mathbb{S}}

\newcommand{\IZ}{\mathbb{Z}}

\renewcommand\dim{\operatorname{dim}}

\newcommand\tensor{{\otimes}}
 
\newcommand\union{\bigcup} 

\newcommand{\<}{\langle}
\renewcommand{\>}{\rangle}
\def\wt{\widetilde}

\newcommand{\wtilde}{\widetilde}

\newcommand{\inj}{{\rm inj}}
\def\d#1{(#1)}

\newcommand{\SO}{{\rm SO}}

\numberwithin{equation}{section}

\theoremstyle{theorem}
\newtheorem{theorem}[equation]{Theorem}
\newtheorem{proposition}[equation]{Proposition}
\newtheorem{corollary}[equation]{Corollary}

\newtheorem{lemma}[equation]{Lemma}

\theoremstyle{remark}
\newtheorem{remark}[equation]{Remark}

\theoremstyle{definition}
\newtheorem{definition}[equation]{Definition}

\newcommand{\RR}{\mathbb{R}}

\newcommand{\CC}{\mathbb{C}}
\newcommand{\ZZ}{\mathbb{Z}}
\newcommand{\e}{\epsilon}
\newcommand{\del}{\partial}

\newcommand{\calC}{{\mathcal C}}
\newcommand{\calD}{{\mathcal D}}
\newcommand{\calF}{{\mathcal F}}

\newcommand{\calI}{{\mathcal I}}
\newcommand{\calL}{{\mathcal L}}

\newcommand{\calO}{{\mathcal O}}
\newcommand{\calQ}{{\mathcal Q}}
\newcommand{\calR}{{\mathcal R}}

\newcommand{\calU}{{\mathcal U}}
\newcommand{\calV}{{\mathcal V}}

\newcommand{\Ric}{\mathrm{Ric}}

\newcommand\SU{\operatorname{SU}}

\newcommand{\uone}{\underline{1}}
\newcommand{\utwo}{\underline{2}}

\title{Fredholm theory for elliptic operators on quasi-asymptotically
  conical spaces}
\author{Anda Degeratu \\ University of Freiburg  \and Rafe Mazzeo \\ Stanford University}

\date{\today}

\begin{document}
\maketitle

\begin{abstract}
We consider the mapping properties of generalized Laplace-type operators $\sL = \nabla^* \nabla + \calR$ 
on the class of quasi-asymptotically conical (QAC) spaces, which provide a Riemannian generalization of 
the QALE manifolds considered by Joyce \cite{Joyce2000}. Our main result gives conditions under which 
such operators are Fredholm when between certain weighted Sobolev or weighted H\"older spaces. 
These are generalizations of well-known theorems in the asymptotically conical (or asymptotically 
Euclidean) setting, and also sharpen and extend corresponding theorems by Joyce.  The methods here
are based on heat kernel estimates originating from old ideas of Moser and Nash, as developed further
by Grigor'yan and Saloff-Coste.  As demonstrated by Joyce's work, the QAC spaces here contain
many examples of gravitational instantons, and this work is motivated by various applications 
to manifolds with special holonomy. 
\end{abstract}

\section{Introduction}\label{sec:introduction}
Let $(Z,g)$ be a complete noncompact Riemannian manifold, and suppose that $\calL = \nabla^* \nabla + \calR$ is a 
generalized Laplace-type operator acting on sections of some bundle $E$ over $Z$.  There are now many different
settings, involving asymptotic conditions on the geometry of this space and on the potential $\calR$, which ensure 
that the action of $\sL$ between appropriate weighted Sobolev or H\"older spaces is Fredholm, or more generally,
just that its nullspace (in one of these weighted spaces) is finite dimensional. The most general 
results assume only certain bounds on the geometry, for example lower bounds on the Ricci curvature and 
the injectivity radius.  To obtain more refined results, however, one typically imposes stringent conditions 
on the asymptotic regularity of the metric, which leads to results about the asymptotics at infinity 
of solutions of $\sL u = 0$, and from there to various existence and
uniqueness results for related nonlinear problems.

To illustrate this range of hypotheses and results, consider the setting where $(Z,g)$ is  ``Euclidean at infinity''. 
One well-known, but very weak, formulation of this asymptotic condition is that balls of radius $r$ have 
volume growing no faster than a fixed constant times $r^n$, with $n = \dim Z$; we say then that $(Z,g)$ has Euclidean
volume growth. One much stronger condition which implies Euclidean volume growth is that $\Ric(g) \geq 0$. 
A famous conjecture by Yau asks whether, for manifolds with nonnegative Ricci curvature,
the space of harmonic functions on $Z$ which grow no faster than $C(1+r^d)$  is finite dimensional for
any fixed $d > 0$.  Yau's conjecture was resolved, in the beautiful work of Colding-Minicozzi \cite{Colding1997} 
and also Li \cite{Li1997}, who prove that this finite dimensionality holds whenever $(Z,g)$ has the
volume-doubling (VD) property and admits a scale-invariant Poincar\'e
inequality (PI).  It is not hard
to see that spaces with (VD) have polynomial (but not necessarily Euclidean) volume growth,
and also that both properties hold when $g$ has nonnegative Ricci curvature. A far-reaching and extensive 
investigation of the implications of these two properties was carried out in the 
work of Grigor'yan and Saloff-Coste, as well as by Sturm. One remarkable set of results here is the fact that 
the two properties (VD) and (PI) are equivalent to the existence of a parabolic Harnack inequality, and that in 
turn is equivalent to the existence of upper and lower Gaussian bounds for the heat kernel, see the books of 
Grigor'yan \cite{Grigoryan2009} and Saloff-Coste
\cite{SaloffCoste2002}, as well as the papers of
Sturm~\cite{Sturm1994, Sturm1995, Sturm1996} for more on this. 

Now suppose that we require $(Z,g)$ to be strongly asymptotic at infinity either to $\RR^n$, or to a quotient 
$\RR^n/\Gamma$, where $\Gamma$ is a finite group of rotations which acts freely on $S^{n-1}$. Such spaces are
called asymptotically Euclidean (AE) and asymptotically locally
Euclidean (ALE), respectively. They are important  in other parts of geometric analysis and mathematical physics; for example, AE spaces play a central 
role in mathematical relativity, while ALE spaces constitute some of the basic models for solitons in Ricci and other 
geometric flows, and the simplest examples of gravitational instantons have an ALE geometry.  Slightly
more generally, we could also require $(Z,g)$ to be asymptotic to a cone over a more general compact Riemannian
manifold $(Y,h)$, which means that $g$ is a `scattering metric' in the sense of Melrose \cite{Melrose1995}. All of these 
spaces have sectional curvatures decaying quadratically in the distance from a fixed compact set. 
In this setting, one can develop quite precise generalizations of the main results of Euclidean scattering theory,
i.e., derive complete asymptotic expansions at infinity for the resolvent kernel and for solutions of the
Helmholtz equation $(\Delta - \lambda^2)u = 0$. This uses the refined tools of geometric microlocal analysis. 

Our interest in this paper is with a slightly more general class of spaces, the motivation for which comes from
the important class of QALE, or quasi-asymptotically locally Euclidean, spaces in complex geometry. These arise 
as resolutions of singular quotients $\CC^n/\Gamma$, where $\Gamma$ is a finite subgroup of $\SU(n)$. The 
quotient $\CC^n/\Gamma$ is a cone over the cross-section $S^{2n-1}/\Gamma$, and if the action of $\Gamma$ 
on the sphere is free, then a QALE space is ALE. In general, this action is not free and the singularities of this 
cone extend to infinity.  A resolution $Z$ of this quotient is a smooth complex manifold equipped with a holomorphic 
map $Z \to \CC^n/\Gamma$, and $Z$ is called a crepant resolution if it satisfies some further topological
conditions which we do not explain here.  A remarkable and beautiful theorem due to Joyce, exposed at length 
in his monograph \cite{Joyce2000}, guarantees the existence of Ricci-flat QALE K\"ahler metrics on crepant resolutions
satisfying a few additional conditions. In other words, Joyce settles the analogue of the Calabi conjecture in this 
particular noncompact setting. A QALE space $(Z,g)$ has Euclidean volume growth, and has ALE asympotics 
along a dense open set of rays converging to infinity, but its full asymptotic structure is more complicated.

In this paper we introduce a Riemannian generalization of these QALE spaces which we call QAC, for quasi-asymptotically 
conic. These stand in the same relationship as general scattering metrics (which we shall call AC -- or asymptotically
conic -- spaces) do to the ALE spaces. Namely, we do not require any particular curvature properties and their 
topology at infinity is not connected to any group action.  Our reason for working in this more general setting
is to develop flexible methods for studying the analysis of elliptic operators on these spaces without appealing directly 
to the complex structure or special holonomy. The definition of QAC spaces is inductive and somewhat complicated,
and \S 2 of this paper describes many aspects of the topology and geometry of these spaces in detail. 

Our main result is a Fredholm theorem for generalized Laplace operators, acting on sections of vector bundles 
over a QAC space $(Z,g)$.   By definition, a generalized Laplacian $\calL$ is an elliptic operator of the form
$\nabla^* \nabla  + \calR$, where $\nabla$ is a connection on a Hermitian vector bundle $E$ and where
$\calR \in \operatorname{End}\,(E)$. In the simpler setting of AC
manifolds it is well known that to obtain Fredholm results, one must let $\calL$ act between Sobolev (or H\"older) spaces with a weighted 
measure.  Thus, let $\rho^{\delta+\frac{n}{2}} L^2(Z, dV_g) = \{ \rho^{\delta+\frac{n}{2}} v: v \in L^2\}$, where $\rho$ is a smooth everywhere positive 
function which is asymptotic to the radial distance on the model cone
at infinity as $\rho$ goes to infinity. 
Weighted Sobolev spaces are defined in an obvious way. A typical -- and now classical -- result is that 
\begin{equation}
\calL\colon\rho^{\delta+\frac{n}{2}} H^2(Z,dV_g) \longrightarrow \rho^{\delta+\frac{n}{2}-2}L^2(Z, dV_g)
\end{equation}
is Fredholm provided that $\delta$ does not lie in a certain discrete set of values $\{\delta_j^\pm\}$, which
$\delta_j^\pm \to \pm \infty$. These omitted values are determined by global spectral data of an induced 
operator on the asymptotic cross-section $Y$ of the AC space $Z$. There is an analogous result for $\calL$ 
acting between weighted H\"older spaces. 

Our goal is to generalize this to QAC spaces. One main issue to be faced is that we must now use not just the radial 
function $\rho$, but also a collection of other functions, $w_1, \ldots, w_k$ defined in \S 2.4 below, to define the 
weighted measures. The integer $k$ is called the depth of the space $Z$ and provides a measure of the
complexity of the inductive definition of this space (or of the stratified structure of its tangent cone at 
infinity).  We use multi-index notation, writing $w^\tau$ for $w_1^{\tau_1} \ldots w_k^{\tau_k}$. 
Deferring the (somewhat intricate) definitions of these weight functions for the moment, we now state our main
\begin{theorem}
Let $(M,g)$ be a QAC space. Let $\calL = \nabla^* \nabla + \calR$ be a generalized Laplacian, and suppose that 
$\calR \geq V\! \cdot \mathrm{Id}$, where $V$ is a scalar function which takes the form $- \Delta( \rho^a w^b)/\rho^a w^b$ 
on each end of $M$\footnote{In this work, $\Delta$ denotes the
  positive Laplacian, i.e.\ $\Delta u = - \mathrm{div}(\mathrm{grad}\,u)$.}. Then 
\begin{equation}
\calL\colon\rho^{\delta+\frac{n}{2}} w^{\tau + \frac{\nu}{2}} H^2(Z;E) \longrightarrow \rho^{\delta + \frac{n}{2} - 2} w^{\tau + \frac{\nu}{2}- \utwo}L^2(Z;E)
\label{mF-Intro}
\end{equation}
is Fredholm provided
\[
2-n-\frac{a}{2} < \delta < \frac{a}{2}, \qquad \utwo - \nu - \frac{b}{2} \leq \tau \leq \frac{b}{2}.
\]
Here $\nu$ is a $k$-tuple of constants related to the dimensions of certain components that arise in a 
topological decomposition of $Z$, see \eqref{shiftparam}, and $\utwo$ is the $k$-tuple $(0, \ldots, 2)$. 
\label{mth}
\end{theorem}
This theorem is proved in \S 7, along with the analogous Fredholm theorem for $\calL$ acting between weighted
H\"older spaces. 

We prove this following the methods of Grigor'yan and Saloff-Coste. Namely we show that $(Z,g)$ satisfies
appropriate weighted versions of the volume-doubling (VD) property and
uniform Poincar\'e inequality (PI). These 
are then used to deduce estimates for the heat kernel associated to
the scalar operator $\Delta + V$; integrating 
from $t=0$ to $t=\infty$, we obtain estimates for the Schwartz kernel $G_{\Delta+V}$ of the Green operator $(\Delta+V)^{-1}$. 
The reduction from the elliptic system on weighted spaces to the scalar operator on unweighted $L^2$ follows 
from standard domination techniques.   For the very special case where the weight parameters $a$ and $b$ 
all vanish, so $V \equiv 0$, these techniques lead to the familiar bound
\begin{equation}
G_\Delta(z, z') \asymp  d(z,z')^{2-n}, \quad n = \dim Z.
\label{modelG}
\end{equation}
With such an estimate, one can prove that for the scalar Laplacian
\[
G_\Delta\colon \rho^{\delta+\frac{n}{2}-2} w^{\tau+\frac{\nu}{2} - \utwo}
L^2(Z,dV_g) \longrightarrow \rho^{\delta+\frac{n}{2}} w^{\tau + \frac{\nu}{2}} H^2(Z,dV_g),
\]
provided $2-n < \delta <0$ and $\utwo - \nu \leq \tau \leq 0$, which is the same as \eqref{mF-Intro} when $a = b = 0$.

The main technical part of our work, then, is to prove that QAC spaces $(Z,g)$ satisfy the two properties (VD) and (PI). 
The rather elaborate argument relies heavily on the recursive definition of these spaces. The reader may be surprised 
that so much work is necessary simply to prove the most basic Fredholm properties, especially since the corresponding 
properties on AC spaces are much simpler, and may be proved in just a few pages. We can only offer one explanation: 
the analysis of elliptic operators on an AC space $Z$ depends heavily on the global properties of the induced operators
on the asymptotic cross-section $Y$, which is a compact smooth manifold. However, when $Z$ is a QAC space, this 
asymptotic cross-section $Y$ is (in a sense to be explained below) a compact stratified space, and the analysis of elliptic 
operators on such spaces is considerably more complicated than in the smooth case. 

We have already mentioned Joyce's nonlinear analysis on QALE spaces, but a crucial ingredient in \cite{Joyce2000} 
is a Fredholm result for the scalar Laplacian, which is very similar to, but less general than, our main result. 
He considers the scalar Laplacian only, but as explained above, our argument involves a reduction to a scalar
operator too. More importantly, because he relies on the maximum principle using certain barrier functions
which he constructs, he obtains the result only for a smaller set of values of the weight parameters. A key
part of our initial motivation to study this problem was to find ways to extend his results to elliptic systems, and 
to obtain the result on the optimal range of weighted spaces. The extension to systems is motivated, in turn, 
by several intended applications to index theory of generalized Dirac operators on these spaces.  Obtaining
the ``correct'' restrictions on the set of allowable weight parameters is expected to be important in other
applications to nonlinear geometric problems. 

At the end of this paper we briefly review Joyce's results in detail and explain their relationship to ours.
We also describe there a mapping property of the Laplacian in the QALE setting derived by Carron~\cite{Carron2011}
which is closely related but slightly less sharp, and finally an interesting corollary about the dimension
of the nullspace of Laplacians without restriction of weights which follows from the work of 
Colding-Minicozzi~\cite{Colding1998} and Li~\cite{Li1997}.

The most precise properties of the resolvent and heat kernel of $\calL$ should presumably be obtained using
the methods of geometric microlocal analysis, see \cite{Melrose1996}, \cite{Hassell2001} and \cite{Mazzeo1991}
for examples of this.  However, such arguments will be more intricate than the ones here, and the methods
and results in this paper suffice for many intended uses.  We intend to revisit this theory using those more
intricate techniques in a later paper. 

To conclude, we note that while the simplest gravitational instantons are ALE, these are just the first in 
a hierarchy of asymptotic geometries. Further examples include the ALF, ALG and ALH spaces, which 
can be thought of as (singular) torus bundles over lower-dimensional ALE spaces. There is a satisfactory
elliptic theory (using geometric microlocal methods) on these more complicated spaces as well.  Amongst the
many explicit spaces with special holonomy, it is now becoming clear that many of these, for example the 
monopole moduli spaces on $\RR^3$ with their Weil-Petersson type hyperK\"ahler metrics, appear to
be singular torus fibrations over QALE spaces.  The first step in analyzing elliptic operators on this important
panoply of spaces is to develop a general elliptic theory on their QALE bases; this present paper is an 
initial attempt at such an analysis. 

\medskip

\noindent{\bf Acknowledgements:}  This work was carried out over a period of years, and we gained from
the insight and advice of many people during this period. We wish to thank, in particular, Gilles Carron for
many illuminating conversations and his great encouragement, as well as Michael Eichmair and Charles Epstein 
for very helpful advice. R.M. was supported by NSF grant DMS-1105050 during the last period of this work. 

\renewcommand{\arraystretch}{1.25} 

\section{QAC geometry}\label{sec:geometries}
In this section we provide details about the class $\sQ$ of quasi-asymptotically conic spaces. As indicated in the 
introduction, the definition of these spaces is recursive and involves two further types of spaces: $\sI$, the class of 
compact stratified spaces with iterated edge metrics, and $\sD$, the `resolution blowups' of elements of $\sI$.  
A QAC space $(Z,g)$ is a  smooth Riemannian manifold which is asymptotically conic in a rather precise sense. 
There is a compact set $K_Z \subset Z$ such that the complement $Z \setminus K_Z$ is diffeomorphic to the exterior cone 
$C_{1,\infty} (Y)$ for some smooth manifold $Y$. The cross-section of $Z$ at radius $\rho$ is diffeomorphic to $Y$ 
of course, but the induced metric, which we write as $h_{1/\rho}$ (to fit with notation below), becomes singular as 
$\rho \to \infty$. Indeed, $(Y, h_{1/\rho})$ converges to the stratified space $Y_0$ with the iterated edge metric $h_0$. 
Consequently, the tangent cone to $(Z,g)$ at infinity is a cone over $(Y_0,h_0)$. These cross-sections are thus a 
family of (metric) resolutions of $Y_0$.   

The recursive nature of this definition is hidden in the fact that the resolution $Y$ of $Y_0$ is described in 
terms of a set of QAC spaces, iterated edge spaces and their resolutions, each of lower complexity than those
appearing in the resolution $Y$. We measure the complexity of this construction by the length of the recursive
definition, which is called the ``depth''.  The QAC spaces of depth $0$
are the asymptotically conic spaces, in the 
usual sense. 

This entire construction is modelled on and inspired by a well-known resolution procedure in complex geometry, leading
to the class of QALE (for quasi-asymptotically locally Euclidean) spaces; these were brought into the geometric analysis 
community through the work of Joyce \cite{Joyce2000}. What we do here is to provide a less rigid formulation 
of that complex resolution process, adapted to the category of real stratified spaces and Riemannian geometry. 

This section is long and somewhat technical. We begin by presenting a simple example from complex geometry
which illustrates the main ideas. This is followed by a review of the definition of stratified spaces with iterated edge
metrics, and their total blowups. At that point we embark on the somewhat intricate description of how to build
resolution blowups and QAC spaces of increasing depth from simpler spaces of these types. The section concludes
with a decomposition scheme for these spaces used extensively later in this paper. 

\subsection{A motivating example: the algebraic geometric resolution  of singularities of $\IC^3/\IZ_4$ as a QAC space}
We begin with a simple example, intended to provide both motivation and intuition for the more general definitions below. 
This is the algebraic geometric resolution of singularities of $\IC^3/\IZ_4$, viewed here as a QAC space of depth $1$ 
with tangent cone at infinity the cone over $\IS^5/\IZ_4$. We start with the algebraic geometric construction, and then 
reinterpret it in the language of iterated edge spaces, resolution blowups, and QAC spaces. 

\subsubsection*{The algebraic geometric resolution of singularities of $\IC^3/\IZ_4$}
As described in the introduction, if $\Gamma$ is a finite subgroup of $\mbox{SU}(n)$, then $\CC^n/\Gamma$ is a cone 
over the stratified space $Y_0=\IS^{2n-1}/\Gamma$, and there is a well-known procedure in algebraic geometry to resolve 
the singularities of this quotient.  The complexity of the resolution depends on the complexity of the partially ordered 
set of isotropy subgroups of this action, and of the linear subspaces fixed by these isotropy subgroups.  The simplest 
situation is when $\Gamma$ fixes only the origin, in which case $\CC^n/\Gamma$ is a cone over the smooth
manifold $\IS^{2n-1}/\Gamma$. We first blow up the origin in $\CC^n/\Gamma$. The resulting space might still be 
singular, but its singular points have lower complexity, and lie in a bounded region. Further blowups decrease the 
complexity of the singularities. After finitely many steps, the resulting space is smooth. This is the ALE resolution of $\CC^n/\Gamma$.

Now consider the next simplest case, where $\Gamma \cong \ZZ_4$ is the subgroup of $\SU(3)$ generated by 
$\alpha(z_1, z_2, z_3)  = (iz_1, iz_2, -z_3)$. The fixed point set and isotropy group structure is now more complicated: 
the origin $(0,0,0)$ is fixed by all of $\Gamma$, while the axis $(0,0) \times \CC$ is fixed by the subgroup $\{1, \alpha^2\} 
\cong \ZZ_2 \subset \Gamma$.  The singular locus of $\IC^3/\IZ_4$ is the image of these fixed point sets. This quotient 
has two strata: the one of highest codimension is the image the origin, and the other is the image of $(0,0) \times \CC \setminus \{(0,0,0)\}$.
Near a point $[0,0,z_3]$ in this lower codimension stratum, $\IC^3/\ZZ_4$ is modelled by a neighborhood of the point
$([0,0],0)$ in $(\CC^2/\ZZ_2) \times \CC$, where $\ZZ_2$ is the stabilizer of the point.  Note that the action of $\ZZ_2$ on
$\CC^2$ fixes only the origin.

The resolution of $\CC^3/\ZZ_4$ is accomplished in two steps. Let
\[
\pi\colon Z^{(0)} \to \IC^2/\IZ_2
\]
be an ALE resolution of $\IC^2/\IZ_2$ (as described above). Then $Z^{(0)}\times \CC$ is a resolution of $(\CC^2/\ZZ_2) \times \CC$,
and the induced action of the quotient group $\Gamma/\ZZ_2$ on this product fixes $\pi^{-1}([0,0])\times \{0\}$. The fixed point 
structure of this new action is simpler, and if we blow it up, we obtain a smooth manifold $Z$. This is a QALE resolution of $\CC^3/\ZZ_4$. 

\subsubsection*{The algebraic geometric  resolution of singularities of
  $\IC^3/\Gamma$ as a QAC space of depth $1$}
Let us now reconsider this QALE space $Z$ and provide a somewhat different point of view on how to construct it. 
Regard $\CC^3/\Gamma$ as the cone $C(\IS^{5}/\Gamma)$. Its cross-section $Y_0 = \IS^5/\Gamma$ has a simple edge singularity
along the image of the fixed point set of the action of $\Gamma$ on $\IS^5$. This fixed point set is the circle 
$\{(0,0,z_3) : |z_3| =1\} \subset \IS^5$, and it has stabilizer $\ZZ_2$. The round metric on $\IS^5$ induces
a metric $h_0$ on the quotient, and in the notation below, $(Y_0,h_0) \in \sI_1$, i.e., it is an iterated edge space of depth $1$.

In the algebraic geometric resolution of $\IC^3/\IZ_4$, we first resolve the singularities of $(Y_0,h_0)$. Note that
each fiber of the normal bundle to the singular stratum $S_1 \subset Y_0$ is identified with $\CC^2/\ZZ_2 = C (\IS^3/\ZZ_2)$;
these can be resolved by replacing some small neighborhood of the singular point in each normal fiber with a truncation of 
the resolution $\pi\colon Z^{(0)} \to \CC^2/\ZZ_2$ above. This yields the resolution blowup $Y$ of $Y_0$.  We can simultaneously
`resolve' the conic metric, and thus obtain a family of metrics $h_{1/\rho}$ with parameter $\rho$ which corresponds to 
the scale at which we are truncating $Z^{(0)}$. (This notation is meant to indicate that $(Y,h_{1/\rho}) \to (Y_0,h_0)$ as 
$\rho \to \infty$ in the Gromov-Hausdorff topology -- and in much stronger senses too.) Using these metric resolutions,
we obtain a resolution of the exterior cone $C_{1,\infty}(Y_0)$, 
\[
Z^{\text{ext}} = C_{1,\infty} (Y),
\] 
with metric 
\[
d\rho^2 + \rho^2 h_{1/\rho}.
\]
This space has boundary $Y$ at $\rho = 1$, and the full resolution $Z$ is obtained by replacing the compact
portion of the cone $C_{0,1}(Y_0)$ by a smooth compact manifold with boundary $K$ with $\del K = Y$.  In
other words, the full QAC space $Z$ is the union $Z^{\mathrm{ext}} \cup K$; the metric $g_Z$ is any smooth 
extension of the metric on $Z^{\text{ext}}$ over $K$.

\subsection{Iterated edge spaces}\label{ssec:it_cone_edge}  
The first step, as promised, is a review of iterated edge geometry. 

\subsubsection{Smoothly stratified spaces}
The basic differential topology of stratified spaces is somewhat intricate, cf.\ the foundational monograph by
Verona \cite{Verona1984} and the cogent exposition by Pflaum \cite{Pflaum2001}.  Basic definitions vary between sources, 
which motivated the effort in \cite{Albin2012} to clarify some of this material. In particular, \S 2 of that paper presents 
the structural axioms of smoothly stratified pseudomanifolds, and shows that these spaces are the same ones as the 
iterated edge spaces considered by Cheeger \cite{Cheeger1984}, and more recently by the second author \cite{Mazzeo2006}. 
That paper also describes a resolution of any such space as a manifold with corners, obtained by successively blowing 
up the strata in order of decreasing depth; the resulting space is endowed with an iterated fibration structure on its 
boundary faces. It is also shown there, conversely, that any manifold with corners with an iterated 
fibration structure can be blown down to a smoothly stratified pseudomanifold.   We recall some of this
here. All of this is taken from \S\S 2-3 of \cite{Albin2012}, to which the reader is referred for more details. 

By definition, a compact smoothly stratified pseudomanifold $Y_0$ decomposes into a disjoint union of connected 
strata, $S = \sqcup_{\alpha \in A} S_\alpha$, where each $S_\alpha$ is a (possibly open) manifold of dimension $d_\alpha$; 
one or more of these strata have maximal dimension $n$, and the union of these maximal strata are dense in $Y_0$. 
The index set $A$ is in bijective correspondence with the set of all strata; it is a partially ordered set, where 
$\alpha > \beta$ if $S_\alpha \subset \overline{S_\beta}$ (this inversion from what would seem natural is because
we wish to order the strata by depth, see below). In particular, with this ordering, the maximal dimensional strata are 
minimal elements. Various axioms describe how the strata 
fit together; for our purposes, the key one is the fact that each $S_\alpha$ has a `tubular neighborhood' $\calU_\alpha$ 
which is the total space of a smooth bundle $\pi_\alpha\colon \calU_\alpha \to S_\alpha$ with fiber a truncated cone 
$C_{0,1}(Y_0^{(\alpha)})$; the cross-section (or link) $Y_0^{(\alpha)}$ of these cones is itself a compact stratified space. This 
process of taking a cone, or bundle of cones, increases the `depth'  of the stratification, and induces a 
decomposition of $Y_0$ into the union of strata of a given depth. We denote by $\delta_\alpha$ the depth of $S_\alpha$. 

Note that we require the fibration of each tubular neighborhood $\calU_\alpha$ to have smooth local trivializations. 
This excludes certain stratified spaces for which these local trivialization functions are continuous but not smooth,
see \cite[\S2]{Albin2012} for an example. Moreover, it is not immediately clear what it means for a map to be smooth
between $\pi_\alpha^{-1} (\calV_\alpha)$ and $\calV_\alpha \times C_{0,1}(Y_0^{(\alpha)})$ (where $\calV_\alpha \subset S_\alpha$
is a small open set) since both domain and range spaces are singular. This notion of smoothness is also 
defined recursively: once one has defined stratified diffeomorphisms between compact smoothly stratified 
spaces of depth at most $k-1$, then the extension to the corresponding cones is provided by suspending these 
diffeomorphisms, and the passage to bundles of cones, and hence to arbitrary spaces of depth $k$, follows from this. 

So far we have suppressed the role of the metric, but we wish to regard these as Riemannian spaces, and hence 
consider suitable pairs $(Y_0, h_0)$.  The precise behavior of the metric is as follows. On each conic fibre $C_{0,1}(Y_0^{(\alpha)})$ 
in any one of the tubular neighborhoods $\calU_\alpha \supset S_\alpha$, we consider either exact conic metrics
$ds^2 + s^2 h_0^{(\alpha)}$, where $h_0^{(\alpha)}$ is an (inductively defined) iterated edge metric on $Y_0^{(\alpha)}$, or else perturbations 
of these of the form $ ds^2 + s^2 h_0^{(\alpha)} + \eta$, where $\eta$ is a smooth perturbation which decays as $s \to 0$, i.e.\ 
$|\eta|_{ds^2 + s^2 h_0^{(\alpha)}} \leq C s^\epsilon$ for some $\epsilon > 0$.  On the tubular neighborhoods $\calU_\alpha$ which are bundles of cones we consider 
metrics $ds^2 + s^2 h_0^{(\alpha)} + \pi_\alpha^* q_\alpha$, where the last term is the pullback of a metric on $S_\alpha$ to
$\calU_\alpha$, or decaying smooth perturbations of these. For the types of problems we study in this paper, it is not
necessary to specify the precise regularity or decay.   This class of metrics was already considered by Cheeger \cite{Cheeger1984},
and he called such a space $(Y_0, h_0)$ conelike. We shall call these iterated edge spaces. 

There is an alternate, and more directly inductive, approach. 
\begin{definition}\label{defn:I_k}
For each $k \geq 0$, define the class $\sI_k$ of compact iterated edge spaces of depth $k$:
\begin{itemize} 
\item An element $(Y_0, h_0) \in \calI_0$  is a compact smooth Riemannian manifold;
\item $(Y_0, h_0) \in \sI_k$ if there is a decomposition $Y_0 = Y_0' \cup Y_0''$, where $(Y_0'', h_0)$ is an
element of $\sI_{k-1}$ with a codimension one boundary along the intersection $Y_0' \cap Y_0''$,
and each component of $Y_0'$ is the total space of a smooth cone bundle over a smooth compact 
base space $S_k$ with fiber a truncated cone $C_{0,1}(Y_0^{(k-1)})$, where $(Y_0^{(k-1)}, h_0^{(k-1)}) \in \sI_{k-1}$. The common 
boundary $\del Y_0' = \del Y_0''$ is also a stratified space of depth $k-1$; it is the total space 
of a bundle over the same base $S_k$ with fibre $Y_0^{(k-1)}$. The base $S_k$ is the maximal depth stratum of $Y_0$.  
\item If $Y_0 \in \sI_k$, then its dimension is given, relative to the decomposition above, by $\dim S_k + \dim Y_0^{(k-1)} + 1$.
\end{itemize} 
\end{definition}
This definition implies that the maximal depth stratum in $Y_0 \in \sI_k$  is necessarily the union of compact smooth manifolds. 

To simplify the notation in the rest of the paper, we let $S_j$ denote the union of all the singular strata of depth
$j$, and write the link along this stratum as $(Y_0^{\d{j-1}},
h_0^{\d{j-1}}) \in \sI_{j-1}$. We thus sometimes indicate the
depth of a particular stratified space explicitly using the
superscript $(j)$; other times, when this information is not
needed, we will ignore it, as the notation tends to become too cluttered. 

\subsubsection{The total blowup}\label{ssec:tb}
We now define, for any iterated edge space $Y_0 \in \sI_k$, its total blowup $\widetilde{Y}_0$, which is a manifold
with corners up to codimension $k$. This is used in several
constructions and arguments in this section.

When $Y_0$ has only isolated conic singularities, this blowup procedure is well known. Each conic point $p$ has a 
tubular neighborhood $\calU$ which is identified with a truncated cone $C_{0,1}(F)$, where the link $F$ is a compact 
smooth manifold. The blowup is obtained by replacing $\calU$ with $[0,1]\times F$. If $s$ is polar distance in this 
cone with respect to some metric $ds^2 + s^2 h$ and $y$ is a local coordinate system on $F$, then the polar coordinate 
system $(s,y)$ lifts to a nonsingular coordinate system on this cylinder, and by fiat determines the smooth structure 
on $\widetilde{Y}_0$ near this boundary component. Carrying this out at all conic points in $Y_0$ yields the total 
blowup $\widetilde{Y_0}$ of $Y_0$. 

If the singularities of $Y_0 \in \sI_1$ are not isolated, then we say that $Y_0$ has simple edges.  In this case,
the singular stratum $S$ (for convenience we assume it is connected) has a tubular neighbourhood $\calU$ 
which is a bundle of truncated cones $C_{0,1}(F)$ over $S$. We can then perform this same blowup of the vertex
in each conic fiber, which gives a manifold with boundary $\widetilde{Y}_0$, the boundary of which is the total 
space of a fibration over $S$ with fibre the link $F$. The tubular neighbourhood $\calU$ lifts to a neighbourhood 
$\tilde{\calU}$ of $\del \widetilde{Y_0}$ which is a bundle of cylinders $[0,1]\times F$ over $S$. 

To define the total blowup of any $Y_0 \in \sI_k$, we use the doubling construction from \cite{Verona1984}, see also 
\cite[\S 2]{Albin2012}, which allows us to appeal to induction. If $Y_0 \in \sI_k$ and $k > 1$, then suppose that 
the total blowup $\widetilde{Y_0^{(j)}}$ has been defined for every $Y_0^{(j)} \in \sI_j$ whenever $j < k$.  
Assume for simplicity that $S_k$, the stratum of maximal depth $k$, is connected. The tubular neighbourhood 
$\calU_k\supset S_k$ is a bundle over $S_k$ with fiber $C_{0,1}(Y_0^{(k-1)})$, for some $Y_0^{(k-1)} \in \sI_{k-1}$,
and similarly, the outer boundary $\del \calU_k$ fibres over $S_k$ with fibre $Y^{(k-1)}_0$.
Now define
\[
2Y_0 = -(Y_0 \setminus S_k) \sqcup (Y_0 \setminus S_k) \sqcup ((-1,1) \times \del \calU_k).
\]
The first term on the right is $Y_0 \setminus S$ with the opposite orientation; each fibre $(-1,0)\times Y_0^{(k-1)}$ 
in the first term and $(0,1)\times Y_0^{(k-1)}$ in the second is attached in the obvious way using the bridge 
$(-1,1)\times Y_0^{(k-1)}$ from the third term. The space $2Y_0$ has an obvious involution $\tau$, and 
is a union of two stratified spaces with codimension one boundary $Y_0^\pm$ meeting along their common boundary.

Manifestly, $2Y_0 \in \sI_{k-1}$, and by induction its total blowup $\widetilde{2Y}_0$, which is a manifold with corners 
up to codimension $k-1$, is assumed to have already been defined. We can do this in a $\tau$-invariant way.
The total blowup $\widetilde{Y}_0$ of $Y_0$ itself is the portion of $\widetilde{2Y}_0$ lying  over the total blowup 
of $Y_0^+$. This is clearly a manifold with corners of codimension $k$.

This total blowup space has a lot of extra structure. Each boundary hypersurface $H_\alpha$ of $\widetilde{Y}_0$ 
corresponds to precisely one singular stratum $S_\alpha$ of $Y_0$; each $H_\alpha$ is the total space of a fibration
with fibre the total blowup of the link in the cone bundle decomposition of the tubular neighbourhood $\calU_\alpha$ 
around $S_\alpha$. The fibres of adjacent faces in $\widetilde{Y}_0$ fit together in a manner dictated by the 
inclusion relations of the closures of the corresponding singular strata. This ensemble of compatible fibrations on 
the boundary faces of a manifold with corners is called an iterated fibration structure. It is proved in \cite{Albin2012} that 
there is a bijective correspondence between iterated edge spaces and manifolds with corners with iterated fibration 
structures, where the association between the objects in either class is by blowup or blowdown. Note that there is a 
partial ordering of the boundary faces $H_\alpha$ of $\wt{Y}_0$, where $H_\alpha < H_\beta$ if $S_\alpha$ is contained 
in the closure of $S_\beta$. If $Y_0$ has depth $k$, then the length of the longest such chain is $k$; similarly, 
the depth $\delta_\alpha$ of $S_\alpha$ is the length of the longest chain with initial term $H_\alpha$. 

For each $\alpha \in A$, let $\tilde{s}_\alpha$ be a defining function for the boundary hypersurface $H_\alpha$ of
$\widetilde{Y_0}$. These defining functions can be chosen to be constant on the fibres of every other boundary hypersurface, 
and hence they can be pushed forward to define functions $s_\alpha$ on $Y_0$. Thus $s_\alpha$ serves as a radial function
on each conic fibre in the tubular neighbourhood $\calU_\alpha$ around $S_\alpha$. By construction, the $\tilde{s}_\alpha$ 
are independent of one another; the corresponding property of independence for the $s_\alpha$ is not so easy
to formulate directly on $Y_0$ without passing through the blowup. 

As motivated further in \S 3, we only consider metrics $h_0$ on $Y_0$ for which the functions $s_\alpha$ are approximate distance 
functions to the corresponding strata $S_\alpha$. For simplicity, we assume that these metrics are scaled so that each $s_\alpha$ 
takes values in $[0,1]$, and $S_\alpha = \{s_\alpha = 0\}$. The $|A|$-tuple $(s_\alpha)$ is the family of radial functions for $Y_0$. 
As at the end of the last subsection, we denote the union of strata of depth $j$ by $S_j$.  The components of any $S_j$ 
may have different dimensions, but since our considerations around these strata are local, we typically 
assume that each $S_j$ is connected, so that the strata and the radial functions are indexed by $j \in \{1, \ldots, k\}$. 

\subsection{The classes $\sD$ and $\sQ$}\label{ssec:D_Q}
The classes $\sD$ and $\sQ$, of resolution blowups of elements of $\calI$ and of QAC spaces, respectively, are closely 
intertwined with one another. These are again filtered by depth, $\sD = \sqcup \sD_k$, and $\sQ = \sqcup \sQ_k$,
with each subclass defined inductively: the definition of the elements in $\sD_k$ depend on $\sQ_j$, $j \leq k$, 
while the elements of $\sQ_k$ are constructed using spaces in $\sD_j$, $j < k$.  

We now explain this more carefully, starting first with the spaces in $\sQ_0$ and $\sD_1$, 
and then proceeding to the general inductive construction.

\subsubsection{The family $\sQ_0$}
An element $(Z,g_Z) \in \sQ_0$ is, by definition, a smooth asymptotically conic (AC) manifold. Thus there exists some compact 
set $K_Z \subset Z$ such that $Z \setminus K_Z$ is diffeomorphic to a product $(R_0,\infty) \times F$, where $F$ is a smooth compact
manifold. The metric $g_Z$ on $Z$ is arbitrary in $K_Z$ and in this exterior admits a decomposition $g_Z = g_0 + g'$ with $g_0 = 
d\rho^2 + \rho^2h_F$ and $|g'|_{g_0} = \calO(\rho^{-\nu})$ for some $\nu > 0$, where $h_F$ is a fixed metrix on $F$. (Note that, 
trivially, the cross-section $(F,h_F)$ lies in $\calD_0$.)  The smooth
radial function $\rho\colon Z\setminus K_Z \to (R_0,\infty)$ extends to a 
smooth positive function on all of $Z$, which we again denote by $\rho$.  It is a distance function near 
infinity.  For simplicity of notation, we suppose that $R_0 = 1$, so that $\rho \geq 1$ on $Z \setminus K_Z$ and $\del K_Z = 
\{\rho =1\}$. 

Fix a basepoint $p \in Z$, and consider the family of rescaled pointed spaces $(Z, \epsilon^2 g_Z,p)$. This converges in the 
pointed Gromov-Hausdorff metric to the metric cone $(\RR^+\times F, d\rho^2 + \rho^2 h_F, 0)$; the convergence is $\calC^\infty$
away from the vertex of this cone. 

\subsubsection{The family $\sD_1$} 
Fix an element $(Y_0,h_0) \in \sI_1$; thus $Y_0$ has either an isolated conic singularity or a 
simple edge singularity, i.e.\ a singular stratum $S$ which is a closed manifold with a 
cone-bundle neighbourhood where the link $F$ of the conic fibre is smooth. For simplicity of
exposition, assume that $Y_0$ has an edge.  Suppose also that the link
$F$ is the boundary of a smooth compact manifold $K$.
From this data one can define a smooth desingularization 
$Y$ of $Y_0$ by replacing each truncated conical fibre $C_{0,1}(F)$ in the cone bundle over $S$ with 
the smooth manifold $K$. 

We now exhibit a metric version of this desingularization which produces a family
of Riemannian spaces $(Y,h_\e)$ which converges, both in the Gromov-Hausdorff sense and 
smoothly away from a certain locus, to the original singular space $(Y_0,h_0)$.  

As in Definition~\ref{defn:I_k}, let $Y_0''$ denote the space where the truncated cones $C_{0,1}(F)$ are removed from each fibre 
over $S$; suppose that the metric induced on the boundary $F$ of each truncated cone is equal to some fixed metric $h_F$. 
Next, assume that the manifold with boundary $K$ is extended to a smooth noncompact AC space $(Z, g_Z)$ carrying a metric
$g_Z = d\rho^2 + \rho^2 h_F + \calO(\rho^{-1})$ on the complement of $K$; hence $(Z,g_Z) \in \sQ_0$.
Let $Z_{1/\e}$ denote the region $\{\rho \leq 1/\e\} \subset Z$, and consider the rescaling 
$(Z_{1/\e}, \e^2 g_Z)$. This space converges to the truncated cone $C_{0,1}(F)$ with conical metric $d\rho^2 +\rho^2 h_F$ as 
$\e \searrow 0$. Using a partition of unity, we may `cap off' each conical fibre in $Y_0$ by replacing $C_{0,1}(F)$ with
this rescaled space.  This defines a family of Riemannian metrics $h_\e$ on the smooth compact manifold $Y$ 
which has the properties stated. Thus altogether, $Y = Y' \cup Y''$ where $Y''$ is identified with the smooth
manifold with boundary $Y_0'' \subset Y_0$ and where $Y'$ is a bundle over $S$ with fibre $K = K_Z \equiv Z_{1/\e}$. 

We shall call this desingularization procedure, and its generalizations below, the {\it resolution blowup} of $Y_0$. 

\subsubsection{From $\sD_k$ to $\sQ_k$}
We can now describe one of the two steps in the general inductive procedure defining QAC spaces. Namely, 
we show how to construct a QAC space $(Z,g_Z)\in \sQ_k$ out of the following data: an element  $(Y,h_\e) \in \sD_k$
and a compact manifold $K$ with $\del K = Y$. (By assumption, then, $Y$ is nullbordant, which
limits the elements of $\sD_k$ to which this can be applied.) Also, denote by $(Y_0,h_0) \in \sI_k$ the
geometric limit of the spaces $(Y,h_\e)$ as $\e \to 0$.

The exterior region $Z^{\mathrm{ext}}$ is defined to be the product $[1,\infty) \times Y$ endowed with the metric
\[
\left. g_Z\right|_{Z^{\mathrm{ext}}} =  d\rho^2 + \rho^2 h_{1/\rho} + g',
\]
where $g' = \calO(\rho^{-1})$. The entire manifold $Z$ is obtained as the union $Z^{\mathrm{ext}} \sqcup K$ identified 
over the common boundary $Y$ of these two pieces; the metric is extended from this exterior 
region arbitrarily over $K$. Note that the interior regions is $K_Z =K$. By definition, any such space is an element of 
$\sQ_k$, and conversely, any element of $\sQ_k$ can be constructed in this way.

We check that $(Z,g_Z)$ has the geometric properties stated in the introduction. Consider the rescaled metrics 
$\e^2 g_Z$. Setting $r = \e \rho$, we see that on the region $\{r \geq \e\}$, 
\[
  \e^2 g_Z = dr^2 + r^2 h_{\e/r},
\]
and this converges smoothly on any region $\{r \geq c > 0\}$ to the conic metric $dr^2 + r^2 h_0$ on $(0,+\infty) \times Y_0$.

\subsubsection{Resolution blowups: desingularization of $Y_0 \in \sI_k$ using QAC spaces
of lower depth} \label{ssec:Dk}
The other half of the inductive step involves the desingularization of certain elements $(Y_0,h_0) \in \sI_k$ using QAC 
spaces of lower depth. The simplest case of this construction, the case $k=1$,
was described above, and that case should provide a good intuitive guide to the general case.

Fix an element $(Y_0,h_0) \in \sI_k$ such that the links of the conic fibres over each stratum are nullbordant. 
We wish to define the resolution blowup $(Y,h_{\e})$ using an appropriate collection of 
QAC spaces $(Z^{\d{j}},g^{\d{j}})\in\sQ_j$, $j<k$. As in
Definition~\ref{defn:I_k}, we have a decomposition $Y_0= Y_0'\cup
Y_0''$ with $Y_0'$ the total space of a smooth cone bundle over $S_k$
the stratum of maximal depth $k$ in $Y_0$. The fiber of this bundle is
$C_{0,1}(Y_0^{\d{k-1}})$, with $(Y_0^{\d{k-1}}, h_0^{\d{k-1}}) \in \sI_{k-1}$.
As explained in Section~\ref{ssec:tb} there is a natural doubling $2Y_0$
of $Y_0''$ across its 
codimension one boundary $W_0=\partial Y_0'' = \partial Y_0'$; this resulting doubled space $2Y_0$ is an 
element of $\sI_{k-1}$ and has an involution which fixes $W_0$ and interchanges the two copies of
$Y_0''$. Note also that $(W_0,\left.h_0 \right|_{W_0}) \in \sI_{k-1}$.

By induction, assume that we have chosen a resolution blowup $2Y$ of $2Y_0$.  This
can be done equivariantly, i.e.\ so that $2Y$ has an involution which is a lift of the involution of 
$2Y_0$. Let $Y''$ be `one half' of this space;  this is naturally identified as a resolution 
blowup of $Y_0'$ along all of its singular strata {\it except} the
codimension one boundary $W_0$. Thus 
$Y''$ is a smooth manifold with boundary $W$. This entire construction is accompanied by 
the existence of a family of metrics $h_\e$. If $\left.h_\e\right|_{W}$ denotes the restriction of this family of metrics 
to $W$, then $(W, \left.h_\e\right|_{W}) \in \sD_{k-1}$ which is a
resolution blowup of $(W_0, \left.h_0 \right|_{W_0})\in \sI_{k-1}$. 

The manifold $W$ is the total space of a bundle over $S_k$ with fibre $Y^{\d{k-1}}$ and family of metrics $h^{\d{k-1}}_\e$, 
so $(Y^{\d{k-1}}, h^{\d{k-1}}_\e)$ is a resolution blowup of
$(Y^{\d{k-1}}_0,h^{\d{k-1}}_0)$. Assuming that $Y^{\d{k-1}}$ is
nullcobordant, $Y^{\d{k-1}} = \partial K_{Z^{\d{k-1}}}$, where
$K_{Z^{\d{k-1}}}$ is the compact truncation of an element
$(Z^{\d{k-1}}, g^{\d{k-1}})\in \sQ_{k-1}$, and also that the
induced metric on $\partial K_{Z^{\d{k-1}}}$ is compatible with
$h^{\d{k-1}}$, we can attach to each $Y^{\d{k-1}}$ a copy of
$K_{Z^{\d{k-1}}}$ to obtain a smooth manifold $Y'$. This is the total
space of a bundle over $S_k$ with fiber $K_{Z^{\d{k-1}}}$, and
$\partial Y' = \partial Y'' = W$. Gluing $Y'$ and $Y''$ along their
common boundary $W$, we obtain a smooth compact manifold $Y$.

To complete the description of $Y$ as a resolution blowup of $Y_0$ we must define a degenerating family
of metrics on it. From the above, it is clear how the family of  metrics is
defined on $Y''$; they are the metrics $\left.h_{\e}\right|_{Y''}$. It
remains to define them on $Y'$. To this end, first choose the QAC
metric $g^{\d{k-1}}$ on $Z^{\d{k-1}}$ equal to $d\rho^2 + \rho^2 h^{\d{k-1}}_{1/\rho}$ on the 
exterior region $(Z^{\d{k-1}})^{\mathrm{ext}} = [1,\infty) \times Y^{\d{k-1}}$ and extended in an arbitrary 
manner over $K_{Z^{\d{k-1}}}$. Then take the truncated region $Z_{1/\e} = Z^{\d{k-1}} \cap \{\rho \leq 1/\e\}$
with metric $\e^2 g_Z$ and attach it to each of the fibres in the
boundary $W = \del Y'$, to obtain the metric $h_{\e}$ on $Y'$. 

It is elementary to verify that if $(Y, h_\e) \in \sD_k$, then its diameter is uniformly bounded for $0 < \e \leq 1$. 
This is needed later in the proof of Proposition~\ref{corintballest}.

This completes the construction of the resolution blowup $(Y,h_\e)$ of the iterated edge space
$(Y_0,h_0) \in \sI_k$, and hence the entire inductive definition of the two classes of spaces, $\sD_k$ and $\sQ_k$. 

\subsubsection{The product of two AC manifolds is QAC of depth $1$}\label{product}
Although it is not strictly needed, we now show that if $Z_1, Z_2 \in \sQ_0$, then the
product $Z_1 \times Z_2$ lies in $\sQ_1$. This example provides good intuition for the
overall construction of QAC spaces, and it is not hard to extend the discussion below
to show that the product of $Z_1 \in \sQ_{k_1}$ and $Z_2 \in \sQ_{k_2}$ lies in $\sQ_{k_1+k_2+1}$. 

To begin, consider two compact Riemannian manifolds $(F_1,\kappa_1)$ and $(F_2, \kappa_2)$
and the cones over them, $(C(F_j), d\rho_j^2 + \rho_j^2 \kappa_j)$. Define the spherical suspension 
$Y_0$ over $F_1\times F_2$ as the product $[0, \pi/2] \times F_1 \times F_2]$ with $F_1$ collapsed
to a point at the right end of this interval and $F_2$ collapsed to a point at the left. This is accomplished
by endowing this product with the metric 
\[
h_0 = d\theta^2 + \cos^2\theta\ \kappa_1 + \sin^2 \theta\ \kappa_2.
\]
We claim that $(Y_0,h_0) \in \sI_1$.   Indeed, near $\theta = 0$, 
\[
h_0 = d\theta^2 + \theta^2\kappa_2 + \kappa_1 + \eta,
\]
where $\eta = (\sin^2 \theta - \theta^2) \kappa_2 + (\cos^2 \theta - 1)\kappa_1 = \sO(\theta^2)$. 
The geometry near $\theta = \pi/2$ is identical. 

Now observe that if we set $\rho^2 = \rho_1^2 + \rho_2^2$, then
\[
d\rho_1^2 + \rho_1^2 \kappa_1 + d\rho_2^2 + \rho_2^2 \kappa_2   = d\rho^2 + \rho^2 h_0,
 \]
so $C(F_1) \times C(F_2) = C(Y_0)$. 

Now suppose that $(Z_1, g_1)$ and $(Z_2, g_2)$ are AC, and modelled at infinity by $C(F_1)$ and $C(F_2)$, 
respectively. 
\begin{proposition}
$(Z_1\times Z_2, g_1 \oplus g_2) \in \sQ_1$.
\end{proposition}
\begin{proof}
Write $Z = Z_1 \times Z_2$. By assumption, $Z_i = K_{i} \sqcup C_{1, \infty} (F_i)$, and 
 \[
\left.g_i\right|_{C_{1,\infty}(F_i)} = d\rho_i^2 + \rho_i^2 \kappa_i + \tilde{g}_i, \qquad \|\tilde{g}_i\| = \sO(\rho_i^{-\nu_i})
\]
for some $\nu_i > 0$ (the norm of $\tilde{g}_i$ being taken with respect to the exact conic metric). It suffices
for our purposes to assume that $\tilde{g}_i =0$.  Recall too that the radial functions are extended to $K_i$ 
so that $\rho_i \geq 1/2$ everywhere.  As before, write $\rho^2 = \rho_1^2 + \rho_2^2$.

Consider the following resolution blowup $(Y,h_{\e})\in \sD_1$ of $(Y_0,h_0)$: replace the neighborhood
$F_1\times C_{0,c}(F_2)$ of the edge $\theta =0$ by the (truncated) product $F_1 \times (Z_2)_{c/\e}$, with
a similar replacement near $\theta = \frac{\pi}{2}$. The constant $c$ is chosen small. These replacements
`round out' the ends of the spherical suspension. This has the degenerating family of metrics 
\[
h_\e : =
\begin{cases}
     \kappa_1 + \e^2 g_2 +\tilde\eta_1 & \text{on}\ F_1\times (Z_2)_{c/\e}\\
     h_0 & \text{on}\ \{ c \leq \theta \leq \pi/2 - c\} \\
     \kappa_2 + \e^2 g_1 + \tilde\eta_2& \text{on}\ (Z_1)_{c/\e} \times F_2, 
\end{cases}
\]
where, as above, the error terms $\tilde{\eta}_i$ decay to order $2$ in $\epsilon$. 
Since the $g_i$ are exact cones outside $K_{i}$, $h_{\e}$ is smooth on $Y$. 

If $g$ is a smooth metric on $Z_1 \times Z_2$ such that 
\[
g= d\rho^2 + \rho^2 h_{1/\rho}\quad \mbox{when}\quad \rho = \rho_1^2 + \rho_2^2 \geq 1,
 \]
then our main assertion is equivalent to the statement that $g' = g - (g_1 \oplus g_2)$ has suitable decay. 
We decompose the end $\{\rho \geq 2\}$ of $Z$ into three regions: $\rho_1/\rho_2 \leq c'$, $c' \leq \rho_1/\rho_2 \leq 1/c'$
and $\rho_1/\rho_2 \geq 1/c'$, where $\tan c' = c$, and analyze $g'$ on each of these.  

On the middle region, $h_{1/\rho} = h_0$ is fixed and $g$ coincides exactly with the product metric $g_1 \oplus g_2$.
Clearly the analysis for the two outer regions is the same, so assume that $\rho_2/\rho_1 \leq c'$. On this region,
$g = d\rho^2 + \rho^2 \kappa_1+ g_2 + \rho^2 \tilde{\eta}_1$, so the issue is to estimate
\[
(d\rho^2 + \rho^2 \kappa_1) - (d\rho_1^2 + \rho_1^2 \kappa_1),
\]
which is straightforward. 
\end{proof}

\subsection{Decompositions and weight functions}\label{ssec:rbdf}
We now introduce a collection of `resolved radial functions' on any resolution blowup space, and a similar collection of resolved 
radial functions on any QAC space. More specifically, suppose that $(Y,h_\e) \in \sD_k$ is a resolution blowup of $(Y_0, h_0) \in \sI_k$.
We shall define families of functions $\{\sigma_{k,\e}, \ldots, \sigma_{1,\e}\}$ on $Y$, which are `smoothings' of the radial functions 
$s_j$ on $Y_0$ introduced at the end of \S 2.2. Similarly, if $(Z,g) \in \sQ_k$ is a QAC space with link $(Y,h_\e)$, 
then we introduce some closely related functions $\{\rho, w=( w_1,
\ldots, w_k)\}$ on $Z$. The $\{\sigma_{1,\e}, \ldots, \sigma_{k,\e}\}$ determine a 
geometric decomposition of $Y$, where the components of this decomposition are the resolution blowups of the corresponding
components in the decomposition of $Y_0$.  We use these decompositions in many of the constructions and proofs below.  
In addition, the functions $\rho$ and $w=(w_k, \ldots, w_1)$ on $Z$ are also used as weight factors in the function spaces below;
these are crucial for the proper formulation of our main Fredholm results. 

Just as in \S 2.2, we start gradually, first defining these decompositions and weight functions in the case $k=1$. To keep track of 
the various spaces below, we systematically include a parenthesized superscript index to indicate the depth of each 
of the spaces being considered. Thus, for example, $Z^{(k)}$ and $g^{(k)}$ denote a `generic' depth $k$ QAC space 
and a typical metric on it. This notation will be used, where appropriate, through the rest of the paper too. 

\medskip

\noindent{$\boxed{\sQ_0}$} 
First, suppose that $(Z^{(0)},g^{(0)})  \in \sQ_0$.  Let $\rho_0$ denote the radial function on the exterior region 
$(Z^{(0)})^{\mathrm{ext}} \cong C_{1,\infty}(Y^{(0)}) \cong [1,\infty) \times Y^{(0)}$, where $Y^{(0)}$ is a smooth compact manifold (i.e.\ 
an element of $\sD_0$). Extend $\rho_0$ to equal $1$ on the compact cap $K_{Z^{(0)}}$.  The functions $w_j$
do not appear in this lowest depth case. 

\medskip

\noindent{$\boxed{\sD_1}$} Next, fix $(Y^{(1)},h_\e^{(1)}) \in \sD_1$, with limiting space $(Y_0^{(1)},h_0^{(1)}) \in \sI_1$. 
Using the radial function $s_1$ on $Y_0^{(1)}$, we can decompose $Y_0^{(1)} = (Y_0^{(1)})' \cup (Y_0^{(1)})''$, 
corresponding to the regions where the radial function $s_1 < 1$ and $s_1 \geq 1$, respectively, just as 
in Definition~\ref{defn:I_k}.   Observe 
that $(Y^{(1)}_0)'$ is a bundle over the depth $1$ stratum $S_1$ with (truncated) conic fibres; 
in the resolution blowup, these conic fibres are replaced by truncations of some QAC space $Z^{(0)}$. This transforms
$(Y^{(1)}_0)'$ into something we shall call $(Y^{(1)})'$. The space $Y^{(1)}$ itself is assembled by gluing together 
$(Y^{(1)})'$ and $(Y_0^{(1)})''$ in the obvious way, which is possible since their boundaries are naturally
identified, and we then rename $(Y_0^{(1)})''$ as $(Y^{(1)})''$ when it is considered as a portion of $Y^{(1)}$. 

It remains to define the functions $\sigma_{1,\e}$ and metrics $h_\e$ on $Y^{(1)}$.  We first consider $\e\rho_0$, a priori defined on 
each fiber $Z^{(0)}$ of the fibration $(Y^{(1)})'\to Z^{(0)}$, as a function on $(Y^{(1)})'$. A metric on this neighborhood
is given by adding to the family of fiber metrics $\e^2 g^{(0)}$ the pullback of any metric from the stratum $S^1$. 
This results in a one-parameter a family of metrics $h_\e^{(1)}$ on this neighbourhood. Since $\e\rho_0 = 1$
at $\del (Y^{(1)})' $, i.e., along $(Y^{(1)})'  \cap (Y^{(1)})''$, we may extend it to equal $1$ on the second piece in
the decomposition $Y^{(1)} = (Y^{(1)})'  \cup (Y^{(1)})''$. Altogether, this defines $\sigma_{1,\e}$ by 
\[
\sigma_{1,\e}(y) = \begin{cases} 
    1 \quad & \mbox{if}\quad y \in (Y^{(1)})'', \\ 
   \e \rho_0(u)  & \mbox{if} \quad y \in (Y^{(1)})' \quad \mbox{projects onto $u\in Z^{\d{0}}$}.
\end{cases}
\]
The region $\{\sigma_{1,\e} = \e\}$ is the union of the compact components $K_{Z^{(0)}}$ in each fiber,
and we set $\rho_0 \equiv 1$ here. 

\medskip

\noindent{$\boxed{\sQ_1}$} If $(Z^{(1)},g^{(1)}) \in \sQ_1$ has link $(Y^{(1)},h_\e^{(1)}) \in \sD_1$, then $Z^{(1)}$ is the 
disjoint union of a compact cap and an exterior region,
\[
  Z^{(1)} = K_{Z^{(1)}} \sqcup (Z^{(1)})^{\mathrm{ext}}, \qquad  (Z^{(1)})^{\mathrm{ext}}\cong [1,\infty) \times Y^{(1)}. 
\]
The exterior region decomposes further into the two regions, $(Z^{(1)})' \cup (Z^{(1)})''$, where each
of these components is the product of the corresponding piece $(Y^{(1)})'$ or $(Y^{(1)})''$, with $[1,\infty)$.  
The natural radial function on this exterior region is
denoted $\rho_1$, which as before is extended to equal $1$ on $K_{Z^{(1)}}$. Note that $(Z^{(1)})''$ is 
simply the exterior cone over $(Y^{(1)})''$, with the standard conic metric,
\[
(Z^{(1)})''= C_{1,\infty} ( (Y^{(1)})''), \qquad \left. g^{{(1)}} \right|_{(Z^{(1)})''} = d\rho_1^2 + \rho_1^2 h_{1/\rho_1}^{(1)},
\]
(recall, however, that $h_\e^{(1)}$ is constant, or at last nondegenerating, over this portion of $Y^{(1)}$, so this is
effectively a conic metric on this piece).  
The region $(Z^{(1)})'$ is the cone over $Y^{(1)})'$. Recall that $(Y^{(1)})'$ is a bundle over the depth $1$ stratum $S_1$ 
whose fiber is the truncation of the QAC space $Z^{(0)}$ of depth $0$. In a metric trivialization of this bundle, the metric 
on $(Z^{(1)})'$ becomes 
\[
  g^{(1)}\lvert_{(Z^{(1)})'} = d\rho_1^2 + \rho_1^2 k_{S_1} + g^{(0)}
  \quad  \text{with} \quad  \rho_0 \leq \rho_1.
\]
Let $\theta = \frac{\rho_0}{\rho_1}$, so that $\frac{1}{\rho} \leq \theta \leq 1$. Writing
\[
  t_1= \rho_1 \cos\theta 
  \quad  \text{and} \quad
  s_1 = \rho_1 \sin \theta,
\]
we then see that this metric is quasi-isometric to the metric induced from the product metric on the `subdiagonal' 
of the product $C(S_1) \times Z^{(0)}$, i.e.\ 
\[ 
g^{(1)} \asymp dt_1^2 + t_1^2 k_{S_1} + g^{(0)} \quad \text{in the region} \quad  \rho_0 =1
\]
and 
\[
g^{(1)} \asymp dt_1^2 + t_1^2 k_{S_1}  + ds_1^2 + s_1^2 h^{(0)} \quad \text{in the region} \quad 1< \rho_0 \leq \rho_1.
\]
This is nothing else than the reverse of the observation in Section~\ref{product}
that the product of two AC spaces is a QAC space of depth $1$.

The weight function $w_1$ is defined by
\[
  w_1(z) = 
  \begin{cases}
    1 \quad & \text{if}\quad z \in  K_{Z^{(1)}} \cup (Z^{(1)})'' \\ 
    \sigma_{1,1/\rho_1}(y) & \mbox{if}\quad  z = (\rho_1,y) \in (Z^{(1)})' = [1,+\infty) \times (Y^{(1)})'.
  \end{cases}
\]

\noindent{$\boxed{\sD_k}$} We now turn to the general case.  Suppose that we have defined all decompositions 
and weight functions for all spaces $Z^{(j)}$ and $Y^{(j)}$ of depth $j < k$. If $Y_0^{(k)} \in \sI_k$, then the region
$(Y^{(k)}_0)'' = \{y \in Y^{(k)}_0: s_k(y) = 1\}$ contains only singularities of depth $k-1$ or lower,
so we may assume (using the doubling construction described above) that the resolution blowup of 
this region has been chosen. We explain how to use the lower
depth smoothed radial functions $\sigma_{j,\e}^{(\ell)}$, $1\leq j, \ell \leq k-1$, to define $\sigma^{(k)}_{j,\e}$, $1\leq j \leq k-1$
here, and then define $(Y^{(k)})''$, the metric $ h^{(k)}_\e$ as well as the extensions of the functions $\sigma_{j,\e}^{(k)}$, 
$1\leq j < k$, and finally define $\sigma_{k,\e}^{(k)}$ on this new region. 

To do all of this, note that $\del (Y^{(k)})''$ is a space $Y^{(k-1)} \in \sD_{k-1}$ and (by implicit hypothesis) we can choose 
some $Z^{(k-1)} \in \sQ_{k-1}$ whose truncations cap off (the fibres of) this end. (We also assume that the metrics agree.)
In other words, we define $(Y^{(k)})'$ to be the bundle over $S_k$ with fibres $Z^{(k-1)}_{1/\e} = \{z \in Z^{(k-1)}: \rho_{k-1}
\leq 1/\e\}$ and the metrics $h^{(k)}_\e$ here to be the sum of the pullback of a metric on $S_k$ and 
$\e^2 g^{(k-1)}$. Hence the metric compatibility is that $g^{(k-1)} = d\rho_{k-1}^2 + \rho_{k-1}^2 h^{(k-1)}_{1/\rho_{k-1}}$, where
$h^{(k-1)}_\e$ is the degenerating metric family on the fibres of $\del (Y^{(k)})''$. 

First, note that the functions $\sigma_{j,\e}^{(\ell)}$, $j, \ell < k$, can be chosen by induction on the double of
$(Y^{(k)})''$ and then restricted back to this subset.  We define
\[
\sigma^{(k)}_{k,\e}(y) = \begin{cases}
1 \quad & \mbox{if}\ y\in (Y^{(k)})'', \\
\e \rho_{k-1}(u) & \mbox{if}\ y \in (Y^{(k)})' \quad \text{projects onto}\ u \in Z^{(k-1)},
\end{cases}
\]
and
\[
\sigma^{(k)}_{j,\e}(y) = \begin{cases}
\sigma^{(k-1)}_{j,\e}(y) \quad & \mbox{if}\ y \in (Y^{(k)})'', \\
\sigma^{(k)}_{k,\e}(y) w^{(k-1)}_{j}(u) & \mbox{if}\ y \in (Y^{(k)})' \quad \text{projects onto}\ u \in Z^{(k-1)},
\end{cases} 
\]
for all $1\leq j < k$.
We leave the reader to check that these definitions give a function that is continuous (or at least has the
right orders of magnitude) at the intersections of the domains of definition. 

From this construction it is also immediate that each $\sigma^{(k)}_{j,\e}$ converges as $\e \to 0$ to the radial 
function $s_j = s_j^{(k)}$ on $Y_0$. 

\medskip

\noindent{$\boxed{\sQ_k}$} Finally, to complete this inductive definition, we describe the decomposition of $Z^{(k)}$, and the definition of $ g^{(k)}$
and the functions $\rho_k$, $w^{(k)}_j$, $1\leq j \leq k$. As in the depth $1$ case, we choose $(Y^{(k)}, h^{(k)}_\e) \in \sD_k$
and a compact cap $K_{Z^{(k)}}$ for $Y^{(k)}$. Then
\[
Z^{(k)} = K_{Z^{(k)}} \sqcup (Z^{(k)})^{\mathrm{ext}}, 
\]
\[
(Z^{(k)})^{\mathrm{ext}} = [1,\infty) \times Y^{(k)} = (Z^{(k)})' \cup (Z^{(k)})'',
\]
where $(Z^{(k)})'' = C_{1,\infty}((Y^{(k)})'')$ with metric $d\rho_k^2 + \rho_k^2 h^{(k)}_{1/\rho_k}$, which is 
QAC of depth $k-1$ since $h^{(k)}_{\e}$ is a resolution blowup metric of depth $k-1$ on $(Y^{(k)})''$.
Similarly, 
\begin{multline*}
 (Z^{(k)})' = \{ z= (\rho_k,y) \in [1, + \infty) \times (Y^{(k)})': 1 \leq \rho_{k-1} (q) \leq \rho_k \\
\mbox{where}\ q \ \text{is the projection of $y$ onto $Z^{(k-1)}$} \}. 
\end{multline*}
The metric here is given by the same formula, $d\rho_k^2 + \rho_k^2 h^{(k)}_{1/\rho_k}$. 

The radial function $\rho_k$ is the obvious one.
Using the functions $\sigma^{(k)}_{j,\e}$, $1\leq j \leq k$, on $Y^{(k)}$, we set
\[
w^{(k)}_j(z) = \begin{cases}
1 \quad & \mbox{if}\ y \in K_{Z^{(k)}} \\
\sigma^{(k)}_{j,1/\rho_k}(y) & \mbox{if}\ z= (\rho_k, y)\in (Z^{(k)})^{\mathrm{ext}}.
\end{cases}
\]

\subsubsection{The decompositions of $Y^{(k)}$ and $Z^{(k)}$}\label{sec:decomp}
We now define decompositions for each $Y^{(k)} \in \sD_k$ and $Z^{(k)} \in \sQ_k$ using the
families of functions defined above. 

First, using the functions $\{\sigma^{\d{k}}_{1,\e}, \ldots, \sigma^{\d{k}}_{k,\e}\}$, set 
\[
Y^{\d{k}}_{\d{j}} : = \{ y \in Y^{\d{k}}: \sigma^{\d{k}}_{j,\e}(y) < 1 \quad
\text{and} \quad \sigma^{\d{k}}_{i,\e}(y) \geq 1 \quad \text{for all}\ i \geq j+1\}. 
\]
There is an analogous decomposition of $Y^{\d{k}}_0$ using the radial functions $s_j$. 
Its `principal component' $(Y^{\d{k}}_0)_{\d{0}}$ is identified with $Y^{\d{k}}_{\d{0}}$.
The pieces here are disjoint and  $Y^{\d{k}}= \bigsqcup_{j=0}^k Y^{\d{k}}_{\d{j}}$. We later define an open cover by thickening each piece.

Next, in terms of the functions $\{\rho_k, w^{\d{k}}= (w^{\d{k}}_1,
\ldots, w^{\d{k}}_k)\}$ on $(Z^{\d{k}}, g^{\d{k}}) \in
\sQ_k$,  we write
\begin{equation}\label{eq:Zkj}
  Z^{\d{k}}_{\d{j}}: = \{z\in Z^{\d{k}} : w^{\d{k}}_j(z) <1 \quad
  \text{and} \quad w^{\d{k}}_i (z) =1 \quad \text{for all}\ i \geq j+1\}. 
\end{equation}
Note that each $Z^{\d{k}}_{\d{k}}$ is noncompact, and moreover, $K_{Z^{\d{k}}} \subset Z^{\d{k}}_{\d{0}}$.
This piece can be viewed as an AC space. On the other hand, if $j \geq 1$, then 
$Z^{\d{k}}_{\d{j}}$ is a cone over $Y^{\d{k}}_{\d{j}}$, and hence is the total space of a bundle 
over $C(S_j)$ with each fiber a truncated QAC space $(Z^{\d{j-1}}, g^{\d{j-1}}) \in \sQ_{j-1}$.

In terms of the functions $\{\rho_{j-1}, w^{\d{j-1}}=(w^{\d{j-1}}_{1},
\ldots w^{\d{j-1}}_{j-1})\}$ on each $Z^{\d{j-1}}$, the functions
$w^{\d{k}}=(w^{\d{k}}_1,\ldots, w^{\d{k}}_{k})$ on this piece are 
\begin{align}\label{eq:w-Zkj}
  w^{\d{k}}_i(z) & = 1 \quad & \text{for $i= j+1,\ldots, k$,}\nonumber\\
  w^{\d{k}}_j(z) & =  \frac{\rho_{j-1}(z_{j-1})}{\rho_k(x)}, &\\[1em]
  w^{\d{k}}_{i}(z) &= w^{\d{k}}_j(z)w^{\d{j-1}}_{i}(z_{j-1}) \quad
&  \text{for $i = 1, \ldots, j-1$}\nonumber
\end{align}
where $z\in Z^{\d{k}}_{\d{j}}$ corresponds to the pair $(q_{j}, z_{j-1})$ with $q_j \in C(S_j)$ and $z_{j-1} \in Z^{\d{j-1}}$.

\begin{lemma}
If $z \in Z^{\d{k}}$ and $w^{\d{k}}_\ell(z) = 1$ for some $\ell \geq 1$, then $w^{\d{k}}_i (z)  =1$ 
for every $i \geq \ell$. 
\end{lemma}
\begin{proof}
If, $w^{\d{k}}_j (z) <1$ for some $j > \ell$, while $w_{i}^{\d{k}}(z) = 1$ for $i \geq j+1$, 
then $z \in Z^{\d{k}}_{\d{j}}$. 
However, \eqref{eq:w-Zkj} implies that $w^{\d{k}}_\ell(z) <1$ instead, which is a contradiction.
\end{proof}
This shows that we have the disjoint decomposition $  Z = \bigsqcup_{j=0}^{k} Z^{\d{k}}_{\d{j}}$.
Many arguments in this work are inductive, and this decomposition is very well suited for these.
Indeed, in the decomposition of a space $Z^{(k)}$, the complement of $Z^{(k)}_{(k)}$ contains the 
portion of $Z^{(k)}$ which has depth no greater than $k-1$, and hence whatever result we are trying
to establish here holds using the induction hypothesis. It remains then only to show that
the result extends to $Z^{\d{k}}_{\d{k}}$. 
There are, however, instances where we also need to examine all the pieces
$Z^{\d{k}}_{\d{j}}$ in turn (see, for example, the estimates for the Green's function in Section~\ref{sec:Green}).

\subsubsection{The thickened decomposition}\label{sec:decomp-t}
It is useful to enlarge the pieces of $Z^{\d{k}}= \bigsqcup_{j=0}^{k} Z^{\d{k}}_{\d{j}}$ slightly.
Fix a constant $0< \eta < 1$ and define the {\em $\eta$-thickened pieces} of $Z^{\d{k}}$ by
\begin{equation}\label{eq:Zkj-t}
  Z^{\d{k}}_{\d{j}}(\eta): = \{z\in Z^{\d{k}} : w^{\d{k}}_j(z) <1 \
\text{and} \ w^{\d{k}}_i (z) > 1-\eta \ \ \text{for all}\ i \geq j+1\}. 
\end{equation}
Note that 
\[
Z_{\d{k}}^{\d{k}} (\eta) = Z_{\d{k}}^{\d{k}} \quad
\mbox{and}\quad Z^{\d{k}}_{\d{j}} \subset Z^{\d{k}}_{\d{j}}(\eta) \subset  \union_{l=j}^kZ^{\d{k}}_{\d{l}}.
\]
Moreover, for $j \geq 1$, $Z^{\d{k}}_{\d{j}}(\eta)$ is the total space of a bundle over $S_j$ with fiber 
$(Z^{\d{j-1}}, g^{\d{j-1}})\in \sQ_{j-1}$; the metric $g^{\d{k}}$ is locally quasi-isometric to a
product metric on $C(S_j) \times Z^{\d{j-1}}$. Note finally that the
weight functions $w^{\d{k}}=(w_1^{\d{k}}, \ldots, w_k^{\d{k}})$ on this piece
can be approximated by 
\begin{align}\label{eq:w-Zkj-t}
  w^{\d{k}}_i(z) & \asymp 1 \quad& i\geq j+1\nonumber\\
  w^{\d{k}}_j(z) & \asymp \frac{\rho_{j-1}(z_{j-1})}{\rho_k(z)},&\\[1em]
  w^{\d{k}}_{i}(z) &\asymp  w^{\d{k}}_j(z)w^{\d{j-1}}_{i}(z_{j-1}) \quad& i \leq j-1.\nonumber
\end{align}
Here we use the notation $f_1 \asymp f_2$ for two functions for which
there exists positive constants $c, C>0$ so that $c f_2 \leq f_1 \leq C f_2$
holds everywhere. We also say $f_1\preceq f_2$ if only the right
hand-side of the inequality is satisfied (see Table~\ref{table} at the end of the paper for
notational conventions).

From the above, it is clear that on a QAC manifold, the functions
$w^{\d{k}}= (w^{\d{k}}_1, \ldots, w^{\d{k}}_k)$ satisfy
\begin{equation}
\label{eq:w}
 w^{\d{k}}_1 \preceq \ldots \preceq  w^{\d{k}}_k.
\end{equation}

\subsection{Metrics and operators on QAC spaces}\label{sec:structure}
We now use the thickened decomposition of $(Z^{\d{k}}, g^{\d{k}}) \in \sQ_{k}$ to give an alternate 
description of the metric $g^{(k)}$ up to quasi-isometry on each piece $Z^{\d{k}}_{\d{j}}(\eta)$. 
The elliptic operators we study here are the `geometric' operators naturally associated
to these metrics.  We shall be studying the actions of these operators on weighted
Sobolev and H\"older spaces, and the eventual goal of this subsection is to study
the action of these operators on the weight functions themselves. 

As a first step, we record the
\begin{lemma}\label{lem:m}
If $(Z^{\d{k}},g^{\d{k}}) \in \sQ_k$ and $0 < \eta < 1$, then 
\begin{equation*}
\begin{aligned}
\left. g \right|_{Z^{\d{k}}_{\d{j}}(\eta)} \ \ \text{is quasi-isometric to}\ 
\end{aligned}
\left\{
\begin{aligned}
& d\rho_k^2 + \rho_k^2 h^{\d{k}}_0  & \text{if}\ j =0,\\
 & d\rho_k^2 + \rho_k^2 \kappa_{S_{j}} + g^{\d{j-1}} & \text{if}\ j \geq 1.
 \end{aligned}
\right.
\end{equation*}
The second identification holds on open sets in $Z^{\d{k}}_{\d{j}}(\eta)$ which are represented
as a product $C_{0,1}(S_j) \times Z^{\d{j-1}}$.
\end{lemma}
The proof is straightforward using the original description of the
metric. 


Associated to any $(Z^{(k)}, g^{(k)}) \in  \calQ_k$ are the class of geometrically natural elliptic operators.
If $\rho\colon \SO(n) \to V$ is any finite dimensional representation, $n = \dim Z$, then there is an induced
vector bundle $E \to Z$, $E = \calF(Z) \times_\rho V$, where $\calF(Z)$ is the bundle of oriented
orthonormal frames. This bundle inherits a connection $\nabla$ from the Levi-Civita connection of $g$, 
and we may then form the rough (or Bochner) Laplacian $\Delta = \nabla^* \nabla$. 

Any operator of the form
\[
\sL = \nabla^*\nabla + \sR, \qquad \sR \in \operatorname{End}(E),
\]
will be called a generalized Laplacian. We always assume that $\sR$ is bounded from below, $\sR \geq -C \mbox{Id}$. 
These are our main objects of study. 

It is equally straightforward to define a general class of first-order Dirac-type operators, but for 
simplicity, we focus exclusively on the second-order case. 

\subsubsection{Action on weight functions}\label{sec:nw} 
Fix $(Z^{\d{k}}, g^{\d{k}})\in \sQ_{k}$, and let $\{\rho_k, w^{\d{k}}=(w^{\d{k}}_{1}, \ldots, w^{\d{k}}_{k})\}$ be
the functions defined above. For any $a \in \RR$ and $b= (b_1, \ldots, b_k) \in \RR^k$, define the 
{\em weight functions}
\[
\rho_k(z)^a w^{\d{k}}(z)^{b} : = \rho_k(z)^a w^{\d{k}}_1(z)^{b_1} \ldots w^{\d{k}}_k(z)^{b_k}, \ z \in Z^{\d{k}}. 
\]
For $1\leq \ell \leq k$, write 
\[
b(\ell):= (b_1, \ldots, b_\ell) \in \RR^{\ell},
\] 
and note that $b(k) = b$. We also define 
\[
 |b(\ell)|: = b_1 + \ldots + b_\ell, 
\]
and make the convention that $b(0) = 0$ and $|b(0)|=0$. See 
Table~\ref{table} at the end of the paper for further notation. 

The following is now easy to deduce from~\eqref{eq:w-Zkj-t}:
\begin{equation}\label{eq:wf0}
\left. \rho_k(z)^a w^{\d{k}}(z)^{b} \right|_{Z^{\d{k}}_{\d{0}}(\eta)} \asymp \rho_k^{a}(z),
\end{equation}
and, if $j \geq 1$, 
\begin{equation}\label{eq:wfj}
\begin{split}
&\left. \rho_k(z)^a w^{(k)}(z)^b \right|_{Z^{(k)}_{(j)}(\eta)} \asymp \\
& \qquad \qquad \rho_{k}(z)^{a-|b(j)|} \rho_{j-1}(z_{j-1})^{|b(j)|} w^{\d{j-1}}(z_{j-1})^{b(j-1)}. 
\end{split}
\end{equation}
In this second expression we adopt a notation that will be used frequently below and, recalling
the local product representation $Z^{(k)}_{(j)} = C_{0,1}(S_j) \times Z^{\d{j-1}}$, write 
$z\in Z^{\d{k}}_{\d{j}}(\eta)$ as $(q_{j}, z_{j-1})$. 

\begin{proposition}\label{lem:Lap}
For $(Z^{(k)}, g^{(k)}) \in \sQ_k$, $a \in \RR$, $b \in \RR^k$ and $z \in Z^{\d{k}}$,  we have
\[
\begin{split}
&\left| \nabla  (\rho_k(z)^a w^{\d{k}}(z)^{b}) \right| \preceq \rho_k(z)^{a-1}  w^{\d{k}}(z)^{b-\uone} \quad \mbox{and}\\[1em]
&\left| \Delta (\rho_k(z)^a w^{\d{k}}(z)^{b}) \right|\preceq \rho_k(z)^{a-2}  w^{\d{k}}(z)^{b-\utwo}.
\end{split}
\]
\end{proposition}
We are using here the notation introduced in Table~\ref{table} that 
$ w^{\d{k}}(z)^{b-\underline{r}} =  w^{\d{k}}(z)^{b}/
w^{\d{k}}_{1}(z)^{r}$, $r = 1, 2$.  Moreover, we used the notation
$f_1\preceq f_2$ for two functions $f_1$ and $f_2$ for which there
exists a positive constant $C>0$ so that $f_1 \leq C f_2$ everywhere.

\begin{proof}
We prove this by induction on the depth. When $k=0$, $Z$ is an AC space and the assertion is clearly true.

Assume now the result holds for all QAC spaces of depth $\ell<k$, and let $(Z^{\d{k}}, g^{\d{k}})
\in \calQ_k$. Decompose $Z^{\d{k}}$ as the union $\bigcup Z^{\d{k}}_{\d{j}}(\eta)$ and express the weight
function using \eqref{eq:wfj} on each piece.  

Using the induction hypothesis, the assertion holds for $Z^{\d{k}}_{\d{j}}(\eta)$, $j \leq k-1$, 
so we must only check that it also holds on $Z^{\d{k}}_{\d{k}}(\eta)$. On this piece,
$g^{\d{k}}$ is quasi-isometric to $d\rho_k^2 + \rho_k^2 \kappa_{S_{k}} + g^{\d{k-1}}$ on 
$C(S_{k-1}) \times Z^{\d{k-1}}$, hence we must estimate 
\begin{equation}
\left| \nabla_{C(S_k) \times Z^{\d{k-1}}}  \left(\rho_{k}(z)^{a-|b(k)|} \rho_{k-1}(z_{k-1})^{|b(k)|}  w^{\d{k-1}}(z_{k-1})^{b(k-1)}\right) \right|
\label{nablaweight}
\end{equation}
and
\begin{equation}
(\Delta_{C(S_k)} + \Delta_{Z^{\d{k-1}}}) \left(\rho_{k}(z)^{a-|b(k)|} \rho_{k-1}(z_{k-1})^{|b(k)|}  w^{\d{k-1}}(z_{k-1})^{b(k-1)}\right).
\label{Lapweight}
\end{equation}
Note that $\nabla_{C(S_k)}$ and $\Delta_{C(S_k)}$ act only on $\rho_k(z)^{a-|b(k)|}$ while $\nabla_{Z^{\d{k-1}}}$ and $\Delta_{Z^{\d{k-1}}}$ act 
on the remaining factors. By induction, 
\begin{multline*}
\left|\nabla_{C(S_k)} \rho_{k}(z)^{a-|b(k)|} \right| \preceq C \rho_k(z)^{a-|b(k)|-1}, \quad \mbox{and} \\
\left|\Delta_{C(S_k)} \rho_{k}(z)^{a-|b(k)|} \right| \preceq C \rho_k(z)^{a-|b(k)|-2},
\end{multline*}
while 
\begin{multline*}
\left|\nabla_{Z^{\d{k-1}}}\rho_{k-1}(z_{k-1})^{|b(k)|}  w^{\d{k-1}}(z_{k-1})^{b(k-1)}\right|  \\\leq
C \left| \rho_{k-1}(z_{k-1})^{|b(k)|-1}  w^{\d{k-1}}(z_{k-1})^{b(k-1)-\uone}\right|.  
\end{multline*} 
Finally, by the same induction
\begin{multline*}
\big\lvert \Delta_{Z^{\d{k-1}}} \rho_{k-1}(z_{k-1})^{|b(k)|} w^{\d{k-1}}(z_{k-1})^{b(k-1)} \big\lvert  \\[.5em]
\leq \rho_{k-1}(z_{k-1})^{|b(k)|-2} w^{\d{k-1}}(z_{k-1})^{b(k-1)-\utwo}. 
\end{multline*}
Thus \eqref{Lapweight} is estimated by
\begin{multline*}
C \left( \rho_k(z)^{a-|b(k)|-2} \rho_{k-1}(z_{k-1})^{|b(k)|} w^{(k-1)}(z_{k-1})^{b(k-1)} \right. \\
\left. + \rho_k(z)^{a-|b(k)|} \rho_{k-1}(z_{k-1})^{|b(k)|-2} w^{(k-1)}(z_{k-1})^{b(k-1)-\utwo}\right)  \\
=  C\left(\rho_k(z)^{a-2} w^{(k)}(z)^b + \rho_k(z)^{a-2} w^{(k)}(z)^{b-\utwo}\right).
\end{multline*}
Since $w^{(k)}_1(z) \leq 1$, the last term dominates.
\end{proof}

\subsubsection{Function spaces}
To simplify notation here, let us omit the superscripts $(k)$ which indicate depth. We wish
to describe the various function spaces considered in this paper. We shall define a scale of weighted
$L^2$-based Sobolev and weighted H\"older spaces. The main results in this paper are that
generalized Laplacians are Fredholm between these spaces for certain ranges of the weight parameters. 

Fix $(Z,g)\in \sQ_k$, with weight functions $\{\rho, w=(w_1, \ldots, w_k)\}$ as defined above. 
First, for any $\delta \in \RR$ and $\tau \in \RR^k$, define 
\[
\rho^{\delta+\frac{n}{2}}w^{\tau + \frac{\nu}{2}}L^2(Z, dV_g) = \{ f :
\rho^{-\delta-\frac{n}{2}}w^{-\tau - \frac{\nu}{2}}f\in  L^2(Z, dV_g)\}.
\]
Here $\nu$ is the $k$-tuple $(\nu_1, \ldots, \nu_k)$ whose entries
depend only on the dimensions $m_j$ of the lower QAC spaces in $Z$,
and are defined in the following manner:
\begin{equation}
  \nu_1 = m_0, \quad \nu_j = m_{j-1}-m_{j-2} \quad \text{for all
    $j\geq 2$}.
\label{shiftparam}
\end{equation}
It has the property that $|\nu(j)| = \nu_1 + \ldots + \nu_j$ is
exactly $m_{j-1}$.
The weight factors are shifted by $n/2$ and $\nu/2$, respectively, to compensate for
the weight factors in the measure $dV_g$. More specifically, on an AC space with cross-section
$(Y,h)$, $dV_g = \rho^{n-1} d\rho\, dA_h$, so 
\[
\rho^{\delta'} \in \rho^{\delta + \frac{n}{2}} L^2(Z, dV_g)\quad
\text{if and only if} \quad  \delta' < \delta.
\]
When $Z$ is QAC, the effect of the $n/2$ in the $\nu/2$ factors is similar: it will follow from Proposition~\ref{lem:anch_est}
that $\rho^{\delta'}w^{\tau'} \in \rho^{\delta+\frac{n}{2}}w^{\tau + \frac{\nu}{2}}L^2(Z, dV_g)$ precisely when
\[
  \delta' < \delta
 \quad \text{and} \quad \delta' +|\tau'(j)| < \delta
  + |\tau(j)|
\]
for all $1\leq j \leq k$.
This allows for simpler bookkeeping and a more direct comparison with the weighted H\"older spaces; also our results can 
then be stated in a manner which are direct generalizations of the standard results for AC spaces (see for example the 
statement of Theorem~\ref{thm:Fredholm}).

We next define the weighted Sobolev spaces of order $s > 0$. When $s \in \mathbb N$, 
$\rho^{\delta + n/2}  w^{\tau + \nu/2} H^s(Z, dV_g)$ consists of functions (or sections) $f$ such that 
$\nabla^j f \in \rho^{\delta + n/2 -j} w^{\tau + \nu/2 -\underline{j}} L^2$ for $0 \leq j \leq s$.  (Recall here that 
the $\underline{j}$ in the exponent for $w$ means that we lower the power of $w_1$ by $j$, but leave the exponents 
of the other $w_i$ unchanged.) The spaces for $s \in \RR$, $s > 0$ are then defined by interpolation, while those for $s < 0$ 
are obtained by duality. 

It is straightforward to show, using Proposition~\ref{lem:Lap}, that 
\begin{equation}
\Delta\colon  \rho^{\delta+\frac{n}{2}}w^{\tau + \frac{\nu}{2}}H^{s+2}(Z, dV_g) \to  \rho^{\delta+\frac{n}{2}-2}w^{\tau+ \frac{\nu}{2}-\utwo}H^s(Z, dV_g) 
\label{mapdeltaweight}
\end{equation}
is bounded for any $\delta \in \RR$ and $\tau \in \RR^k$. The same holds for the generalized Laplacian 
\begin{equation}
\begin{split}
\sL= \nabla^*\nabla + \sR\colon  \rho^{\delta+\frac{n}{2}}w^{\tau +   \frac{\nu}{2}}H^{s+2}(Z, E, dV_g) \to  \qquad \qquad \qquad \\
\qquad \qquad \qquad \rho^{\delta+\frac{n}{2}-2}w^{\tau+ \frac{\nu}{2}-\utwo}H^s(Z, E, dV_g) 
\end{split} 
\label{mapdeltaweight-gL}
\end{equation}
provided $\sR \in \operatorname{End}(E)$ satisfies the pointwise bounds
\begin{equation}
  |\sR| \preceq \rho^{-2} w^{-\utwo} = \rho^{-2} w_1^{-2}.
\label{bddR}
\end{equation}
 However, for certain discrete values of the weight parameters, 
\eqref{mapdeltaweight} and \eqref{mapdeltaweight-gL} do not  have closed range, while for other values, their kernel or cokernel 
may be infinite dimensional.  Our goal is to determine the weights for
which \eqref{mapdeltaweight} and \eqref{mapdeltaweight-gL} are
Fredholm. 

We now turn to weighted H\"older spaces. We first record a Lemma which shows that
the weight functions are essentially constant on small geodesic balls. 
\begin{lemma}\label{lem:weight-const}
Let $\inj_g(Z)$ denote the injectivity radius of $(Z, g) \in \sQ_k$. Then there exists a constant 
$c \in (0, \inj_g(Z))$ so that for all $z\in Z$, 
\[
\rho(z') \asymp \rho(z), \quad w_{j} (z') \asymp   w_j (z) \quad
 \forall \ 1\leq j \leq k
 \]
 for all $z' \in B(z, c)$. 
\end{lemma}
\begin{proof}
We do this by induction. To keep track of the depth we reintroduce
sub/superscripts. In the depth $0$ (AC) case, there exists $R_0 >1$ so that 
if $\rho_0(z) > R_0$, then $B(z, c) \subset \{ z \in Z:\rho_0(z) >1\}$.  For any such $z$,
$(1-c) \rho_0(z) \leq \rho_0(z') \leq (1+c) \rho_0(z)$ when $z' \in  B(z, c)$.  The remaining
region in $Z$ is compact, so the bound is trivial there. 

Now assume that the assertion holds for all the QAC spaces of depth strictly less than $k$, and let us 
prove it for $(Z,g) \in \sQ_k$. Use the disjoint decomposition $Z = \bigsqcup_{j=0}^k Z^{\d{k}}_{\d{j}}$ and
its thickening $Z = \bigcup_{j=0}^k Z^{\d{k}}_{\d{j}}(\eta)$, where $\eta$ is chosen so that if $z \in Z^{\d{k}}_{\d{j}}$,
then $B(z, c) \subset Z^{\d{k}}_{\d{j}}(\eta)$. By the induction hypothesis, it suffices to assume that 
$B(z, c) \subset Z^{\d{k}}_{\d{k}}$.  Since $Z^{\d{k}}_{\d{k}}$ is a cone over $Y^{\d{k}}_{\d{k}}$, which is in
turn the total space of a bundle over $S_k$ with fiber certain truncations of $Z^{\d{k-1}} \in \sQ_{k-1}$, 
we can choose $c$ so that all these balls lie in trivialized charts. On each such neighborhood we 
can assume that the metric is a product, so
\[
B(z, c) \subset B(q_{k}, c) \times B(z_{k-1},c) \subset C(S_k) \times Z^{\d{k-1}}
\]
and
\[
w^{\d{k}}_k(z) = \frac{\rho_{k-1}(z_{k-1})}{\rho_k(p)}, \quad
w_j^{\d{k}}(z) = w^{\d{k}}_k(z) w^{\d{k-1}}_{j} (z_{k-1}) \quad 1\leq j \leq k-1. 
\]
By induction again, $\rho_{k-1}(z'_{k-1})$ and $w^{\d{k-1}}_j(z'_{k-1})$ are comparable to 
$\rho_{k-1}(z_{k-1})$ and $w^{\d{k-1}}_j(z_{k-1})$, respectively, if $z_{k-1}' \in B(z_{k-1}, c)\subset Z^{\d{k-1}}$. 
It is also clear that  $\rho_k(z')$ is comparable to $\rho_k(z)$ when $z'\in B(z, c)$, so the result follows. 
 \end{proof}

\begin{remark}
This Lemma makes clear that any QAC space $(Z,g)$ has bounded geometry, i.e., it is covered by balls with
radius uniformly bounded below, on each of which the curvature tensor and its derivatives are uniformly
bounded. 
\end{remark}

Using this lemma, we define for each $\gamma \in (0,1)$ the weighted H\"older seminorm on $Z$ as the supremum 
over all balls $B(z, c)$ in $Z$ of the weighted H\"older seminorm on that ball with respect to the distange $d$:
\begin{equation}
[ T]_{0, \gamma} = \sup_{z \in Z} \, \rho(z)^{-\gamma} w(z)^{-\underline{\gamma}} \sup_{ z'\neq z \atop z' \in B(z,c)} \frac{ |T(z) - T(z')|}{ d(z,z')^\gamma}.
\label{seminorm}
\end{equation}
(If $T$ is a tensor, then we parallel transport $T(z')$ on the radial geodesic from $z$ to $z'$ to make
sense of the difference.) 

Now define the weighted H\"older space $\rho^{\delta}w^{\tau} \sC^{k,\gamma}_g(Z)$ to be the set of functions 
$f\in \sC^{k}(Z)$ for which the norm
\[
  \|f\|_{\rho^{\delta}w^{\tau} \sC^{k,\gamma}}= \sum_{j=0}^k \sup_{Z}
  |\rho^{-\delta+j} w^{-\tau+\underline{j}}\,\nabla^jf| + [\rho^{-\delta+k}w^{-\tau+\underline{k}}\,\nabla^kf]_{0,\gamma}
\]
is finite.  Clearly, if $f\in\rho^\delta w^\tau \calC^{k,\gamma}_g(Z)$, then $f=\calO_k(\rho^\delta w^\tau)$; in particular
\[
  \rho^{\delta'}w^{\tau'} \calC^{k,\gamma}_g(Z) \subset \rho^{\delta'+\frac{n}{2}}w^{\tau' +\frac{\nu}{2}}H^k(Z,dV_g)
\]
for all $\delta'\in \RR$ and $\tau \in\RR^k$ so that
\[
  \delta'<\delta 
 \quad \text{and} \quad \delta' +|\tau'(j)| < \delta
  + |\tau(j)|
\]
for all $1\leq j \leq k$.

Once again, it is straightforward that 
\begin{equation}
\Delta\colon \rho^{\delta}w^{\tau}\calC^{k+2,\gamma}_g(Z) \to  \rho^{\delta-2}w^{\tau -\utwo}\calC^{k,\gamma}_g(Z)
\label{mapdeltaweight2}
\end{equation}
is bounded for any $k \geq 0$, and for any $\delta \in \RR$ and $\tau
\in \RR^k$.  The same holds for the generalized Laplacians
\begin{equation}
\sL=\nabla^*\nabla+\sR\colon  \rho^{\delta}w^{\tau}\calC^{k+2,\gamma}_g(Z, E) \to
\rho^{\delta-2}w^{\tau -\utwo}\calC^{k,\gamma}_g(Z, E),
\label{mapdeltaweight2-gL}
\end{equation}
provided $\sR\in \operatorname{End}(E)$ satisfies the pointwise bound
in~\eqref{bddR}.
As before, they fail to have closed range when $\delta$ lies in a discrete
set (these are the analogues of the familiar indicial roots in the AC case).  

Our main result in this paper concerns when the mappings \eqref{mapdeltaweight} and 
\eqref{mapdeltaweight2} for the scalar Laplacian
$\Delta$ and \eqref{mapdeltaweight-gL} and
\eqref{mapdeltaweight2-gL} for the
generalized Laplacian $\sL = \nabla^* \nabla + \sR$ acting on
sections of a geometric vector bundle $E$ are Fredholm.
Because they take some notation to set up, we defer the statement of the precise results to \S 7. 

\section{Heat kernel and Green function}
As outlined in the introduction, we shall show that the generalized Laplacians $\calL$ we consider are Fredholm between 
the Sobolev and H\"older spaces, for appropriate ranges values of the
weight parameters. We do this by constructing the 
corresponding Green functions $G_{\sL}$ and studying their mapping  properties on these function spaces. 
A crucial observation is that if we construct Green functions with suitable mapping properties on each end 
of $Z$ (with any elliptic boundary conditions at the compact boundary), then we may glue these kernels
together to obtain a good parametrix for $\calL$ on the entire space $Z$. This allows us to reduce
to spaces which are connected at infinity, where there is a good set of tools.

The method we employ to construct $G_{\sL}$ is to express it in terms of the heat kernel $H_{\sL}$ using the Laplace transform:
\begin{equation}
G_\calL (z,z') = \int_0^\infty e^{-t \calL}(z,z')\, dt = \int_0^\infty H_\calL(t,z,z')\, dt.
\label{Laplace}
\end{equation}
Note that this integral converges on 
any finite interval $[0,T]$ for $z \neq z'$; the more difficult issue is to estimate the behavior of $H_{\sL}$ as 
$t \nearrow \infty$ and to prove that this integral converges there. This requires the estimate $|H_{\sL}(t,z,z')| \leq C t^{-1-\e}$ 
for some $\e > 0$ when $t$ is large enough.  As we explain later,  the precise decay rate of $H_{\sL}$ as $t$ tends to infinity (even just along the 
diagonal $\{z = z'\}$) determines the asymptotic off-diagonal bounds for $G_{\sL}(z,z')$ as $d(z,z') \to \infty$.

The first main step in this analysis is to obtain these large-time bounds on $H_{\sL}$. To do this, we first show that 
under certain restrictions on the endomorphism term $\calR$ in $\calL = \nabla^* \nabla + \calR$, $H_\calL$ 
is dominated by the heat kernel for the scalar problem $H_{\Delta + V}$, where $V = -\Delta h/h$ for some 
strictly positive smooth function $h$. This scalar heat kernel is equivalent to the heat kernel for a weighted
operator $\Delta_\mu$, which is self-adjoint with respect to the measure $d\mu = h^2 dV_g$. We then 
follow a powerful and general method of estimation of such scalar heat kernels developed by Grigor'yan and 
Saloff-Coste \cite{Grigoryan2005} which gives sharp upper and lower bounds for $H_{\Delta_\mu}$ in terms of the volumes of 
geodesic balls of radius $\sqrt{t}$ with respect to $d\mu$. These estimate require that the `weighted manifold' 
$(Z, g, d\mu)$ satisfies a set of geometric and analytic properties. 

In this section we explain these reductions carefully and then, in the next section, verify that these geometric properties hold.
Tracing back through the reductions, we obtain bounds for $G_\calL$ itself. The latter sections of this paper 
use these bounds to determine the weighted spaces on which $G_\calL$ is bounded.

\subsection{Heat kernel domination}
Consider the operator $\sL = \nabla^* \nabla + \sR$ on an arbitrary complete manifold $Z$,
possibly with compact boundary, acting on sections of a Hermitian vector bundle $E$ with connection over
$Z$, and where $\sR$ is a semibounded endomorphism of $E$. If $Z$ has boundary, we consider the Friedrichs
extension of $\sL$. Thus $\sL$ is a self-adjoint operator acting on $L^2(Z, dV_g)$ (with Dirichlet
boundary conditions if there is a boundary), and there is a uniquely defined minimal heat kernel $H_{\calL}$. 
In this subsection we recall one standard technique to estimate $H_{\sL}$ in terms of the heat kernel of a scalar operator. 

Fix a smooth strictly positive function $h$ on $Z$ and consider the measure $d\mu = h^2 dV_g$.  This determines the 
weighted $L^2$ inner product
\[
\langle u , v \rangle_\mu = \int_Z \langle u, v \rangle\, d\mu,
\]
where the inner product in the integrand is the Hermitian metric on the bundle.  Let $\nabla^{*,\mu}$ be 
the adjoint with respect to this measure, i.e., $\nabla^{*,\mu} = h^{-2} \nabla^* h^2$, and $\nabla^{*,\mu} \nabla$ 
the rough Laplacian, 
\begin{equation}
\nabla^{*,\mu}\nabla = h^{-2} \nabla^* h^2 \nabla.
\label{roughLap}
\end{equation}

\begin{lemma}\label{lem:muL} 
Let $V =  - \frac{\Delta h}{h}$. The operator
\[
  \widetilde{\calL}= \nabla^{*,\mu} \nabla + \sR - V\!\cdot\mathrm{Id}
\]
is the Doob transform of $\calL$, i.e.
\begin{equation}\label{eq:muL}
  \widetilde{\calL} = \frac{1}{h}\circ\sL \circ h.
\end{equation}
In particular, the Doob transform of the scalar operator $\Delta+V$ is
\begin{equation}
 \Delta_\mu = \frac{1}{h} \circ \Delta \circ h + V.
\label{eq:muDel}
\end{equation}
Finally, the heat kernels for $\calL$ and $\widetilde{\calL}$ are
related by
\begin{equation}\label{eq:muHeat}
 H_{\sL}(t,z,z') = h(z) h(z') H_{\wtilde\sL} (t,z,z')
\end{equation}
 for all $(t,z,z') \in (0, +\infty) \times Z \times Z$.
\end{lemma}
\begin{proof}
For the first claim, let $u,v$ be two $\calC^\infty_0$ sections of $E$. Then
\begin{multline*}
\<\nabla^{*,\mu}\nabla u, v\>_{\mu}= \<\nabla u, \nabla v\>_{\mu}  = \<h\nabla u, h \nabla v\> \\
= \<\nabla (hu),\nabla (hv)\> - \<\nabla(hu),dh\tensor v\> - \< dh\tensor u, \nabla(hv)\>+ \<dh\tensor u,dh\tensor v\>.
 \end{multline*}
The first term on the right is just $\< (\frac{1}{h} \circ \nabla^*\nabla \circ h) \,  u ,v\>_{\mu}$.  The second and the fourth terms 
together give $-\< h \nabla u, dh \otimes v \>$. Finally, the third term is
\begin{multline*}
- \< dh \tensor u, \nabla(hv)\> = - \<\nabla^* (dh\tensor u), hv\> = - \<\Delta h\,  u, hv\> + \< \iota(\nabla h) \nabla u, h \otimes v\> \\
= - \<\frac{\Delta h}{h} u, v\>_{\mu} + \<h \nabla u, dh \otimes v\>
 \end{multline*}
Inserting the definition of $V$ gives $\nabla^{*,\mu}\nabla =
\frac{1}{h} \circ \nabla^*\nabla \circ h + V$, which proves \eqref{eq:muL};
\eqref{eq:muDel} is a special case. 

Finally, write 
\[
u(t,z) = \int_{Z} H_{\wtilde\sL} (t,z,z') f(z')d\mu(z'), 
\]
so that we have
\[
 \del_t u+ \wtilde \sL u = 0, \qquad u(0,z) = f(z)
\]
as an evolution equation weakly in $L^2(Z,d\mu)$.
We have already verified that $\wtilde \sL = \frac{1}{h}\circ \sL
\circ h$, hence
\[
\del_t (hu) + \sL(hu) = 0, \qquad (hu)(0,z) = h(z)f(z)
\]
weakly in $L^2(Z,dV_g)$. This gives
\begin{multline*}
(hu)(t, z) = \int_Z H_{\sL}(t,z,z') h(z')f(z')\, dV_g(z') \Longrightarrow  \\
u(t, z) = \int_Z h(z) H_{\sL}(t,z,z') h(z') f(z')\, d\mu_g(z'),
\end{multline*}
which implies \eqref{eq:muHeat}.
\end{proof}

The reasons for making these transformations emerges in the following
chain of lemmas. 
\begin{lemma}\label{lem:pmuH}
If $h$ is chosen so that $\sR - V\!\cdot\mathrm{Id} \geq 0$, then 
 \begin{equation}\label{eq:pmuH}
|H_{\wtilde\sL}(t,z,z')|   \leq |H_{\nabla^{*,\mu}\nabla}(t,z,z')|.
 \end{equation}
\end{lemma}
\begin{proof}
We apply the Trotter product formula, which asserts that if $A$ and $B$ are self-adjoint and semibounded, 
then $e^{-t(A+B)} = \lim_{k\to \infty}(e^{-tA/k} e^{-tB/k})^k$. With $A = \nabla^{*,\mu}\nabla$ and $B = \sR - V$, this gives  
\[
 e^{-t\wtilde\sL} = \lim_{k\to \infty} \big(e^{-t\nabla^{*,\mu}\nabla/k} e^{-t(\sR-V\,\mathrm{Id})/k}\big)^k.
\]
Taking the pointwise norm of the terms of the above sequence and using
$|e^{-\tau (\sR - V)}| \leq 1$ for $\tau \geq 0$, we have
\[
 |H_{\nabla^{*,\mu}\nabla}(t/k,z,z')  e^{-t(\sR-V\,\mathrm{Id})/k}| \leq |H_{\nabla^{*,\mu}\nabla}(t/k,z,z')|.
\]
Composing this $k$ times gives the result.
\end{proof}
The next lemma adapts a technique of  Donnelly and Li \cite[Lemma 4.1]{Donnelly1982}. 
\begin{lemma}\label{lem:muH}
The pointwise norm of the heat kernel for $\nabla^{*,\mu}\nabla$ is dominated by the heat kernel of $\Delta_\mu$, 
\begin{equation}\label{eq:muH}
|H_{\nabla^{*,\mu}\nabla}(t,z,z')| \leq H_{\Delta_{\mu}}(t,z,z').
 \end{equation}
\end{lemma}
\begin{proof}
We first claim that
 \begin{equation}\label{eq:heat_i}
\left(\frac{\partial}{\partial t} + \Delta_{\mu}\right) | H_{\nabla^{*,\mu}\nabla}| \leq 0.
 \end{equation}
To prove this, define $K_{\e} = (|H_{\nabla^{*,\mu}\nabla}|^2 + \e^2)^{1/2}$. Then for $\e > 0$, 
\begin{multline*}
 K_{\e}\, \del_t K_\e= \frac{1}{2} \del_t K_{\e}^2= \frac{1}{2}
 \del_t (| H_{\nabla^{*,\mu}\nabla}|^2 +\epsilon^2) \\
= \< \del_t  H_{\nabla^{*,\mu}\nabla}, H_{\nabla^{*,\mu}\nabla}\> =  
\< -\nabla^{*,\mu}\nabla H_{\nabla^{*,\mu}\nabla},  H_{\nabla^{*,\mu}\nabla}\>;
\end{multline*}
similarly, 
\begin{multline*}
K_{\e}\, \Delta_{\mu} K_{\e} = \frac{1}{2}\Delta_{\mu} (K_{\e}^2) + |\nabla K_{\e}|_{\mu}^2= 
\frac{1}{2} \Delta_{\mu} | H_{\nabla^*\nabla}|^2 +|\nabla K_{\e}|_{\mu}^2\\
=\< \nabla^{*,\mu}\nabla  H_{\nabla^{*,\mu}\nabla},  H_{\nabla^{*,\mu}\nabla}\> -
|\nabla H_{\nabla^{*,\mu}\nabla}|^2+|\nabla K_{\e}|_{\mu}^2,
\end{multline*}
where we have computed using that $\Delta_\mu = h^{-2} \nabla^* h^2 \nabla$. 
Adding these two expressions together gives
\[
K_{\e} (\del_t + \Delta_{\mu}) K_{\e}=-|\nabla  H_{\nabla^{*,\mu}\nabla}|^2 + |\nabla K_{\e}|^2.
\]
However, 
\[
|\nabla K_\e| = \frac{ \big|\nabla | H_{\nabla^{*,\mu}\nabla}|^2\big|}{2\sqrt{|H_{\nabla^{*,\mu}\nabla}|^2 + \e^2}}
\leq |\nabla H_{\nabla^{*,\mu}\nabla}| \frac{|H_{\nabla^{*,\mu}\nabla}|}{\sqrt{|H_{\nabla^{*,\mu}\nabla}|^2 + \e^2}} \leq |\nabla H_{\nabla^{*,\mu}\nabla}|,
\]
which gives the desired differential inequality for $K_\e$.
Finally, letting $\e \to 0$ gives the same inequality for $|H_{\nabla^{*,\mu}\nabla}|$. 

For the second step of the proof, fix $t>0$ and define
\[
 f(s,z,z') = \int_Z H_{\Delta_{\mu}}(t-s,z,w)\, |H_{\nabla^{*,\mu}\nabla}(s,w,z')| d\mu(w),
\]
so that
\[
 \int_{0}^t \del_s f(s,z,z') \, ds = | H_{\nabla^{*,\mu}\nabla}(t,z,z')|-  H_{\Delta_{\mu}}(t,z,z').
\]
On the other hand,
\[
\begin{array}{rl}
 \int_{0}^t \del_s f(s,z,z') \, ds 
&= \int_0^t\del_s \left( \int_Z
  H_{\Delta_{\mu}}(t-s,z,w)\,|H_{\nabla^{*,\mu}\nabla}(s,w,z')| \, d\mu(w)\right)\, ds\\[1mm]
&= \int_{0}^t \int_Z \Delta_{\mu} (H_{\Delta_{\mu}}(t-s,z,w))\,|H_{\nabla^{*,\mu}\nabla}(s,w,z')|\, d\mu(w)ds\\[1mm]
& \qquad 
 + \int_{0}^t \int_Z  H_{\Delta_{\mu}}(t-s,z,w)\,\del_s \left(| H_{\nabla^{*,\mu}\nabla}(s,w,z')|\right)\, d\mu(w) ds\\[1mm]
& = \int_0^t \int_Z  H_{\Delta_{\mu}}(t-s, z,w)\,(\del_s + \Delta_{\mu})\left(| H_{\nabla^{*,\mu}\nabla}(s,w,z')|\right)\,d\mu(w) ds,
\end{array}
\]
which is nonpositive by the first part of the proof and since $H_{\Delta_\mu} \geq 0$. 
\end{proof}

We have now verified the main chain of estimates
\begin{align*}
|H_{\sL}(t,z,z') |& = h(z)h(z')\ |H_{\widetilde \sL}(t,z, z')|\\[1mm]
& = h(z)h(z')\ |H_{\nabla^{*,\mu}\nabla + \sR - V}(t,z, z')|\\[1mm]
& \leq h(z)h(z')\ |H_{\nabla^{*,\mu}\nabla}(t,z, z')| \\[1mm]
& \leq h(z)h(z')\ H_{\Delta_{\mu}}(t,z,z') = H_{\Delta + V}(t,z,z'),
\end{align*}
which proves the 
\begin{theorem}\label{thm:heatL}
With all notation as above, 
let $h$ be a smooth strictly positive function, define $V = - (\Delta h)/h$, and suppose that 
\[
\sR - V\!\cdot\mathrm{Id} \geq 0.
\]
Then 
\begin{equation}\label{eq:hsV}
|H_{\sL}(t,z,z')|\leq H_{\Delta + V}(t,z,z') = h(z)h(z')\ H_{\Delta_{\mu}}(t,z,z'), 
 \end{equation}
and hence
\begin{equation}\label{eq:hsV2}
|G_{\sL}(z,z')|\leq G_{\Delta + V}(z,z') = h(z)h(z')\ G_{\Delta_{\mu}}(z,z'). 
\end{equation}
\end{theorem}

\subsection{The Grigor'yan and Saloff-Coste bounds} 
We now recall the geometric and analytic properties of the complete
weighted Riemannian manifold $(Z,g,d\mu)$ used
to obtain heat kernel bounds for the scalar operator $\Delta_{\mu}$ by
the method of Grigor'yan and Saloff-Coste\footnote{In their work,
  Grigor'yan and Saloff-Coste use the Laplace-Beltrami operator $\mathrm{div}
  \circ \mathrm{grad}$. Here we use the positive Laplacian, $\Delta u =
  \mathrm{div}(\mathrm{grad}(u))$, and adapt all their statements to
  our setting.}.  
\begin{definition}
The complete weighted Riemannian manifold $(Z,g,d\mu)$ satisfies:
\begin{itemize}
\item[ $(VD)_\mu$] the weighted {\em volume doubling property} if there
  exists a constant $C_D >0$ so that
\begin{equation}\label{eq:vd}
\mu(B(p,2r)) \leq C_D \, \mu(B(p,r))
\end{equation}
for all $p \in Z$ and all radii $r>0$;
\item[$(PI)_{\mu,\delta}$] the uniform weighted {\em Poincar\'e inequality}
with parameter $\delta \in (0,1]$, if there exists a constant $C_P >0$
so that 
\begin{equation}\label{eq:pi}
\int_{B(p,r)} (f- \bar{f})^2 \, d\mu \leq C_P \, r^{2} \int_{B(p,\delta^{-1} r)} |\nabla f|^2\, d\mu
\end{equation}
for every $f\in W^{1,2}_{loc}(M)$ and all $p\in M$ and radii $r>0$; here $\bar{f}$ is the average
$\frac{1}{\mu(B(p,r))}\int_{B(p,r)} f \,d\mu$.
\item[$(PI)_{\mu}$] the uniform weighted Poincar\'e inequality
  $(PI)_{\mu}$, if above we can choose $\delta =1$.
\end{itemize}
\end{definition}
\begin{remark}\label{rem:jerison}
An intricate covering argument due to Jerison~\cite{Jerison1986} shows that if $(Z,g,d\mu)$ satisfies 
$(VD)_{\mu}$ and $(PI)_{\mu, \delta}$ for some $\delta \in (0,1)$, then it satisfies $(PI)_{\mu}$, see also \cite[Ch. 5.3]{SaloffCoste2002}.
\end{remark}

The verification that a QAC space $(Z,g, d\mu)$ with an appropriately chosen measure satisfies $(VD)_\mu$ and $(PI)_\mu$ 
is not easy and occupies much of \S 4. Let us grant these for the time being and state the heat kernel bounds 
that follow from them. This was done by Grigor'yan and Saloff-Coste,
and we choose to apply them in the form stated in~\cite[Theorem
2.7]{Grigoryan2005} (see the history for this type of results therein).

\begin{theorem}\label{thm:Hmuest}
Let $(Z,g, d\mu)$ be a weighted QAC manifold satisfying $(VD)_\mu$ and $(PI)_{\mu}$. Then\footnote{Recall
  that the notation that $f_1 \asymp f_2$ means that there
  exist two constants $c, C>0$ so that $c f_2 \leq f_1 \leq C f_2$;
  see Table~\ref{table} on page~\pageref{table} for notational conventions.}
\begin{equation}\label{eq:Hmuest}
H_{\Delta_\mu} (t,z,z')  \asymp \left(\mu(B(z,\sqrt t))\mu(B(z',\sqrt t))\right)^{-\frac12} e^{-c d(z,z')^2/t}
\end{equation}
for all $(t,z,z') \in (0,+\infty) \times Z \times Z$. By the results
in Theorem~\ref{thm:heatL}, this implies that 
\begin{equation}
  H_{\Delta +V}(t,z,z') \asymp  h(z) h(z') \left(\mu(B(z,\sqrt{t})) \mu(B(z',\sqrt{t}))\right)^{-1/2} e^{- c d(z,z')^2/t},
\label{heatbounds1}
\end{equation}
where $V = -\Delta h/h$. 
\end{theorem}
We translate this back to Green function estimates for $\Delta +
V$ using \eqref{Laplace}.
\begin{theorem}\label{thm:Gmuest}
If $(Z,g,d\mu)$ satisfies $(VD)_{\mu}$ and $(PI)_{\mu}$, then 
\[
  G_{\Delta + V} (z,z') \asymp h(z) h(z') \int_{d(z,z')}^{+\infty} s \left(\mu(B(z,s)) \mu(B(z',s))\right)^{-1/2}\, ds.
\]
\end{theorem}
\begin{proof}
Following \cite[p.192]{Li1986}, insert \eqref{heatbounds1} into the
integral in~\eqref{Laplace} defining $G_{\Delta+V}$ in terms of $H_{\Delta+V}$, 
and decompose this integral as 
\[
G_{\Delta+V}(z,z') = \int_{0}^{d(z,z')^2} H_{\Delta + V}(t,z,z') \, dt + \int_{d(z,z')^2}^{+\infty} H_{\Delta + V}(t,z,z') \, dt. 
\]
When $t\geq d(z,z')^2$, we have $e^{-c d(z,z')^2/t} \asymp 1$, so the second sumand is comparable to 
\begin{equation}
h(z) h(z')\int_{d(z,z')^2}^{+\infty} \left(\mu(B(z,\sqrt t))\mu(B(z',\sqrt t))\right)^{-1/2} \, dt. 
\label{heatbounds2}
\end{equation}

We claim that the first summand satisfies the same estimate.  Indeed, use \eqref{heatbounds1} again and change 
the variable of integration to $\tau = d(z,z')^4/t$, so $\tau \geq d(z,z')^2$ and $d(z,z')^2/t = \tau/d(z,z')^2$. 
This term is then comparable to 
\begin{multline*}
h(z) h(z')\int_{d(z,z')^2}^{+\infty}
\left(\mu(B(z, d(z,z')^2/\sqrt{\tau}) \mu(B(z',d(z,z')^2/ \sqrt{\tau} ) \right)^{-1/2}\\
\times e^{-c \frac{\tau}{d(z,z')^2}}\,  \left(\frac{d(z,z')^4}{\tau^2}\right)\, d\tau.
\end{multline*}
\noindent
Using $(VD)_{\mu}$ and $\sqrt{\tau} = (d(z,z')^2/\sqrt{\tau}) (\tau/d(z,z'))^2$, we obtain
\[
\mu(B(z,\sqrt{\tau})) \leq  \left(\frac{\tau}{d(z,z')^2}\right)^\gamma   \mu(B(z, d(z,z')^2/ \sqrt{\tau }  ) )
\]
for some $\gamma > 1$ which is independent of $z,z'$ and $\tau$.  Observing also that $e^{-c\lambda} 
\lambda^{\gamma-2} \leq C $ when $\lambda = \tau/d(z,z')^2 \geq 1$, this integral is bounded by 
\eqref{heatbounds2}, as claimed.    On the other hand, 
\[
\mu(B(z, d(z,z')^2/ \sqrt{\tau }  ) ) \leq \mu(B(z,\sqrt{\tau}))
\]
so an even simpler argument gives the bound in the other direction. 

Finally, setting $t = s^2$ in \eqref{heatbounds2} yields the expression in the statement of the theorem.
\end{proof}

\section{Geometric estimates}
Fix $(Z, g) \in \sQ_k$ and for $a\in \RR$ and $b \in \RR^k$ consider
the measure $d\mu_{a,b}=\rho^a w^b dV_g$ on it. Our goal in this section is to
determine when $d\mu_{a,b}$ satisfies the volume doubling and Poincar\'e inequalities,
which are the hypotheses of Theorem~\ref{thm:Hmuest} and
Theorem~\ref{thm:Gmuest}.

To simplify notation, we often omit the superscript $(k)$ on $Z$, $g$ and the functions $w$. We also
continue to assume that there is a unique maximal chain of strata  $S_k < S_{k-1} < \ldots < S_1$ 
in the degenerate limit $Y_0$ of the cross-section $Y$ of $Z$; the extension to the general case is straightforward.
Let $Z^{(j-1)} \in \sQ_{j-1}$ be the QAC space used to desingularize the cone bundle along the stratum $S_{j}$, 
and write $m_{j-1} = \dim Z^{(j-1)}$ and $n = \dim Z$.  We refer to Table~\ref{table} on page~\pageref{table} for
further notation. 

\begin{theorem}\label{thm:main4}
The weighted space $(Z,g,d\mu_{a,b}) \in \sQ_k$ satisfies $(VD)_{\mu}$
and $(PI)_{\mu}$ provided the parameters $a\in \RR$ and $b \in \RR^k$ satisfy
\begin{equation}\label{eq:cVC}
 a+n \geq 0 \quad  \text{and} \quad |b(j)| + m_{j-1} \geq 0 \quad \text{for all $1\leq j \leq k$.}
\end{equation}
\end{theorem}

We shall actually show that $(Z,g, d\mu_{a,b})$ satistifies $(VD)_{\mu}$ and $(PI)_{\mu,\delta}$ for some $\delta \in (0,1)$
and then appeal to Remark~\ref{rem:jerison} to conclude that $(PI)_{\mu}$ holds with $\delta = 1$.

The difficulty in proving this theorem is that balls $B(p,r)$ in $Z$ are characterized differently, depending 
on the location of $p$ and the size of $r$ relative to $\rho(p)$. To describe this, we follow the terminology of \cite{Grigoryan2005}:


\begin{definition}\label{defn:ra}
Fix a basepoint $o \in K_Z$. A geodesic ball $B(o,R)$ centered at $o$ is called {\em anchored}, 
and we denote its volume by
\[
\sA_k(R;a,b) = \mu_{a,b}(B(o,R)).
\] 
Now fix a parameter $c \in (0,1)$. A geodesic ball $B(p,r)$ is called {\em remote} if $r \leq c \, 
\mbox{dist}\,(p,o)$, and we then write
\[
\sR_k(p,r;a,b) = \mu_{a,b}(B(p,r)).
\]  
Finally, if $B(p,r)$ is any ball, possibly neither anchored nor remote, then we write
\[
\sV_k(p,r;a,b) = \mu_{a,b}(B(p,r)).
\]
\end{definition}
Note that $\mbox{dist}\,(o,p)$ is comparable to $\rho(p)$, so we use the latter henceforth. 

A central idea in the analysis below is that under certain conditions on the weight parameters,
and assuming also a topological condition on $Z$ (which is easy to check in our setting)
it is sufficient to check $(VD)_\mu$ and $(PI)_{\mu,\delta}$ only on remote balls. The chain of reasoning
is supplied by the following results.

\begin{proposition}\label{lem:vc}
Suppose that the weight parameters $a\in \RR$ and $b \in \RR^k$ satisfy condition~\eqref{eq:cVC}.
Let $c$ be the parameter used to characterize remote balls in $(Z,g)$. Then there exists a 
constant $C_V > 0$ such that 
\begin{equation}
\sA_k( \rho(p); a,b) \leq C_V \, \sR_k( p, c\, \rho(p); a,b).
\label{vc}
\end{equation}
\end{proposition}
The use of such an `anchored/remote' volume comparison to obtain volume doubling appears,
for example, in \cite{Li1995}.

\begin{definition}
 A metric space $(Z,d)$ is said to have the property of relatively connected annuli (RCA) with respect 
to a point $o\in Z$ if there exists a constant $C_A  >1$ such that for any $r>C_A^2$ and 
for every $p,q \in M$ with $d(o,p) =  d(o,q) =r$, there exists a continuous path 
$\gamma\colon [0,1] \to Z$ with $\gamma(0) =p$, $\gamma(1) = q$, and with image contained 
in the annulus $B(o,C_A r) \setminus B(o,C_A^{-1}r)$. 
\end{definition}
\begin{remark}
In our setting $Z$ may have more than one end, which clearly causes
(RCA) to fail. We only need to require that (RCA) holds on each end separately. Indeed, if we
restrict attention to each end, then using (RCA), the arguments in \S 7 show 
that the restrictions of $L_Z$ to the individual ends is Fredholm on the appropriate weighted
spaces, and it is straightforward to deduce from this that $L_Z$ is Fredholm on all of $Z$.
We explain this in more detail later.
\end{remark}

\begin{proposition}[\cite{Grigoryan2005}, Theorem 5.2]
\label{thm:gsc52}
Suppose that $(Z,g,d\mu_{a,b})$ satisfies the (RCA) property with respect a point $o\in Z$,
and assume that $(VD)_{\mu}$ and $(PI)_{\mu, \delta}$ with parameter $\delta \in (0,1]$
hold for all remote balls with respect to $\{o\}$. Then
$(Z,g,d\mu_{a,b})$ satisfies $(VD)_{\mu}$ and $(PI)_{\mu,\delta}$ for
all balls if and only if it satisfies the volume comparison estimate \eqref{vc}.
\end{proposition}
To give just a very rough idea of the idea behind this last Proposition, the proof proceeds by 
reducing to a combinatorial problem on the discretization of concentric annuli.  This requires 
that one show that certain Poincar\'e inequalities hold on each such annulus, which
in turn requires connectivity of the annuli, i.e.\ the (RCA) condition.

We devote the remaining of this section to the proof of
Theorem~\ref{thm:main4}, and implicitly of Proposition~\ref{lem:vc}.
The strategy for proving it for a QAC space of
depth $k$ which satisfies the (RCA) property is inductive and uses
Proposition~\ref{thm:gsc52}.  
Both properties $(VD)_{\mu}$ and $(PI)_{\mu,\delta}$ are easy to
check on spaces of depth $0$. Therefore we assume that $(VD)_{\mu}$
and $(PI)_{\mu,\delta}$ as well as the inequality~\eqref{vc} in
Proposition~\ref{lem:vc} have been verified on
all QAC spaces of depth less than $k$. 
To obtain the inductive step it is then enough by
Proposition~\ref{thm:gsc52} to check that on space of depth $k$,  both
$(VD)_{\mu}$ and $(PI)_{\mu,\delta}$ hold on its remote balls, and also to check inequality~\eqref{vc}.
To establish the volume comparison results required by $(VD)_{\mu}$ on
remote balls and by~\eqref{vc}, we first obtain asymptotic formulas
for the measure of balls in a QAC space. Concretely, we obtain an
estimate for the volume of anchored balls, then 
find an inductive relationship between
the measures of remote balls in depth $k$ and volumes of balls (not necessarily
anchored or remote) in lower depth, and end with a further estimate on
the volume of balls which are non-remote. The first two estimates
are
enough to establish $(VD)_{\mu}$, while the third is needed to obtain
the volume comparison inequality~\eqref{vc}. Both these results
require also the constraint on $a$ and $b$ given in~\eqref{eq:cVC}. Once this is done, we then
verify $(PI)_{\mu}$ on remote balls, and thus complete the
inductive step to show that $(VD)_\mu$ and $(PI)_{\mu}$ hold in
depth $k$.  

\subsection{Estimates for the volume of balls in a QAC space}

In this section we give some estimates for the weighted volume of balls in
$(Z^{\d{k}}, g^{\d{k}})$ with measure $d\mu_{a,b}= \rho_k^a (w^{\d{k}})^bdV_{g^{\d{k}}}$
for $a\in \RR$ and $b\in \RR^k$. We first give such an estimate for anchored
balls. Next we prove a crude estimate for the volume of remote balls, in
terms of the volume of balls in QAC spaces of lower depth. Combining
both these estimate we then easily derive the volume comparison
estimate~\eqref{vc} in Proposition 4.1 and $(VD)_\mu$ for remote balls.

\subsubsection{Estimates on the volume of anchored balls in depth $k$}\label{sec:ab}
We begin by estimating the weighted volume $\sA_k(R;a,b)$ of an
anchored ball of radius $R$. If $R\leq 1$, standard comparison theorems
give $\sA_k(R;a,b) \asymp R^n$. As such we may assume that $R>1$.
\begin{proposition}\label{lem:anch_est}
With all notation as above, then for $R > 1$ and any choice of measure
$d\mu_{a, b}$ with parameters $a\in \RR$ and $b\in \RR^k$,  
\begin{equation}
\sA_k(R;a,b) \asymp 1 + R^{a+n}\left(1 + \sum_{j=1}^k R^{- m_{j-1} -|b(j)|}\right).
\label{volanch}
\end{equation} 
\end{proposition}
\begin{proof}
 We prove this by induction. When $k=0$, then $(Z^{\d{0}}, g^{\d{0}})$
 is an AC space and our measure is $d\mu_a(x) = \rho_0(x)^a
 dV_{g^{\d{0}}}(x)$. Then $\sA_0 (R; a)\asymp 1 + R^{a+n}$ with the constant
 coming from the compact part $K_{Z^{\d{0}}}$ and the term $R^{a+n}$
 coming from the AC end. 

We now assume that the statement is true for all QAC spaces of depth strictly less than $k$ and prove it for $(Z^{\d{k}}, g^{\d{k}})$ 
a QAC space of depth $k$. Using $Z^{\d{k}}=  \bigsqcup_{j=0}^k Z^{\d{k}}_{\d{j}}$ from~\eqref{eq:Zkj},  we have
\[
\sA_k(R; a, b)  = \mu_{a,b} (\{\rho_k \leq R\} ) 
= \sum_{j \leq k} \mu_{a, b}\left(\{\rho_k \leq R\} \cap Z^{\d{k}}_{\d{j}}\right). 
\]
The terms with $j < k$ involve integrals over regions which stay well away from the depth $k$ 
stratum, so by induction 
 \[
\sum_{j =0}^{k-1}\mu_{a,b}\left(\{\rho_k \leq R\} \cap Z^{\d{k}}_{\d{j}}\right) \asymp 1 + R^{a+n}(1 + \sum_{j=1}^{k-1} R^{-m_{j-1} - |b(j)|}).
\]
Thus we need only focus on the term with $j=k$. This too is done by induction. As in Section~\ref{sec:decomp}, $Z^{\d{k}}_{\d{k}}$ 
is a cone over $Y^{\d{k}}_{\d{k}} \subset Y^{\d{k}}$, which is itself a bundle over $S_k$ with fiber a truncation of $Z^{\d{k-1}}$. 
Since $S_k$ is a compact smooth manifold, we can break the integral up into a finite number of integrals over trivialized
neighbourhoods of this bundle where the metric is quasi-isometric to the product 
\[
d\rho_k^2 + \rho_k^2 d\kappa_{S_k}^2 + g^{(k-1)},
\]
see Lemma~\ref{lem:m}. Formula~\eqref{eq:w-Zkj} gives
\[ 
w^{\d{k}}_k(z) = \frac{\rho_{k-1}(z_{k-1})}{\rho_k(z)}<1  \quad\text{and}\quad
w^{\d{k}}_i(z) = \frac{\rho_{k-1}(z_{k-1})}{\rho_k(z)} w^{\d{k-1}}_i(z)
\]
when $i \leq k-1$. Hence, we have 
\begin{align*}
   d\mu_{a,b} 
   & \asymp \rho_k^{a+n-m_{k-1}-1-|b(k)|} \rho_{k-1}^{|b(k)|}
   (w^{\d{k-1}})^{b(k-1)} d\rho_k \ dV_{\kappa_{S_k}}\ dV_{g^{\d{k-1}}},
\end{align*}
since $\dim S_k = n-m_{k-1}-1$.  However, $\rho_{k-1}\leq \rho_k$ on $Z^{\d{k}}_{\d{k}}$, and hence 
\[
  \mu_{a,b} (\{\rho_k \leq R\} \cap Z^{\d{k}}_{\d{k}})
  \asymp  \int_{1}^{R} \rho_k^{a+n-m_{k-1}-1-|b(k)|}
  \sA_{k-1}(\rho_{k}; |b(k)|, b(k-1))\, d\rho_k.
\]
Using the inductive hypothesis, this becomes
\[
\int_1^R \rho_k^{a+n-m_{k-1}-|b(k)|-1}\bigg(1 + \rho_k^{|b(k)|+m_{k-1}}\bigg( 1 + 
\sum_{j=1}^{k-1} \rho_k^{-m_{j-1} - |b(j)|}\bigg)\bigg)\, d\rho_k,
\]
which reduces to the expression for $\sA_k(R;a,b)$.  This completes the inductive step.
\end{proof}

Note that if $a$ and $b$ satisfy~\eqref{eq:cVC} then the estimate~\eqref{volanch} for
the volume of an anchored ball becomes
\begin{equation}\label{volanch-2}
 \sA_k(R; a, b) \asymp R^{a+n}
\end{equation}
for all $R>1$. 

\subsubsection{A preliminary estimate for the volume of remote balls in depth $k$}\label{sec:rb}
Unfortunately we cannot obtain such a simple and definitive estimate for the volumes of remote balls. 
As a first step, we now relate remote balls in depth $k$ to balls, which may be neither remote nor anchored, in lower depth. 
 
For this we use the thickened decomposition $Z^{\d{k}} = \bigcup_{j=0}^k Z^{\d{k}}_{\d{j}}(\eta)$, $\eta\in
(0,1)$, introduced in~\eqref{eq:Zkj-t}. The following lemma, which establishes a relation between
the biggest remote balls $B(p, c\rho_k(p))$ and this decomposition, is key in the argument below.
\begin{lemma}\label{lem:rb}
Let $(Z^{\d{k}},g^{\d{k}}) \in \sQ_k$ and let $c \in (0,1)$ be a remote parameter. Then there exists $\eta \in (0,1)$,
depending only on $(Z^{\d{k}},g^{\d{k}})$ and $c$, so that for each $p \in Z^{\d{k}}$, the ball $B(p, c\rho_k(p))$ 
is contained in $Z^{\d{k}}_{\d{j}}(\eta)$ for some $j\leq k$. In fact, we can choose $\eta$ so that if 
\[
w_k^{\d{k}}(p), \ldots, w_{j+1}^{\d{k}}(p) \geq 1- 2c  \quad \text{and} \quad w_j^{\d{k}}(p) < 1- 2c
\]
for some $j \leq k$, then
\[
B(p, c\rho_k(p)) \subset Z^{\d{k}}_{\d{j}}(\eta).
\]
\end{lemma}
\begin{proof}
First, if $w^{\d{k}}_k(p) < 1-2c$ then clearly $B(p, c\rho_k(p)) \subset Z^{\d{k}}_{\d{k}}$.  On the other hand, if
$w^{\d{k}}_k(p) \geq 1-2c$, it follows that $w^{\d{k}}_k(z) > 1- 4c$  for all $z \in B(p, c\rho_k(p))$, and hence
$B(p, c\rho_k(p)) \subset \union_{j=0}^{k-1} Z^{\d{k}}_{\d{j}} (4c)$. Since this only involves QAC spaces of depth $k-1$,
the conclusion follows by induction. 
\end{proof}

With this Lemma, we see that each maximal remote ball $B(p, c\rho_k(p))$ lies in $Z^{\d{k}}_{\d{j}} (\eta)$ for some 
$\eta\in (0,1)$ and $j \leq k$.  If $j=0$, then this remote ball lies in $Z^{\d{k}}_{\d{0}}(\eta)$, and the measure $d\mu_{a,b}$ 
on this set is comparable to $\rho_k^a\,dV_{g^{\d{k}}}$. On the other hand, when $j \geq 1$, then this remote ball 
lies in $Z^{\d{k}}_{\d{j}}(\eta)$, and so we can assume that it lies in the product $C(S_{j}) \times Z^{\d{j-1}}$. Denote the components 
of $p$ in this splitting by $(q_j,p_{j-1})$. Note that $B(p_{j-1}, r)$ lies in the region $\{\rho_j \leq
\rho_k(p)\}$ in $Z^{(j-1)}$. We use this as follows.
\begin{proposition}\label{lem:rem}
Fix $p \in Z^{\d{k}}$, and suppose that $w_i^{\d{k}}(p) > 1- 2c$ for all $i \geq j+1$ while $w_j^{\d{k}}(p)< 1- 2c$. 
Then for all $0< r \leq c \rho_k(p)$ 
\begin{equation}\label{eq:c-rem}
\begin{split}
\sR_k( & p,r; a, b) \\
\quad & \asymp \begin{cases}
     \rho_k(p)^a r^n & \text{if $j=0$}\\
     \rho_k(p)^{a-|b(j)|} r^{n-m_{j-1}} \sV_{j-1}(p_{j-1}, r;|b(j)|, b(j-1)) & \text{if $j \geq 1$}.
  \end{cases}
\end{split}
\end{equation}
\end{proposition}
\begin{proof}
First note that $\rho_k(z) \asymp \rho_k(p)$ for all $z \in B(p, c\rho_k(p))$. Moreover, when $j=0$, the estimate is obvious, 
since $B(p, c\rho_k(p)) \subset Z^{\d{k}}_{\d{0}}(\eta)$ and this is an AC space. 

When $j \geq 1$, then $B(p, c\rho_k(p))\subset Z^{\d{k}}_{\d{j}}(\eta)$ and this is the total space of a bundle over
$C(S_{j})$ with fiber $Z^{\d{j-1}}$. As such, we can replace the ball $B(p,r)$ by a product
of balls $B(q_j,r) \subset C(S_{j})$ and $B(p_{j-1},r) \subset Z^{\d{j-1}}$.  The factor $r^{n-m_{j-1}}$ appears from
the first factor since $\dim C(S_{j}) = n - m_{j-1}$. For the second factor, we can use 
the inductive definition of the functions $w^{\d{k}}_i$ to write (as in Section~\ref{sec:nw}) 
\[
  w^{\d{k}}(z)^{b} 
 \asymp w^{\d{k}}_{1}(z)^{b_{1}}\ldots w^{\d{k}}_{j}(z)^{b_j} 
 \asymp \left(\frac{\rho_{j-1}(z_{j-1})}{\rho_k(z)}\right)^{|b(j)|} w^{\d{j-1}}(z_{j-1})^{b(j-1)}
\]
for all $z\in B(p,r)$. Here we use that  $w^{\d{k}}_i (z) \asymp 1$ for all $i \geq j +1$. 
This shows that the measure of this second factor is comparable to 
\[\rho_k(p)^{-|b(j)|}\sV_{j-1}(p_{j-1}, r; |b(j)|, b(j-1)).\] 
%
\end{proof}

To proceed, it would be necessary to subdivide further into the cases when the projected 
ball $B(p_{j-1},r) \subset Z^{(j-1)}$ is remote, anchored, or neither. Since this is not needed right away,
we defer this computation to Section~\ref{sec:rem-sharp}.   On the other hand, for the proof of the 
volume comparison estimate~\eqref{vc} in Proposition~\ref{lem:vc}, we do require an estimate 
of the volume of the maximal remote balls $B(p, c\rho_k(p))$; this is carried out in the following subsection.

\subsubsection{Estimates on the volume of non-remote balls}\label{sec:nrb}
Using the estimates above for the volume of anchored and remote balls, we are now able to give sharp 
estimates for the measure of all non-remote balls. 

\begin{proposition}\label{corintballest}
Assume that $a\in \RR$ and $b \in \RR^k$ satisfy condition~\eqref{eq:cVC}. 
 Let $c \in (0,1)$ be a remote parameter.  If $r \geq c \rho_k(p)$, then
\begin{equation}\label{eq:nrb}
\sV_k(p,r; a, b) \asymp r^{a+n}
\end{equation}
for all $p \in Z^{\d{k}}$.
\end{proposition}

\begin{proof}
We prove this by induction.  The estimate in the case $k=0$ is obvious by rescaling. 

Now assume  that \eqref{eq:nrb} holds for all QAC spaces of depth strictly less than $k$. 

Let $(Y^{\d{k}}, h^{\d{k}}) \in \sD_k$ be the cross-sections of $(Z^{\d{k}}, g^{\d{k}})$. From the discussion in 
Section~\ref{ssec:Dk} these spaces have diameter uniformly bounded above, say by some constant $C>2$.
For any $p \in Z^{\d{k}}$, if $r> C \rho_k(p)$,   then 
\[
\{x : \rho_k(x) \leq (C - 1) r \} \subset B(p,r) \subset \{x: \rho_k(x) \leq (C + 1) r\},
\] 
and hence, for such values of $r$, $\sV_k(p,r;a,b) \asymp \sA_k(r;a,b) \asymp r^{a+n}$ by~\eqref{volanch-2}.
On the other hand, for radii in the intermediate range $c \rho_k(p) \leq r \leq C\rho_k(p)$, we have
 \begin{equation}
  \sR_k(p, c\rho_k(p); a, b) \leq \sV_k(p,r;a,b) \leq  \sA_k(C \rho_k(p); a, b).
 \label{intranger}
 \end{equation}
  From~\eqref{eq:c-rem} and the induction hypothesis, the left side
  behaves like $\rho_k(p)^{a+n}$, since $\rho_{j-1}(p_{j-1}) \leq
  \rho_k(p)$ for all $1\leq j \leq k$.  From~\eqref{volanch-2} the right
  side is also of order  $\rho_k(p)^{a+n}$. Therefore~\eqref{eq:nrb}
  is true for all $r \geq c\rho_k(p)$. 
\end{proof}

\subsection{Proof of the volume comparison estimate~\eqref{vc} in Proposition~\ref{lem:vc}}
Using the volume estimates above, we can now easily derive the volume comparison results needed in 
the proof of Theorem~\ref{thm:main4}.

Suppose that the volume doubling estimate has been established for all spaces of depth less than $k$.  

Let $B(p,r)$ be any remote ball in $Z$. It is contained the maximal remote ball $B(p,c\rho_k(p))$.
By Lemma~\ref{lem:rb}, $B(p,c\rho_k(p))$ is contained in some $Z^{\d{k}}_{\d{j}}(\eta)$. If $j=0$, then the 
estimate in Propostion~\ref{lem:rem} shows that $(VD)_{\mu}$ holds. If $j \geq 1$, then 
\[
\sR_k (p, r; a, b) \asymp  \rho_k(p)^{a-|b(j)|} r^{n-m_{j-1}} \sV_{j-1}(p_{j-1}, r;|b(j)|, b(j-1)),
\]
and using the inductive hypothesis for the second factor we see that $(VD)_{\mu}$ is true for spaces
of depth $k$ as well.

Finally, consider any $p \in Z^{\d{k}}$.  Apply \eqref{volanch-2} with $R = \rho_k(p)$ and \eqref{eq:nrb} with 
$r = c\rho_k(p)$ to obtain \eqref{vc}.

\medskip
 
\subsection{Proof of $(PI)_{\mu,\delta}$ on remote balls}
We now turn to the proof that $\mbox{(PI)}_{\mu,\delta}$ holds on every remote ball in a space of depth $k$.
Various forms of the Poincar\'e inequality in the literature are appropriate for our purposes. 
We quote the following, which follows from \cite[Corollary 1]{Chen1997}.

\begin{proposition}[\cite{Chen1997}, Corollary 1]
Let $(M^n,g)$ be a compact Riemannian manifold, possibly with
boundary, with $\mbox{Ric} \geq -(n-1)a g$ for some $a>0$. Let $B(p,r)$ be a geodesic ball disjoint from $\del M$. For any function $f$, denote by
$\bar{f}$ its average over this ball. Then there exists a constant $C>0$ depending on $M$ and $g$
such that for any $f \in H^1(B(p,r))$,
\begin{equation}\label{eq:poincare}
\int_{B(p,r)} |f - \bar{f}|^2\, dV_g  \leq C \,r^{2} \int_{B(p,r)}|\nabla f|^2\, dV_g.
\end{equation}
\label{ChenLi}
\end{proposition}

We apply this as follows. 
\begin{proposition}\label{lem:pi_rem}
Fix the weighted QAC space $(Z,g, d\mu_{a,b}) \in \sQ_k$ and the remote parameter $c$. Suppose that we have verified 
$(PI)_{\mu,\delta}$ on all component spaces $Z^{(j)}$ with some parameter $\delta \in (0,1]$ in \eqref{eq:poincare}. 
Then there exists $C_P >0$ and a new $\delta_1 \in (0,1]$ so that for all $p\in Z$ and all $r\leq c \rho_k(p)$ we have
\begin{equation}\label{eq:pi_rem}
\int_{B(p,r)} |f- \bar{f}|^2 \, d\mu_{a,b} \leq C_P \, r^{2} \int_{B(p, \delta_1^{-1}r)} |\nabla f|^2 \, d\mu_{a,b}, 
\end{equation}
for every $f\in H^1_{\mathrm{loc}}(Z)$. 
\end{proposition}
\begin{proof}
We apply Proposition~\ref{ChenLi} on the various types of remote balls in $Z$, using the 
corresponding Poincar\'e inequality and/or a rescaling argument on each of these as needed. 
To simplify notation, assume that $\bar{f} = 0$.

We prove this statement by induction. If $k=0$ then we are on an AC space and~\eqref{eq:pi_rem} is true by rescaling and
\eqref{eq:poincare}.  Assume then that the statement is true for all QAC spaces of depth strictly less than $k$.

Let $p \in Z^{\d{k}}$. From Lemma~\ref{lem:rb}, there exists some $j \leq k$ such that $B(p, r) \subset Z^{\d{k}}_{\d{j}}(\eta)$ for 
all $0 < r\leq c\rho_k(p)$. If $j=0$, then our remote ball lies in an AC space and the statement is true.

If $B(p,r) \subset Z^{\d{k}}_{\d{j}}(\eta)$ for some $j \geq 1$, then consider $Z^{\d{k}}_{\d{j}}(\eta)$ 
locally as a product over a neighbourhood in $C(S_j)$ with fiber $Z^{\d{j-1}}$, with the product metric, 
and $B(p,r)$ as lying in the product $B(q_{j},r) \times B(p_{j-1}, r) \subset C(S_j) \times Z^{(j-1)}$. Extend $f$ by zero to the
rest of $B(q_j,r)\times B(p_{j-1},r)$; then the average of $f$ over this larger set is still zero. 
Note that $\rho_k$ is approximately constant over this set, so that if
we consider the measures  $dV_{C(S_j)}$ on $B(q_j,r)\subset C(S_j)$ and $d \mu_{|b(j)|, b(j-1)}$ on $B(p_{j-1}, r)\subset Z^{\d{j-1}}$, then  
\begin{multline*} 
\int_{B(p,r)} |f|^2 \, d\mu_{a,b} \preceq  \\
\rho_k(p)^{a-|b(j)|)}   \int_{B(q_j,r)} \int_{B(p_{j-1},r)} |f(q,z)|^2\, dV_{C(S_j)}(q) d\mu_{|b(j)|, b(j-1)}(z).
\end{multline*}
By induction, $\mbox{(PI)}_{\mu}$ holds on each of the two factors. Now write $\bar{h}_1$ and $\bar{h}_2$ for 
the average of any function $h$ over the balls $B(q_j,r)$ and  $B(p_{j-1},r)$, respectively.  The iterated partial averages of 
$f(y,z)$ satisfy $\overline{ (\bar{f}_2) }_1 = 0$, so integrating 
\[
|f(q,z)|^2 \leq  | f(q,z) - \bar{f}_2(q)|^2 + |\bar{f}_2(q)|^2.
\]
and using the Poincar\'e inequality on each of the two terms gives
\begin{multline*}
\int_{B(q_j,r)} \left( \int_{B(p_{j-1},r)} |f(q,z) - \bar{f}_2(q)|^2\,
  d\mu_{|b(j)|, b(j-1)}(z)\right) dV_{C(S_j)}(q) \\
+ \int_{B(p_{j-1},r)} \left( \int_{B(q_j,r)} |\bar{f}_2(q)|^2 \, dV_{C(S_j)}(q)\right) d\mu_{|b(j)|, b(j-1)}(z) \\
\leq C r^{2} \int_{B(q_j,\delta^{-1}r)} \int_{B(p_{j-1},\delta^{-1}r)} |\nabla_{q,z} f(q,z)|^2 \, d\mu_{|b(j)|, b(j-1)}(z) \, dV_{C(S_j)}(q).
\end{multline*}
\end{proof}

\section{Estimates on the Green function}\label{sec:Green}
Through the work in the last section, we have now verified all hypotheses of Theorem~\ref{thm:Hmuest}. 
Hence for the measure  $d\mu_{a,b} = \rho^a w^b  dV_{g} = (\rho_k)^a
(w^{\d{k}})^b dV_{g^{\d{k}}}$, where $a \in \RR$ and $b \in \RR^k$
satisfies~\eqref{eq:cVC}, i.e.\
\begin{equation*}
a + n \geq 0 \quad \text{and} \quad  |b(j)| + m_{j-1} \geq  0 \quad
\text{for all $j = 1, \ldots, k$}
\end{equation*}
Theorem~\ref{thm:heatL} and Theorem~\ref{thm:Gmuest} give
\begin{equation}
\begin{split}
|G_{\calL}(& z, z')| \preceq G_{\Delta+V}(z, z') \asymp \\ 
& (\rho_k(z) \rho_k(z') )^a (w^{\d{k}}(z) w^{\d{k}}(z'))^b \int_{d(z,z')}^\infty \frac{s}{\sqrt{\calV_k(z,s;a,b) \calV_k(z',s;a,b)}}\, ds.
\end{split}
\label{fundest5}
\end{equation}
This fundamental estimate is the basis for understanding the mapping properties of $\sL$. However, in order
to apply it, we must understand the behaviour of the volume function $\calV_k(z,r;a,b)$ better. Estimate~\eqref{eq:nrb} 
gives a good understanding of this for the nonremote balls, i.e\ for
the balls $B(z,r)$ with $r\geq c \rho_k(z)$. We now need something comparable 
for remote balls, i.e.\ the balls $B(z,r)$ with $0< r \leq c\rho_k(z)$, since the estimate for their volume in~\eqref{eq:c-rem} is not 
enough for our purposes.  According to the notation introduced in
Definition~\ref{defn:ra}, for this type of balls $\sV_k(z, r;a,b) =\sR_k(z, r;a,b)$.

\subsection{A sharp estimates on the volume of remote   balls}\label{sec:rem-sharp}
We need to estimate the measure of remote balls $\sR_k(p,r;a,b)$ in a way that does not
directly refer to the measures of balls in lower depth, as is the case
of the estimate~\eqref{eq:c-rem}. To do this, we introduce the concept of  a 
``remote chain'' at a point, designed to keep track of how complicated
the geometry of the QAC soace is near this point. 

Roughly speaking, the iterated structure of the QAC space makes every
point lie either in AC piece $Z^{\d{k}}_{\d{0}}$ or in a product of cones times
the AC piece of a lower depth QAC space.
By thickening the QAC pieces, we can choose these invariants
associated to a point so that all its remote balls also lie in this type of spaces. 

Concretely, we fix a remote parameter $c\in (0,1)$. To each point $z \in
Z^{\d{k}}$ we associate a {\em length} $s$ with $0\leq s < k$ and {\em remote chain of
length $s $}.  
When $s = 0$, the chain is empty and corresponds to the fact that
$z\in B(z, c\rho_k(z)) \subset Z^{\d{k}}_{\d{0}}(2c)$. When $s \geq 1$, it is a sequence of indices $k \geq j_1 > \ldots > j_s > 0$ 
and points $z_{j_\ell-1} \in Z^{\d{j_\ell-1}}$, $\ell=1, \ldots, s$, described as follows: Choose $j_1$ so that 
\[
  w^{\d{k}}_k(z) > 1 - 2c, \ldots, w^{\d{k}}_{j_1 +1} (z) > 1-2c,
  \quad \text{and} \quad w^{\d{k}}_{j_1}(z) < 1- 2c.
\]
By Lemma~\ref{lem:rb}, the maximal remote ball $B(z, c\rho_k(z))$ lies
in $Z^{\d{k}}_{\d{j_1}}(\eta)$, for $\eta>0$ depending on $c$ and $p$.
This is the total space of a bundle over $C(S_{j_1})$ with fiber $Z^{\d{j_1-1}}$. Thus $z$ is identified
with the pair $(q_{j_1}, z_{j_1-1}) \in C(S_{j_1}) \times Z^{\d{j_1-1}}$. (We are abusing notation in
the usual way by regarding this fibration as a product, which is legitimate since this ball lies
in a trivialized neighbourhood.)  We may then continue this process to choose the remaining
$j_i$, $i \geq 2$, or equivalently, regard the remote chain $j_2 > \ldots > j_s >1$ and points 
$z_{j_1-1} \in Z^{\d{j_1-1}}$ as being already determined by induction. 

From this definition, there is a chain of inequalities
\[
  w^{\d{k}}_{j_s}(z) < \ldots < w^{\d{k}}_{j_1}(z) < 1-2c.
\]
Moreover, by \eqref{eq:w-Zkj-t}, 
\begin{align*}
 w^{\d{k}}_{j_1+1}(z), \ldots, w^{\d{k}}_{k}(z) & \asymp 1, &\\
 w^{\d{k}}_{j_{\ell+1}+1}(z), \ldots, w^{\d{k}}_{j_{\ell}}(z) & \asymp
 \frac{\rho_{j_\ell-1}(z_{j_\ell-1})}{\rho_k(z)} & \text{for $1\leq \ell \leq s-1$},\\
 w^{\d{k}}_{1}(z), \ldots, w^{\d{k}}_{j_s}(z) & \asymp \frac{\rho_{j_s-1}(z_{j_s-1})}{\rho_k(z)}.
\end{align*}

We can now proceed to estimate the volume of remote balls. 
\begin{proposition}\label{prop:vrb}
Let $z \in Z^{\d{k}}$. 

(i) If the remote chain associated to $p$ has length $s=0$, then
\[
    \sR_k(z,r;a,b) \asymp \rho_k(z)^a r^n.
\]
 
(ii) If the remote chain associated to $p$  has length $s \geq 1$,
then, letting $j_\ell$ be the indices of the chain,
\[
   \sR_k(z,r;a, b) \asymp  \rho_k(z)^{a-|b(j_\ell)|} w^{\d{k}}(z)^{b- b(j_{\ell})}r^{n+|b(j_{\ell})|}
\]
provided 
\[
c\, w^{\d{k}}_{j_\ell}(z) \rho_k(z) \leq r \leq c\,   w^{\d{k}}_{j_{\ell-1}}(z)\rho_k(z).
 \]
This holds even when $\ell = 1$ or $s$ using the convention that $w^{\d{k}}_{j_{0}}(z) =1$,
$w^{\d{k}}_{j_{s+1}}(z) =0$, and $b(j_{s+1})=0$.
\end{proposition}
\begin{proof}
We have set things up so that this can be proved inductively on the length $s$.  
When $s=0$, $B(z, r) \subset  Z^{\d{k}}_{\d{0}}(\eta)$ for some $\eta \in (0,1)$ and every $0< r< c\rho_k(z)$.
Proposition~\ref{lem:rem} then gives the desired estimate.

If $s\geq 1$, then by Proposition~\ref{lem:rem}
\[
\sR_k(z,r; a, b) \asymp \rho_k(z)^{a-|b(j_1)|} r^{n-m_{j_1-1}} \sV_{j_1-1}(z_{j_1-1}, r;|b(j_1)|, b(j_1-1)) ,
\]
If $r > c \rho_{j_1-1}(z_{j_1-1}) = c\, w^{\d{k}}_{j_1}(z) \rho_k(z)$, then $B(z_{j_1-1}, r)$ is not remote in $Z^{\d{j_1-1}}$, so
the estimate follows from Proposition~\ref{corintballest}. On the other hand, when $r< c \rho_{j_1-1}(z_{j_1-1}) = c\,
 w^{\d{k}}_{j_1}(z) \rho_k(z)$, then $B(z_{j_1-1}, r)$ is remote in $Z^{\d{j_1-1}}$ and the estimate follows by the induction hypothesis.
\end{proof}

\subsection{An upper bound on the Green function}
We now use this information about the volumes of balls to bound the Green function. Returning to the expression 
\eqref{fundest5}, our immediate goal is to bound the integral
\[
\int_{d(z,z')}^{+\infty} \frac{s}{\sqrt{\sV_k(z,s;a,b) \sV_k(z',s;a,b)}} \, ds.
\]
The first step is to observe that since $s \geq d(z,z')$, we can replace this with the slightly simpler integral
\[
\II(z,z') = \int_{d(z,z')}^{\infty} \frac{s}{\sV_k(z,s;a,b)} \, ds.
\]
To show that these are equivalent, observe that $B(z',s) \subset B(z, s+ d(z,z')) \subset B(z, 2s)$, so by
volume doubling, $\sV_k(z',s;a,b) \leq C_D \sV_k(z, s;a,b)$. Interchanging the roles of $z$ and $z'$, we see
that $\sV_k(z,s;a,b) \asymp \sV_k(z',s;a,b)$ for $s \geq d(z,z')$. 

The ball $B(z,s)$ is nonremote when $s > c \rho_k(z)$; as such Proposition~\ref{corintballest} gives 
$\sV_k(z,s;a,b) \asymp s^{a+n}$ when $s$ is large,  and hence this integral converges provided $a+n >2$. 
Assuming this, then $\II(z,z') \asymp d(z,z')^{2-a-n}$ for $d(z,z') > c \rho_k(z)$. On the other hand, 
when $d(z,z') < c \rho_k(z)$ then $z'$ lies in a remote ball around $z$, and the estimate on $\II(z,z')$ 
depends on where $z'$ lies with respect the remote chain associated to $z$. 

\begin{lemma}\label{lem:II-est}
Assume that 
\begin{equation}\label{eq:eG}
a + n > 2 \quad \text{and} \quad  |b(j)| + m_{j-1} \geq  0 \quad
\text{for all $j = 1, \ldots, k$.}
\end{equation}
Let $z, z' \in Z^{\d{k}}$. 

(i) If $d(z,z') > c\rho_k(z)$, then
\[
\II(z,z') \asymp d(z,z')^{2-a-n}.
\]

(ii) If $d(z,z') < c\rho_k(z)$ and $z$ has a remote chain of length $s=0$, then
\[
\II(z,z') \asymp \rho_k(z)^{-a} d(z,z')^{2-n}.
\]

(iii) If $d(z,z') < c\rho_k(z)$ and $z$ has a remote chain $k \geq j_1
> \ldots > j_s >0$ of length $s\geq 1$, then
\[
  \II(z,z') \asymp \rho_k(z)^{-a+|b(j_\ell)|} w^{\d{k}}(z)^{-b+b(j_{\ell})}d(z,z')^{2-n-|b(j_{\ell})|}
\]
when 
\[
   c\, w^{\d{k}}_{j_\ell}(z) \rho_k(z) \leq d(z,z') \leq c\,   w^{\d{k}}_{j_{\ell-1}}(z)\rho_k(z).
 \]
This holds for all $1\leq \ell \leq s+1$,  with the convention that $w^{\d{k}}_{j_{0}}(z) =1$,
$w^{\d{k}}_{j_{s+1}}(z) =0$, and $b(j_{s+1})=0$.
\end{lemma}
\begin{proof}
 Part (i) follows from Proposition~\ref{corintballest}, while (ii) and
 (iii) follow from Proposition~\ref{prop:vrb}.
\end{proof}

Assembling all these estimates, we arrive at the 
\begin{theorem}\label{thm:gfe}
For any $(Z, g) \in \sQ_k$, let $h(z) = \rho_k(z)^{a/2}
w^{\d{k}}(z)^{b/2}$ for some $a \in \RR$, $b \in \RR^k$,  and set $V = 
-\Delta h/h$. Consider $G_{\Delta + V}$, as given by \eqref{Laplace}
for the generalized scalar Laplacian $\Delta+V$. 
If \eqref{eq:eG} is satisfied, i.e.\ $a+n > 2$ and $|b(j)| +  m_{j-1} \geq 0$ for $j \leq k$, then: 

\medskip 

(i) If $d(z,z') > c\rho_k(z)$, 
\[
G_{\Delta+V}(z,z') \asymp \rho_k(z)^{\frac{a}{2}}  w^{\d{k}}(z)^{\frac{b}{2}} \rho_k(z')^{\frac{a}{2}}  w^{\d{k}}(z')^{\frac{b}{2}} d(z,z')^{2-a-n}.
\]

(ii)  If $d(z,z') < c\rho_k(z)$ and $z$ has a remote chain of length $s=0$, then
\[
G_{\Delta + V}(z,z') \asymp \rho_k(z)^{-\frac{a}{2}}  w^{\d{k}}(z)^{\frac{b}{2}} \rho_k(z')^{\frac{a}{2}}  w^{\d{k}}(z')^{\frac{b}{2}} d(z,z')^{2-n}.
\]

(iii) If $d(z,z') < c\rho_k(z)$ and $z$ has a remote chain $k \geq j_1> \ldots > j_s >0$ of length $s\geq 1$, then
\[
  G_{\Delta + V}(z,z') \asymp \rho_k(z)^{-\frac{a}{2}+|b(j_\ell)|}  w^{\d{k}}(z)^{-\frac{b}{2}+b(j_{\ell})} \rho_k(z')^{\frac{a}{2}}  w^{\d{k}}(z')^{\frac{b}{2}} 
d(z,z')^{2-n-|b(j_{\ell})|}
\]
when 
\[
c\, w^{\d{k}}_{j_\ell}(z) \rho_k(z) \leq d(z,z') \leq c\,   w^{\d{k}}_{j_{\ell-1}}(z)\rho_k(z).
 \]
This holds for all $1\leq \ell \leq s+1$,  with the convention that $w^{\d{k}}_{j_{0}}(z) =1$,
$w^{\d{k}}_{j_{s+1}}(z) =0$, and $b(j_{s+1})=0$.
\end{theorem}
\begin{remark}
 In particular, for the scalar Laplaciam, i.e. when $a=0$  and $b=0$,
 the above gives
 \begin{equation}\label{eq:scG}
   G_{\Delta}(z,z') \asymp d(z,z')^{2-n}.
 \end{equation}
\end{remark}

\section{The Schur test}
We now use the estimates on $G_{\Delta +V}(z,z')$ from Theorem \ref{thm:gfe} to determine the values of the weight parameters 
$\delta \in \RR$, $\tau \in \RR^k$ such that the corresponding
integral operator 
\begin{equation}
G_{\Delta+V}\colon  \rho^{\delta+\frac{n}{2}-2 }w^{\tau + \frac{\nu}{2} -  \utwo}L^2(Z, dV_g)  \to\rho^{\delta+\frac{n}{2}}w^{\tau+\frac{\nu}{2}}H^2(Z, dV_g)
\label{fredtaudelta}
\end{equation}
given by this kernel is a bounded map. Here $\nu$ is the $k$-tuple with entries 
 $\nu_1 = m_0$, $\nu_j= m_{j-1}-m_{j-2}$ for $j \geq 2$ depending
solely on the dimensions $m_j$. It is chosen so that $|\nu (j)| =
m_{j-1}$, and its appeareance will become clear during this section.
We continue to omit the superscripts $(k)$ to keep the notation lighter. 

The boundedness of \eqref{fredtaudelta} is equivalent to the boundedness of 
\begin{equation}\label{eq:K}
\sK\colon L^2(Z,dV_g) \to L^2(Z, dV_g),
\end{equation}
given by the kernel
\begin{equation}
\sK(z,z') = \rho(z)^{-\delta -\frac{n}{2}}w(z)^{- \tau-\frac{\nu}{2}
}\ G_{\Delta+V}(z,z')\ \rho(z')^{\delta-2 + \frac{n}{2}} w(z')^{ \tau +\frac{\nu}{2}-\utwo}
\end{equation}
We approach this using the classical Schur test \cite{Halmos1978}, which states that if $f_1$ and $f_2$ are two positive 
measurable functions on $Z$ such that 
\[
\left|\int_{Z} \sK(z, z') f_1(z) dV_g(z)\right| \preceq f_2(z')  \quad \text{and} \quad 
\left| \int_{Z} \sK(z, z') f_2(z') dV_g(z') \right| \preceq f_1(z),
\]
for all $z, z' \in Z$, then \eqref{eq:K} is bounded. We shall take $f_1 = f_2 = \rho^{-n/2} w^{-\nu/2}$, and
the main task ahead is to estimate integrals of the form
\[
\int_{Z} G_{\Delta+V}(z,z')\ \rho(z')^{\alpha} w(z')^{\beta} \, dV_g(z') 
\]
with $\alpha \in \RR$ and $\beta \in \RR^k$.

\begin{lemma}\label{lem:forSchur}
If $a\in \RR$ and $b \in \RR^k$ satisfy~\eqref{eq:eG},  i.e.\ 
\[
a+n > 2\ \mbox{and}\ |b(j)| + m_{j-1} \geq 0, \quad j \leq k,
\]
and if $\alpha \in \RR$, $\beta \in \RR^k$ are chosen so that 
\begin{align}
- & n - \frac{a}{2} \leq \alpha< -2+\frac{a}{2} \label{eq:alpha} \\
- & m_{j-1} - \frac{|b(j)|}{2} \leq |\beta(j)|, \qquad   j \leq k, \label{eq:beta-1}\\
\mbox{and} \qquad & \ \beta \leq \frac12 b - \utwo
\label{eq:beta-2} 
\end{align} 
then
\[
\int_{Z} G_{\Delta+V}(z,z')\ \rho(z')^{\alpha}  w(z')^{\beta} \, dV_g(z')  \preceq \rho(z)^{\alpha + 2} w(z)^{\beta+\utwo} 
\]
for all $z \in Z$. 
\end{lemma}
\begin{remark}
The conditions \eqref{eq:beta-1} and \eqref{eq:beta-2} together are slightly more restrictive than \eqref{eq:eG}. 
Indeed, \eqref{eq:beta-2} gives $|\beta(j)| \leq \frac12 |b(j)| - 2$, so using \eqref{eq:beta-1}, 
\begin{equation}
-m_{j-1} - \frac12 |b(j)| \leq \frac12 |b(j)| - 2 \Longrightarrow 2 \leq m_{j-1} + |b(j)|.
\label{eq:beta-3}
\end{equation}
This will be used in several places below. 
\end{remark}
\begin{proof}
Fix $z \in Z$, and decompose the region of integration into the three subdomains $Z\setminus B(o, 2\rho(z))$, 
$B(o, 2\rho(z))\setminus B(z, c\rho(z))$ and $B(z, c\rho(z))$, where the later is a maximal remote ball. 

On the first region, $d(z,z') \asymp \rho(z')$, so by Theorem~\ref{thm:gfe}(i), 
\begin{align*}
& \int _{Z\setminus B(o, 2\rho(z))} G_{\Delta+V}(z,z')\ \rho(z')^{\alpha} w(z')^{\beta} \, dV_g(z')\\
& \quad \preceq \rho(z)^{\frac{a}{2}}w(z)^{\frac{b}{2}} \int_{Z\setminus B(0, 2\rho(z))}  d(z,z')^{2-n-a}\rho(z')^{\frac{a}{2} + \alpha}  w(z')^{\frac{b}{2}+\beta} \, dV_g(z')\\
& \quad \preceq \rho(z)^{\frac{a}{2}}w(z)^{\frac{b}{2}} \int_{Z\setminus B(0, 2\rho(z))}\rho(z')^{2-n-\frac{a}{2} + \alpha}  w(z')^{\frac{b}{2}+\beta} \, dV_g(z') \\
& \quad = \rho(z)^{\frac{a}{2}}w(z)^{\frac{b}{2}} \int_{2\rho(z)}^\infty d\sA_k(\rho; 2-n-\frac{a}{2} + \alpha, \frac{b}{2} + \beta). 
\end{align*}
Using \eqref{volanch} and the inequalities
\[
2-\frac{a}{2} + \alpha < \min\left\{0,  m_{j-1} + \frac{|b(j)|}{2} + |\beta(j)|\right\}, 
\]
which follow from \eqref{eq:alpha} and~\eqref{eq:beta-1}, this is bounded by 
\[
\begin{split}
\rho(z)^{2+\alpha}w(z)^{\frac{b}{2}} \left(1 +\sum_{j=1}^k \rho(z)^{-m_{j-1}-\frac{|b(j)|}{2} -|\beta(j)|}\right) \\
\preceq \rho(z)^{2+\alpha}w(z)^{\frac{b}{2}} \leq \rho(z)^{\alpha + 2} w(z)^{\beta + \utwo};
\end{split}
\]
the first inequality here relies on \eqref{eq:beta-1} again, while the second one uses  \eqref{eq:beta-2} and the fact that each $w_i \leq 1$. 

On the second domain, $d(z,z') \asymp  \rho(z)$ (instead of  $\rho(z')$), so
\begin{align*}
& \int_{B(o,2\rho(z))\setminus B(z, c\rho(z))} G_{\Delta+V}(z,z')\ \rho(z')^{\alpha} w(z')^{\beta} \, dV_g(z')\\
& \quad \preceq \rho(z)^{2-n-\frac{a}{2}}w(z)^{\frac{b}{2}} \int_{B(o, 2\rho(z))}  \rho(z')^{\frac{a}{2} + \alpha}  w(z')^{\frac{b}{2}+\beta} \,dV_g(z').
\end{align*}
By~\eqref{volanch} again, this is bounded by 
\[
\rho(z)^{2-n-\frac{a}{2}}w(z)^{\frac{b}{2}}  + \rho(z)^{\alpha+2}w(z)^{\frac{b}{2}} \left(1 +  \sum_{j=1}^k \rho(z)^{-m_{j-1}-\frac{|b(j)|}{2} - |\beta(j)|}\right). 
\]
Using \eqref{eq:alpha} and~\eqref{eq:beta-1}, this  is bounded by $\rho(z)^{\alpha+ 2}w(z)^{\frac{b}{2}}$, and hence as before by 
$\rho(z)^{\alpha+ 2}w(z)^{\beta + \utwo}$. 

Observe that up until this point we have not seen the need for the full gain in the power of $w$ to $w^{\beta + \utwo}$. This only
appears in the last step, where we break up the integral over the maximal remote ball $B(z, c\rho(z))$ into further pieces determined
by the remote chain associated to $z$. 

If this remote chain has length $s=0$, then by Theorem~\ref{thm:gfe}(ii), 
\begin{align*}
& \int_{B(z, c\rho(z))} G_{\Delta+V}(z,z')\ \rho(z')^{\alpha} w(z')^{\beta} \, dV_g(z') \\
& \preceq \rho(z)^{-\frac{a}{2}}  w(z)^{\frac{b}{2}} \int_{B(z, c\rho(z))} d(z,z')^{2-n}\rho(z')^{\frac{a}{2}+\alpha} w(z')^{\frac{b}{2}+\beta} \, dV_g(z')\\
& \preceq \rho(z)^{-\frac{a}{2}}  w(z)^{\frac{b}{2}} \int_{0}^{c\rho(z)} r^{2-n}\ d\sR_k(z, r; \frac{a}{2}+\alpha,\frac{b}{2}+\beta).
\end{align*}
Since $r \mapsto \sR_k(z, r; \cdot, \cdot)$ is monotone, an integration by parts and Proposition~\ref{prop:vrb}(i) bound this by
\begin{equation*}\label{eq:3b-0}
\rho(z)^{2+\alpha} w(z)^{\frac{b}{2}} \leq \rho(z)^{2+\alpha} w(z)^{\beta+2}.
\end{equation*}

Finally, suppose that the remote chain $k \geq j_1 > \ldots > j_s >0$ for $z$ has length $s\geq 1$. Then by
Proposition~\ref{prop:vrb}(ii)
\begin{align*}
\int_{B(z, c\rho(z))} & G_{\Delta+V}(z,z')\  \rho(z')^{\alpha} w(z')^{\beta}\, dV_g(z')\\ 
& \preceq \sum_{\ell=1}^{s+1} \left| \rho(z)^{-\frac{a}{2} + |b(j_\ell)|} w(z)^{-\frac{b}{2}+b(j_\ell)} \right. \\
& \qquad \qquad  \left. \times\int_{B_{j_\ell-1} \setminus B_{j_\ell}} d(z,z')^{2-n-|b(j_\ell)|}\rho(z')^{\frac{a}{2} +\alpha} w(z')^{\frac{b}{2} + \beta} \, dV_g(z') \right|\\
& \preceq \sum_{\ell=1}^{s+1} 
             \left| \rho(z)^{-\frac{a}{2} + |b(j_\ell)|} w(z)^{-\frac{b}{2}+b(j_\ell)} \right. \\
& \qquad \qquad   \left. \times      \int_{c\,w_{j_{\ell}}(z)\rho(z)}^{c\,w_{j_{\ell-1}}(z)\rho(z)}
               r^{2-n - |b(j_\ell)|}\ d\sR_k(z, r; \frac{a}{2} + \alpha, \frac{b}{2}+\beta)\right|;
\end{align*}
here $B_{j_\ell-1} \setminus B_{j_\ell} = B(z, c\,w_{j_{\ell-1}}(z)\rho(z))\setminus B(z, c\,w_{j_{\ell}}(z)\rho(z))$.
By Proposition~\ref{prop:vrb}(ii), this is estimated by 
\begin{multline*}
\sum_{\ell=1}^{s+1} \rho(z)^{\alpha+  \frac{|b(j_\ell)|}{2}-|\beta(j_\ell)|} w(z)^{\frac{b(j_\ell)}{2} + \beta - \beta(j_\ell)}
\left| \int_{c\,w_{j_{\ell}}(z)\rho(z)}^{c\,w_{j_{\ell-1}}(z)\rho(z)} r^{1+ |\beta(j_\ell)|-\frac{|b(j_\ell)|}{2}}dr \right| \\
\preceq \sum_{\ell=1}^{s+1} \rho(z)^{\alpha + 2} w(z)^{\frac{b(j_\ell)}{2} + \beta - \beta(j_\ell)} \left|w_{j_{\ell}}(z)^{2+|\beta(j_\ell)|-\frac{|b(j_\ell)|}{2}} \right|. 
\end{multline*}
In the last step we used \eqref{eq:beta-2} and that $w_{j_\ell}(z) <w_{j_{\ell-1}}(z) \leq 1$. Each summand is therefore bounded by
$\rho(z)^{\alpha+2}w(z)^{\beta+\utwo}$ times 
\[
\begin{split}
\left(\prod_{i=2}^{j_\ell} w_i(z)^{\frac12 b_i - \beta_i}\right) w_1(z)^{\frac12 b_1 - 2 - \beta_1} 
w_{j_\ell}(z)^{2 + |\beta(j_\ell)| - \frac12 |b(j_\ell)|} \\
\prod_{i=2}^{j_\ell} \left( \frac{w_i(z)}{w_{j_\ell}(z)}\right)^{\frac12 b_i - \beta_i} \left( \frac{w_1(z)}{w_{j_\ell}(z)}\right)^{\frac12 b_1 -
\beta_1 - 2}
\end{split}
\]
The proof is finished by observing that $w_i(z) \preceq w_{j_\ell}(z)$ when $i \leq j_\ell$ while each exponent is nonnegative,
so this displayed expression is bounded. 
\end{proof}

We can now proceed directly to the analysis of the mapping \eqref{fredtaudelta}. 
\begin{theorem}\label{thm:Fredholm}
With all notation as above, assume (as in \eqref{eq:beta-3}) that 
\[
a+n > 2\ \mbox{and}\ |b(j)| + m_{j-1} \geq 2.
\]
Then \eqref{fredtaudelta} is a bounded mapping provided 
\[
2-n -\frac{a}{2} < \delta < \frac{a}{2} \quad \text{and} \quad 
\utwo-\nu-\frac{b}{2} \leq \tau \leq \frac{b}{2}.
\] 
\end{theorem}
\begin{proof}
As explained earlier, we apply the Schur test to the mapping associated to the kernel $\sK$ with the functions 
$f_1 = f_2 =  \rho^{-\frac{n}{2}}w^{-\frac{\nu}{2}}$. 
 
First, by Lemma~\ref{lem:forSchur}, 
\begin{align*}
& \qquad  \int_Z \sK(z,z') \rho(z')^{-\frac{n}{2}} w(z')^{-\frac{\nu}{2}} \, dV_g(z') \\
& = \rho(z)^{-\delta -\frac{n}{2}}w(z)^{ -\tau-\frac{\nu}{2}} \int_{Z} G_{\Delta+V}(z,z')\ \rho(z')^{\delta-2} w(z')^{\tau- \utwo}\, dV_g(z')\\
& \leq  \rho(z)^{-\frac{n}{2}} w(z)^{-\frac{\nu}{2}},
 \end{align*}
so long as $\delta-2$ satisfies inequality \eqref{eq:alpha},
\begin{equation}\label{eq:delta-1}
 -n-\frac{a}{2} \leq \delta-2 < -2 + \frac{a}{2},
\end{equation}
 and $\tau  - \utwo$ satisfies 
\eqref{eq:beta-1},  \eqref{eq:beta-2}. These inequalities state that 
\begin{equation}\label{eq:B-z}
-m_{j-1}- \frac{|b(j)|}{2} \leq |\tau (j)|-2,  \qquad \tau - \utwo  \leq \frac{b}{2} -\utwo 
 \end{equation}
for $j \leq k$.  

On the other hand,
\begin{align*}
& \qquad  \int_Z \sK(z,z') \rho(z)^{-\frac{n}{2}} w(z)^{-\frac{\nu}{2}}\, dV_g(z)\\
& = \rho(z')^{\delta-2 + \frac{n}{2}} w(z')^{\nu + \tau-\utwo} \int_Z
G_{\Delta+V}(z,z')\ \rho(z)^{-\delta -n}w(z)^{- \tau -\nu} \, dV_g(z)\\
& \leq \rho(z')^{-\frac{n}{2}} w(z')^{-\frac{\nu}{2}}
 \end{align*}
provided $-\delta-n$ satisfies
\begin{equation}\label{eq:delta-2}
 -n-\frac{a}{2} \leq -\delta -n < -2 + \frac{a}{2},
\end{equation}
while now $\tau$ must satisfy 
\begin{equation}\label{eq:B-z'}
-m_{j-1}- \frac{|b(j)|}{2} \leq -|\tau (j)|- |\nu(j)|,  \qquad -\tau - \nu  \leq \frac{b}{2} -\utwo 
\end{equation}
for $j \leq k$. 

We see that~\eqref{eq:delta-1} and~\eqref{eq:delta-2} combine to give
the inequality for $\delta$ in the statement of the theorem, while the
one for $\tau$ is given by the right-hand sides of~\eqref{eq:B-z}
and~\eqref{eq:B-z'}. Note that since $|\nu(j)|=m_{j-1}$, the left-hand sides of these
inequalities are then automatically satisfied.
\end{proof}

\section{Fredholm theorems}\label{fredthms}
In this final section we state and prove the main results of this paper.  Up until this point, we have been discussing
QAC spaces and metrics on them which satisfy a collection of structural hypotheses. In particular, we have
been supposing that $(Z, g) \in \sQ_k$ has only one end, so that it satisfies the condition (RCA), that the metric
$g$ be as in Lemma \ref{lem:m}, and that the generalized Laplacian $\sL = \nabla^* \nabla + \sR$ acts on a bundle $E$ 
over $Z$ where $\sR$ is a self-adjoint endomorphism of $E$ such that $\sR - V \!\cdot \mathrm{Id} \geq 0$,
where $V = -\Delta h/h$, $h = \rho^a w^b$.  Under all these hypotheses, we can apply Theorem \ref{thm:Fredholm}.
This shows that if $\delta$ and $\tau$ satisfy the inequalities in that theorem, then
\begin{equation}
\sL\colon \rho^{\delta+\frac{n}{2}} w^{\tau +\frac{\nu}{2}} H^2(Z; E)
\longrightarrow \rho^{\delta+\frac{n}{2}-2} w^{\tau + \frac{\nu}{2} - \utwo} L^2(Z; E)
\label{mF}
\end{equation}
is an isomorphism. 
Indeed, under these conditions, the Green function $G_{\sL}$ is a
bounded inverse to \eqref{mF}.

In this section we generalize this in two ways. First we explain that the conditions on $\sR$, $a$ and $b$
need only be satisfied near infinity, and furthermore, $Z$ may have any finite number of ends, although under
these weaker conditions, \eqref{mF} is only Fredholm, and may have nontrivial kernel and cokernel.  
The higher regularity analog of \eqref{mF} also holds.  We also state and prove an analog of this Fredholm
result for $\sL$ acting on weighted H\"older spacese over $Z$.

\begin{theorem}
\label{Sobthm}
Let $(Z,g) \in \sQ_k$ and let $\sL = \nabla^* \nabla + \sR$ be a generalized Laplacian acting on sections of
a bundle $E$ over $Z$. Suppose that there is some compact set $K_Z \subset Z$ such that on each
end $\sE_\ell$ of $Z$, i.e.\ each component of $Z \setminus K_Z$, there are weight parameter sets $a = (a_\ell)$,
$b = (b_\ell)$ such that $V = - \Delta (\rho^a w^b)/ \rho^a w^b$
satisfies 
\begin{equation}\label{eq:potential}
 V \!\cdot \mathrm{Id} \leq \sR \quad \text{on $\sE_\ell$.}
\end{equation}
Suppose further that each set $a_\ell, b_\ell$ satisfies all the conditions listed in Theorem \ref{thm:Fredholm}. 
Then for all $s \in \RR$, 
\begin{equation}
\sL\colon  \rho^{\delta+\frac{n}{2}} w^{\tau +\frac{\nu}{2}} H^{s+2}(Z; E)
\longrightarrow \rho^{\delta+\frac{n}{2}-2} w^{\tau + \frac{\nu}{2} - \utwo} H^s(Z; E)
\label{mF2}
\end{equation}
is Fredholm for all $\delta = (\delta_\ell)$ and $\tau=(\tau_\ell)$ satisfying 
\[
2-n -\frac{a_{\ell}}{2} < \delta_{\ell} < \frac{a_{\ell}}{2} \quad \text{and} \quad \utwo-\nu-\frac{b_{\ell}}{2} \leq \tau_{\ell} \leq \frac{b_{\ell}}{2}.
\] 
Here the weighted Sobolev spaces have weight pairs $\delta_\ell$, $\tau_\ell$ on the end $\sE_\ell$.
\end{theorem}
\begin{proof} We first prove the statement for $s=0$.

We assume that each $\sE_\ell$ is a manifold with compact boundary. Consider the operator $\sL$ on 
$\sE_\ell$ with Dirichlet boundary conditions; this is the self-adjoint operator associated by the Friedrichs
extension to the semibounded quadratic form $||\nabla u||^2 + \langle \sR u, u \rangle$. 
All arguments in the earlier part of this paper go through unchanged, so that using the heat kernel
we can construct the exact inverse for each of these operators. These are represented by Green functions $G_\ell(z, z')$.

Without loss of generality, we can take $K_Z$ a little bigger so that
the intersection with each of the ends $\sE_l$ is non-empty.
Now choose a partition of unity $\{\chi_\ell\}$, with $\ell$ indexing the ends, and with $\ell = 0$ corresponding
to the set $K_Z$; thus  each $\chi_\ell$ is a smooth nonnegative function supported in $\sE_\ell$, and $\sum \chi_\ell = 1$.
Also choose smooth nonnegative functions $\tilde{\chi}_\ell$ with
slightly larger support in $\sE_l$ so that 
$\tilde{\chi}_\ell = 1$ on the support of $\chi_\ell$. Using this
data, define the convolution kernel
\begin{equation}
\tilde{G}(z,z') = \sum_{\ell} \tilde{\chi}_\ell(z) G_\ell(z,z') \chi_\ell(z').
\label{param}
\end{equation}
Clearly, for each $z, z'\in Z$ we have
\[
\begin{split}
\sL_z\tilde{G} (z,z') & = \sum_\ell  \left( \tilde{\chi}_\ell(z) \sL_z  G_\ell(z,z') \chi_\ell(z')  + [ \sL_z, \tilde{\chi}_\ell](z) G_\ell(z,z') \chi(z')\right) \\
& = \sum_\ell \tilde{\chi}_\ell(z) \delta(z-z') \chi_\ell(z') + \sum_\ell [ \sL_z, \tilde{\chi}_\ell](z) G_\ell(z,z') \chi(z') \\
& = \delta(z-z') - \tilde{R}_1(z,z').
\end{split}
\]
Since the supports of $\nabla \tilde{\chi}_\ell$ and $\chi_\ell$ are disjoint, $\tilde{R}_1(z,z')$ is a smooth section. Furthermore,
since the integral operator $G_\ell$ corresponding to $G_\ell(z,z')$ acts on $\rho^{a-\utwo} w^{b-\utwo}L^2(Z,E)$, the same
is true for the integral operators $\tilde{G}$ and $\tilde{R}_1$
associated to the kernels $\tilde{G}(z,z')$ and $\tilde{R}_1(z,z')$,
giving bounded operators
\[
  \tilde{G}, \tilde{R}_1\colon \rho^{\delta+\frac{n}{2}-2} w^{\tau +\frac{\nu}{2}-\utwo} L^{2}(Z; E) \longrightarrow \rho^{\delta+\frac{n}{2}} w^{\tau + \frac{\nu}{2}} H^2(Z; E).
\]
Moreover for any $f$ which is also $\sC^{\infty}(Z,E)$, $\tilde{R}_1
f$ is in $\calC^\infty_0(Z,E)$. Thus $\tilde{R}_1$ is a compact
operator and
\[
  \sL\circ \tilde{G}= \text{Id} - \tilde{R}_1.
\]
On the other hand, at each $z\in Z$
\[
\begin{split}
 (\tilde{G}\circ \sL)(u)(z) 
& = \int_Z \tilde{G}_1(z,z') (\sL u) (z')\ dV_{g}(z')\\
&= \sum_{\ell} \tilde{\chi}_\ell(z) \int_Z G_\ell(z,z')\chi_\ell(z')(\sL u)(z')\ dV_{g}(z')\\
& = u(z) - \sum_\ell\int_Z G_{\ell}(z,z') [\sL,\chi_\ell] u(z')\ dV_{g}(z')\\
& = u(z) - \int_Z \tilde{R}_2(z,z') u(z')\ dV_g(z').
\end{split}
\]
Thus
\[
  \tilde G \circ \sL = \text{Id} - \tilde R_2.
\]
As before, the integral operator $\tilde R_2$ corresponding to the
kernel $\tilde R_2(z,z')$ is a bounded operator
\[
  \tilde{R}_2\colon \rho^{\delta+\frac{n}{2}-2} w^{\tau +\frac{\nu}{2}-\utwo} L^{2}(Z; E) \longrightarrow \rho^{\delta+\frac{n}{2}} w^{\tau + \frac{\nu}{2}} H^2(Z; E).
\]
However, it no longer maps smooth sections into smooth compactly
supported sections, since the kernel $\tilde{R}_2(z,z')$ is
compactly supported only in the $z'$ variable. But its adjoint is
compact, for the same reason that $\tilde{R}_1$ is compact, so it too
must be compact.

We have now produced an approximate inverse for the mapping
\eqref{mF2} when $s=0$, i.e.\ an inverse up to compact errors,
which proves that in this case \eqref{mF2} is Fredholm. A standard argument (e.g.\ commuting with powers of $\sL$ to 
handle $s$ a positive even integer, then using interpolation and duality to handle all other $s \in \RR$) proves
that \eqref{mF2} is Fredholm for all $s$.
\end{proof}

Estimates for this parametrix lead quickly to the corresponding result on weighted H\"older spaces.
\begin{theorem}
\label{Holdthm}
With all notation exactly as in Theorem \ref{Sobthm}, the mapping
\begin{equation}
\sL\colon \rho^{\delta} w^{\tau} \calC^{s+2,\gamma}_g(Z, E) \longrightarrow \rho^{\delta-2} w^{\tau-\utwo} \calC^{s,\gamma}_g(Z,E)
\label{mF2a}
\end{equation}
is Fredholm for all $s$ nonnegative integers and $\gamma\in (0,1)$ provided $\delta$ and $\tau$ satisfy the same 
inequalities as in Theorem~\ref{thm:Fredholm}. 
\end{theorem}
\begin{proof} It is sufficient to prove the statement for $s=0$ since higher regularity follows from local Schauder estimates
and Lemma~\ref{lem:weight-const}. 

Defining $\tilde{G}$ as above, we must show that 
\begin{equation}
\tilde{G}\colon \rho^{\delta-2} w^{\tau - \utwo} \calC^{0,\gamma}_g (Z,E) \longrightarrow \rho^{\delta} w^{\tau} \calC^{2,\gamma}_g(Z, E)
\label{mF3}
\end{equation}
is bounded, and that the remainder terms $\tilde{R}_1$, $\tilde{R}_2$ are compact between the appropriate spaces. 

Thus fix $f \in \rho^{\delta-2} w^{\tau - \utwo} \calC^{0,\gamma}_g(Z,E)$ and set $u = \tilde{G} f$. The inequalities for $\delta-2$ 
and $\tau-\utwo$ match those for $\alpha$ and $\beta$ in Lemma~\ref{lem:forSchur}, which implies directly that 
$u \in \rho^\delta w^\tau L^{\infty}(Z,E)$, with 
\[
\|u\|_{\rho^{\delta}w^{\tau} \sC^{0}} \leq  C \|f\|_{\rho^{\delta-2}w^{\tau-\utwo}\sC^{0}}.
\]

Of course, $u \in \sC^{2,  \gamma}_{\text{loc}}(Z)$, but we must show that 
$|\nabla^{i} u| \preceq \rho^{\delta -i}w^{\tau -\underline{i}}$, $i=1,2$ and $[\nabla^2u]_{\gamma; B(z,c)}\preceq \rho^{\delta-2-  \gamma}
w^{\tau - 2 -\underline{\gamma}}$ on balls $B(z, \frac{1}{2}c)$, with constants uniform in $z$. The radius $c$ is the constant 
in Lemma~\ref{lem:weight-const}, and $[\nabla^2u]_{\gamma; B(z,\frac{1}{2}c)}$ is the unweighted H\"older seminorm 
on $B(z,\frac{1}{2}c)$.
For simplicity, assume that the operator is scalar, since the general case involves only a change of notation. 
Finally, we can assume that this constant $c$ is also a remote parameter, and $B(z, c\rho(z))$ is a remote ball.

Write $\tilde u = \rho^{-\delta} w^{-\tau} u$, $\tilde{f} = \rho^{-\delta} w^{-\tau} f$, so that 
$\tilde \sL \tilde u = \tilde f$, where $\tilde \sL = \sL - \rho^{-\delta} w^{-\tau}[\sL, \rho^{\delta}w^{\tau}]$.
Use the decompositions $Z = \bigsqcup_{j=0}^k Z^{\d{k}}_{\d{j}} = Z = \bigcup_{k=0}^k Z^{\d{k}}_{\d{j}}(\eta)$, where
the thickening parameter $\eta$ is chosen so that $B(z, c\rho(z)) \subset Z^{\d{k}}_{\d{j}}(\eta)$ when $z \in Z^{\d{k}}_{\d{j}}$.
We analyze each $Z^{\d{k}}_{\d{j}}$ separately. To keep track of the depth, we re-introduce sub/superscripts.

When $j=0$, then $w^{\d{k}}_{\ell} \asymp 1$ for all $\ell$, and $\rho_k(z') \asymp \rho_k(z)$ when $z' \in B(z,c\rho_k(z))$. 
Now rescale $\tilde u$ and $\tilde f$ by setting
\[
\bar u(r, y) = \tilde u( \alpha \, r, y), \quad \bar f (r,y) = \tilde f(\alpha \,r, y), 
\qquad \mbox{where}\quad \alpha = \frac{1}{2}c \rho_k(z), 
\]
which we regard as functions on the ball of radius $2$ around $\bar z = (\frac{2}{c}, y)$. 
The standard local Schauder estimate on $\{\rho_k(z) \leq \frac{2}{c} + c\}$ gives
\begin{equation}\label{eq:Schauder1}
\|\bar u\|_{\sC^{2,\alpha}(B(\bar z, 1))}  \leq C \left(\|\bar u\|_{\sC^{0} (B(\bar z, 2))} + \|\bar  f\|_{\sC^{0,\gamma}(B(\bar z, 2))} \right),
\end{equation}
where the constant $C$ is uniform in $z$.  Clearly $\|\bar u\|_{\sC^{0} (B(\bar z, 2))} \preceq \|u\|_{\sC_{\delta,\tau}^{0}(Z)} 
\leq \|f\|_{\sC^{0}_{\delta-2,\tau-\utwo}}$ and $\|\bar f\|_{\sC^{0,\gamma}(B(\bar z, 2))} \leq \|f\|_{\rho^{\delta}w^{\tau}\sC^{0,\gamma}} \leq
\|f\|_{\rho^{\delta-2}w^{\tau-\utwo}\sC^{0,\gamma}}$, so the terms on the right are uniformly bounded by $\|f\|_{\rho^{\delta-2}w^{\tau-\utwo}\sC^{0,\gamma}}$.
Since $\nabla_r v \asymp \rho_k \nabla_{\rho_k} \tilde u$, the estimates for $\nabla u$, $\nabla^2 u$ and $[\nabla^2 u]$ 
then hold on $B(z, \frac{1}{2}c\rho_k(z))$, and hence on the smaller ball $B(z, \frac{1}{2}c)$.

If $1\leq j \leq k$, then $Z^{\d{k}}_{\d{j}}$ is quasi-isometric to the product $C(S_j) \times Z^{\d{j-1}}$, so we assume that 
\[
B(z, c\rho_k(z)) \subset B(q_j, c\rho_k(z)) \times  B(z_{j-1}, c\rho_k(z)).
\]
Now define
\begin{align*}
 & \bar u (r, \sigma, t, \tau) = \tilde u (\alpha \, r,  \sigma, \alpha w^{\d{k}}_{\d{j-1}}(z)\, t, \tau), \\
 &   \bar f (r, \sigma, t, \tau) = \tilde f (\alpha\, r, \sigma, \alpha w^{\d{k}}_{j-1}(z)\, t, \tau), \qquad \alpha =  \frac{1}{2}c \rho_k(z),
\end{align*}
and consider these as living on a ball of radius $2$ in the region $\rho_{k}(z) \leq \frac{2}{c} + c$. Note that by \eqref{eq:w-Zkj-t}, 
$\rho_k(z)w^{\d{k}}_{j}(z)= \rho_{j-1}(z_{j-1})$ on $Z^{\d{k}}_{\d{j}}$.  

As before, apply the local Schauder estimate~\eqref{eq:Schauder1}.  As before, the right hand side of that estimate is 
uniformly bounded by $\|f\|_{\rho^{\delta-2}w^{\tau-\utwo}\sC^{0,\gamma}}$. Furthermore, $\nabla_r \asymp \rho_k\nabla_{\rho_k}$ and 
$\nabla_t \asymp \rho_{j-1} \nabla_{j-1}$, and in addition $\rho_{j-1} \leq \rho_k$, so we obtain uniform bounds for 
\[
  \sup_{z'\in B(z, \frac{1}{2}c\rho_k(z)} |\rho^{-\delta + i }  w^{-\tau} w_{j-1}^i\nabla^i u (z')|
\]
and
\begin{multline*}
\sup_{\substack{z', z''\in B(z, \alpha) \\z'\neq z''}}  \min(\rho_{j-1}(z_{j-1}'), \rho_{j-1}(z_{j-1}''))^{\gamma}\ \times  \\
  \frac{|\rho^{-\delta+2}w^{-\tau}w_{j-1}^2\nabla^2 u(z')-\rho^{-\delta+2}w^{-\tau}w_{j-1}^2\nabla^2 u(z'')|}{d(z',z'')^{\gamma}}.
\end{multline*}
Furthermore, on the smaller ball $B(z, \frac{1}{2}c) \subset B(z, \frac{1}{2}c \rho(z))$ the weight functions are comparable to
$\rho_k(z)$ and $w^{\d{k}}(z)$ and $w_1^{\d{k}} \preceq w_{j-1}^{\d{k}}$, which gives the desired estimate, i.e., that  
$u \in \rho^\delta w^\tau \calC^{k+2,\gamma}(Z,E)$ with an a priori bound.  A similar argument, as in the proof of 
Theorem~\ref{Sobthm}, shows that $\tilde{R}_1$, $\tilde{R}_2$ map into spaces of $\calC^\infty$ functions with 
appropriate decay, and are therefore compact. 

We have now shown that $\sL$ is invertible modulo compact operators, and hence is Fredholm between these
weighted H\"older spaces. 
\end{proof}

\section{Applications and relationship to previous work}
We have already explained in the introduction that the work of Joyce \cite{Joyce2000}, and our interest in finding
more flexible methods to generalize his linear results,  was a primary motivation for our work.  This section
is primarily devoted to a description of the precise relationship of our results to his.  We also describe
an important use of elliptic theory on QALE spaces carried out by Carron. Finally, we state a consequence of 
our work which is directly related to the result of Colding-Minicozzi and Li concerning
spaces of polynomially bounded harmonic functions.

\subsection{The work of Joyce}\label{ssec:Joyce}
To explain Joyce's work, we must first translate notation.  
 
In our setting, $(Z,g)$ is QAC of depth $k$; the cross-section $Y_0$ of its tangent cone at infinity 
$C(Y_0)$ is a smoothly stratified space of depth $k$ with iterated edge metric.  We use the radial 
functions $(\rho, w_k, \ldots w_)$, where $\rho$ is a smoothed radial function on the cone $C(Y_0)$.
The level sets $\{\rho = \mbox{const.}\}$ are resolution blowups $(Y, h_{1/\rho})$ of $(Y_0, h_0)$. 
The restrictions of the remaining radial functions $w_j$'s to each $(Y,h_{1/\rho})$ are 
smoothed versions of the distance functions $s_j$ to the strata of $(Y_0, h_0)$.

Joyce deals exclusively with QALE manifolds, where $Y_0 = S^{n-1}/G$ is a quotient
of the sphere by a finite subgroup $G$ of $\mathrm{SU}(n)$.  He uses only two radial
functions, $\rho$ and $\sigma$; his choice of $\rho$ is the same as ours, but his 
$\sigma$ is a smoothed distance function to the entire singular set in $C(Y_0)$. More precisely,
\[
\sigma = \min(\rho w_1, \ldots, \rho w_k), 
\]
but since $w_1 \preceq w_{2} \preceq \ldots \preceq w_k$, we have that $\sigma = \rho w_1$.
He then considers the weighted spaces $\rho^{\beta + \frac12 (n-m_0)} \sigma^{\gamma + \frac12 m_0}L^2$,
where $m_0$ is the dimension of the QAC space $Z^{(0)}$ corresponding to the lowest depth stratum
and $n = \dim Z$. 

To translate this to one of the spaces $\rho^{\delta + \frac{n}{2}} w^{\tau + \frac{\nu}{2}} L^2$ that we consider,
recall that $\tau \in \RR^k$ is a full $k$-tuple of weight parameters and $\nu$ is a dimensional shift vector,
\[
\nu_1 = m_0, \quad \text{and} \quad  \nu_j = m_{j-1}-m_{j-2} \quad \text{for $2\leq j \leq k$},
\]
where $m_j$ is the dimension of the QAC space $Z^{(j-1)}$ associated to the stratum $S_{j-1}$
(so $m_j$ is the codimension of $C(S_j)$ in $C(Y_0)$).   It is not hard to see that 
\[
\delta = \beta + \gamma, \quad \tau + \frac{\nu}{2} = (\gamma + \frac{m_0}{2}, 0 , \ldots, 0),
\]
and hence $\tau_1 = \gamma$ and $\tau_j = -\frac{\nu_j}{2}$, $j \geq 2$. 

\begin{theorem}{\cite[Theorem 9.5.7 and Corollary 9.5.2]{Joyce2000}}
Let $(Z,g)$ be a QALE space which is asymptotic to $\RR^n/G$ at infinity. Let $m_0$ denote
the codimension of the singular set of $\RR^n/G$. If 
\begin{equation}\label{eq:Joyce-range}
\delta < \tau_1, \quad  2-m_0 < \tau_1 < 0, \quad \frac{2-n}{2} - \lambda < \delta < \frac{2-n}{2} + \lambda,
\end{equation}
where $\lambda = \sqrt{ (\frac{n-2}{2})^2 +  \tau_1 (n-m_0)}$ (which is less than $(n-2)/2$ since $\tau_1< 0$), and 
$\tau_j = -\frac{\nu_j}{2}$ for $j \geq 2$, then 
\[
\Delta\colon \rho^{\delta + \frac{n}{2}} w^{\tau + \frac{\nu}{2}} H^2(Z) \longrightarrow \rho^{\delta + \frac{n}{2}-2} w^{\tau + \frac{\nu}{2}-\utwo} L^2(Z)
\]
is an isomorphism.
\end{theorem}
He goes on to make Conjecture 9.5.16, that this map is an isomorphism on the larger set of weight parameters
\begin{equation}\label{eq:Joyce-conj}
\delta < \tau_1, \quad 2-m_0 < \tau_1 < 0, \quad 2-n < \delta < 0. 
\end{equation}
This is implied by our Theorem~\ref{Sobthm}.
\begin{theorem}\label{thm:unweighted}
 Let $(Z,g)$ be a QAC space of depth $k$. Assume that $(Z,g_Z)$ has only one end. Let $\sL= \nabla^*\nabla + \calR$ 
be a generalized  Laplacian acting on the sections of a  bundle $E$ over $Z$, with $\sR\geq 0$.
If 
\[
2-n < \delta < 0  \quad \text{and} \quad  2- \nu \leq \tau \leq 0, 
\]
then
\[
\sL\colon\rho^{\delta + \frac{n}{2}} w^{\tau + \frac{\nu}{2}} H^2(Z, E;dV_g) \longrightarrow
\rho^{\delta + \frac{n}{2}-2} w^{\tau + \frac{\nu}{2}-\utwo} L^2(Z,E; dV_g) 
\]
is an isomorphism.
\end{theorem}

Joyce's method relies on the existence of a barrier, i.e., a smooth positive function $F$ satisfying 
\[
F \asymp \rho^{\delta} w_1^{\tau_1}, \qquad  \Delta F \geq \rho^{\delta -2} w_1^{\tau_1-2}.
\]
He constructs such functions when $\delta$ and $\tau$ lie in the range~\eqref{eq:Joyce-range}. 
One can construct solutions on a sequence of compact domains which exhaust $Z$ with a fixed
right hand side, and then use the maximum principle with this barrier function to control the sequence
of solutions. 

\subsection{The work of Carron} 
We now describe a paper by Carron~\cite{Carron2011} which use results about elliptic operators on 
QALE manifolds to reach an important geometric conclusion.  This is included as counterpoint
and motivation for this type of linear theory.

As in Joyce's work, suppose that $(Z,g)$ is a QALE manifold asymptotic to the quotient $\RR^n/G$. He 
employs the same two weight functions as Joyce:  $\rho$, which is the distance to a fixed point $o \in Z$, and
$\sigma$, a smoothed version of the distance function to the singular set.  

Carron considers the same equation $\Delta u = f$ in the special case where $g$ has nonnegative Ricci tensor. 
He then invokes an estimate for the Green function due to Li and Yau, giving the familiar bound
$G(z.z') \preceq d(z,z')^{2-n}$ to conclude that if $|f| \preceq \rho^{-\epsilon}\sigma^{-2}$ for some $\e >  0$, then 
$|u| \preceq \log(\rho+2) \rho^{-\e}$.  The point here is that since $\sigma$ contains a factor of $\rho$ (see
the description of Joyce's weight functions above), the overall decay rate is like $\rho^{-2-\epsilon}$, so this
fits in with (but is slightly weaker than) what we attain here -- but it is sufficient for his purposes.

He uses this as follows. There are two canonical QALE Ricci-flat K\"ahler metrics on the Hilbert scheme 
$\text{Hilb}_0^n(\CC^2)$ of $n$ points in $\CC^2$. The first is the hyperk\"ahler metric constructed by Nakajima 
by hyperk\"ahler reduction; the second is Joyce's Ricci-flat QALE metric.  Carron proves that these metrics coincide
using a now-standard method of Yau to show that two Ricci-flat metrics in the same K\"ahler class coincide. 
This relies on finding a potential function for a suitable $f$, i.e. a function $\phi$ satisfying
$\partial \overline{\partial} \phi = f$, and this can be obtained in a standard way once one has solved
the Poisson equation. 

\subsection{Spaces of solutions with a polynomial bound} 
We have proved in Theorem~\ref{thm:main4} that any QAC space $(Z,g) \in \sQ_k$ satisfies 
the two key properties (VD) and (PI).  We point out here that these two properties are all
that are needed in the arguments in~\cite[Theorem 1.2]{Colding1998} and~\cite[Theorem 1]{Li1997} 
to prove the following
\begin{corollary}
Let $(Z,g)$ be a QAC space and $E$ a Hermitian vector bundle over $Z$ of rank $m$ with nonnegative 
curvature. Then for all $d\geq 1$, the space $\sH_d(M,E)$ of harmonic sections of $E$ (with respect to
the connection Laplacian) growing no faster than distance to the power $d$ is finite dimensional: 
\[
\dim \sH_d(M,E) \leq C\, m\, d^{\log_2 C_D-\epsilon},
\]
where $C_D$ is the volume-doubling constant for $(Z,g)$, and $C$ and $\epsilon$ depend only on $C_D$.
\end{corollary}
In particular, the nullspace of the connection Laplacian
\[
\nabla^*\nabla\colon \rho^{\delta+\frac{n}{2}}w^{\tau+\frac{\nu}{2}}H^2(M;E) \to
\rho^{\delta+\frac{n}{2}-2}w^{\tau+\frac{\nu}{2}-\utwo}L^2(M;E) 
\]
is finite dimensional for any $\delta \in \RR$, $\tau \in \RR^k$.

Note in particular that this result makes no restrictions on $d$, so this goes well beyond the possible
weight parameters allowed in our theorem.  Of course, the theorems of Colding-Minicozzi and Li do
not prove that the connection Laplacian on $E$ is Fredholm on this weighted space. 

\clearpage

\begin{table}
\begin{center}
\resizebox{\textwidth}{!} {
\begin{tabular}{r|l}
  {{\bf Symbol}} & {{\bf Denotes}}\\
 \hline
  $f_1\preceq f_2$  & $f_1$ and $f_2$ functions\\  & and $\exists$ $C >0$ so that   $f_1(x) \leq C f_2(x)$ for all $x$. \\
  $f_1 \asymp f_2$ &   $f_1$ and $f_2$ functions\\ & and  $\exists$ $c, C
  >0$ so that  $c f_2(x) \leq f_1(x) \leq C f_2(x)$ for all $x$.\\
 \hline
  $g_1 \asymp g_2$ & $g_1$ and $g_2$ Riemannian metrics which
  are quasi-isometric.\\
 \hline
   $(Z,g)$ or $(Z^{\d{k}},g^{\d{k}})$ & a QAC space of depth $k$\\
   $\rho$ or $\rho_k$ & radial function on a QAC space of depth $k$\\
   $w$ or $w^{\d{k}}$ & $w = (w_k, \ldots, w_1)$ or $w^{\d{k}} =
   (w^{\d{k}}_k, \ldots, w^{\d{k}}_1)$  defining functions\\ & on a QAC
   space of depth $k$\\
\hline
  $Z^{\d{k}}_{\d{j}}$ & all the points $p \in Z^{\d{k}}$ s.t.\\ &
                    $w^{\d{k}}_k(p) = \ldots = w^{\d{k}}_{j+1}(p) =1$ and $w^{\d{k}}_j(p) < 1$;\\
   $Z^{\d{k}}_{\d{j}}(\eta)$ & all the points $p \in Z^{\d{k}}$ s.t.\\ &
                     $w^{\d{k}}_k(p), \ldots , w^{\d{k}}_{j+1}(p)
                    \geq  1-\eta$ and $w^{\d{k}}_j(p) < 1$;\\
\hline
   $b$ & $(b_{1}, \ldots, b_k)\in \RR^k$\\
   $\frac{b}{2}$ & $(\frac{b_1}{2},\ldots, \frac{b_k}{2})\in \RR^k$\\
   $w^{b}$ or $(w^{\d{k}})^b$ & $w_1^{b_1}\cdot\ldots\cdot w_k^{b_k}$ or $(w^{\d{k}}_1)^{b_1}\cdot\ldots\cdot (w^{\d{k}}_k)^{b_k}$\\
   $\nu$ & $(\nu_1, \ldots, \nu_k)$ with $\nu_1 = m_0$, and $\nu_j =   m_{j-1}-m_{j-2}$ for $2\leq j \leq k$\\
         & where $m_{j-1} = \dim Z^{\d{j-1}}$ and $Z^{\d{j-1}}$ the  fiber over $C(S_j)$ in $Z^{\d{k}}_{\d{j}}$\\
   $\underline{r}$ & $(r, 0,  \ldots, 0, 0)\in \RR^k$\\
 $b + \underline{r}$ & $(b_1+r,b_2 \ldots,  b_k)\in \RR^k$\\
 $b \leq \beta$ & $b_j \leq \beta_j$ for all $1\leq j \leq k$\\
   $b(\ell)$ & $(b_1, \ldots, b_\ell) \in \RR^{\ell} \subset \RR^k$
   for all $1\leq \ell \leq k$; $b(0) =0$\\
  $|b(\ell)|$ & $b_1+ \ldots +b_\ell$ for all $1\leq \ell \leq k$;
  $|b(0)| = 0$\\
 \hline
 $d \mu_{a,b}$ & the measure $\rho^a w^b dV_{g}$ on the QAC
 space $(Z,g)$ of depth $k$;\\ &   here $a \in \RR$ and $b \in \RR^k$.\\
 \hline
  $B(p,r)$ & geodesic ball in a QAC space centered at $p$ with radius $r$.\\
 anchored ball& a ball of the form $B(o, R)$ with $R>1$\\
  remote ball & a ball of the form $B(p,r)$ with $r \leq c\rho(p)$\\
  $\sA_k(R; a,b)$ &  $\mu_{a,b} (B(o,R))$\\
 $\sR_k(p,r;a,b)$ & $\mu_{a,b}(B(p,r))$ with $B(p,r)$ a remote ball\\
 $\sV_k(p,r;a,b)$ & $\mu_{a,b}(B(p,r))$ for any $B(p,r)$ 
\end{tabular}}
\label{table}
\caption{Frequently used notation}
\end{center}
\end{table}

\clearpage

\bibliography{ad-refs}

\end{document}